\documentclass[12pt,reqno]{amsart}
\usepackage{amsthm,amsfonts,amssymb,euscript}

\newcommand{\bea}{\begin{eqnarray}}
\newcommand{\eea}{\end{eqnarray}}
\def\beaa{\begin{eqnarray*}}
\def\eeaa{\end{eqnarray*}}
\def\ba{\begin{array}}
\def\ea{\end{array}}
\def\be#1{\begin{equation} \label{#1}}
\def \eeq{\end{equation}}

\def\be{{\beta}}

\def\eps{\epsilon}

\def\al{\alpha}

\def\R{{\mathbb{R}}}

\def\Z{{\mathbb{Z}}}

\def\T{{\mathbb{T}}}

\newtheorem{theorem}{Theorem}[section]
\newtheorem{lemma}[theorem]{Lemma}
\newtheorem{proposition}[theorem]{Proposition}

\newtheorem{definition}[theorem]{Definition}

\numberwithin{equation}{section}

\begin{document}

\title[Energy-critical defocusing NLS]{Global well-posedness of the energy-critical defocusing NLS on $\mathbb{R}\times\mathbb{T}^3$}

\author{Alexandru D. Ionescu}
\address{Princeton University}
\email{aionescu@math.princeton.edu}

\author{Benoit Pausader}
\address{Brown University}
\email{benoit.pausader@math.brown.edu}

\thanks{The first author was supported in part by a Packard Fellowship.}

\begin{abstract}
We prove global well-posedness in $H^1$ for the energy-critical defocusing initial-value problem
\begin{equation*}
(i\partial_t+\Delta_x)u=u|u|^2,\qquad u(0)=\phi,
\end{equation*}
in the semiperiodic setting $x\in\R\times\T^3$.
\end{abstract}
\maketitle
\tableofcontents

\section{Introduction}\label{Intro}

Let $\mathbb{T}:=\R/(2\pi\mathbb{Z})$. In this paper we consider the energy-critical defocusing equation
\begin{equation}\label{eq1}
(i\partial_t+\Delta_x)u=u|u|^2
\end{equation}
in the semiperiodic setting $x\in\mathbb{R}\times\mathbb{T}^3$. Suitable solutions on a time interval $I$ of \eqref{eq1} satisfy mass and energy conservation, in the sense that the functions
\begin{equation}\label{conserve}
M(u)(t):=\int_{\R\times\T^3}|u(t)|^2\,dx,\qquad E(u)(t):=\frac{1}{2}\int_{\R\times\T^3}|\nabla u(t)|^2\,dx+\frac{1}{4}\int_{\R\times\T^3}|u(t)|^4\,dx,
\end{equation}
are constant on the interval $I$. Our main theorem concerns global well-posedness in $H^1(\R\times\T^3)$ for the initial-value problem associated to the equation \eqref{eq1}.

\begin{theorem}\label{Main1} (Main theorem) If $\phi\in H^1(\R\times\T^3)$ then there exists a unique global solution $u\in X^1(\R)$ of the initial-value problem
\begin{equation}\label{eq1.1}
(i\partial_t+\Delta)u=u|u|^2,\qquad u(0)=\phi.
\end{equation}
In addition, the mapping $\phi\to u$ extends to a continuous mapping from $H^1(\R\times\T^3)$ to $X^1([-T,T])$ for any $T\in[0,\infty)$, and the quantities $M(u)$ and $E(u)$ defined in \eqref{conserve} are conserved along the flow.
\end{theorem}

The uniqueness spaces $X^1(I)\subseteq C(I:H^1(\R\times\T^3))$ in the theorem above are defined precisely in section \ref{prelim}. These spaces were already defined and used by Herr--Tataru--Tzvetkov \cite{HeTaTz} and \cite{HeTaTz2}.

The large-data theory of critical and subcritical semilinear Schr\"{o}dinger equations is well understood in Euclidean spaces, at least in the defocusing case, see for example \cite{Cazenave:book,Tao:book,B,G,CKSTTcrit,RV,V}. See also \cite{Gr1,ShSt1} for previous results about the wave equation. Our proof of Theorem \ref{Main1} uses, as a black box, the global well-posedness and scattering of nonlinear solutions with bounded energy in the Euclidean case, the main theorem in \cite{RV}.

On the other hand, the theory of Schr\"odinger equations on general compact manifolds and spaces with smaller volume is much less understood, even in the defocusing subcritical case, see for example \cite{Bo1,BuGeTz,BuGeTz2,GePi,He,HeTaTz,HeTaTz2,IoPa,TaTz}. With the exception of the recent theorem in \cite{IoPa}, the results in these papers do not address the case of large-data critical problems.

Theorems similar to Theorem \ref{Main1} can be proved by the same method for the energy-critical defocusing NLS on $\R^2\times\T^2$, $\R^3\times\T$, $\R\times\T^2$, $\R^2\times\T$. In some of these cases\footnote{These cases are easier at the level of small-data results, since the necessary Strichartz estimates can be derived by counting arguments, as in \cite{HeTaTz} and \cite{HeTaTz2}. In our case, the proof of the Strichartz estimates in Proposition \ref{Stric2} seems to require the circle method and some of the bounds in \cite{Bo2}. In our case even the small-data theory is new.} it is possible to also derive global asymptotic information about the solution, i.e. prove scattering. We refer to \cite{TzVi} for an example of such result, in the case of small data. In our case, scattering cannot hold for the energy-critical equation on $\mathbb{R}\times\mathbb{T}^3$. In fact scattering fails even for the cubic equation on $\mathbb{R}$ (which corresponds to solutions that do not depend on the periodic variables for the cubic NLS on $\R\times\T^3$), where one expects modified scattering instead, see \cite{Oz}.

On the other hand, the energy-critical problem appears more complicated in dimensions $d\geq 5$, at least if one considers solutions on $\R\times\T^{d-1}$, due to the lower power in the nonlinearity and the relative weakness of the available scale-invariant Strichartz estimates. The problem also appears more difficult in the purely periodic case $\T^d$, $d\geq 4$; in this case even the small data theory is not known. The purely periodic problem on $\T^3$ was recently solved by the authors, see \cite{IoPa}. In general, the difficulty of the critical NLS problem on $\R^m\times\T^{n}$ increases if the dimension $m+n$ is increased or if the number $m$ of copies of $\R$ is decreased.

In a previous work on the hyperbolic space $\mathbb{H}^3$ \cite{IoPaSt}, the authors and G. Staffilani could isolate concentration of energy at a point as the only obstruction to global existence and could use Euclidean results from \cite{CKSTTcrit} to prevent this phenomenon. Here we consider another potential difficulty caused by the geometry, namely the presence of trapped geodesics. The significance of our result is that, as long as there is a dispersive direction, global wellposedness still prevails, even in the case of large data in 4 dimensions.

We summarize below the main ingredients in the proof of Theorem \ref{Main1}:
\begin{enumerate}
\item Critical stability theory. In section \ref{localwp} we develop a suitable large-data local well-posedness and stability theory of the nonlinear Schr\"{o}dinger equation \eqref{eq1.1}. This relies on the new Strichartz estimates on $\R\times\T^3$ in Proposition \ref{Stric2} and the critical theory developed in \cite{HeTaTz} and \cite{HeTaTz2}.

\item Profile decompositions. In section \ref{profiles} we construct certain profiles adapted to our geometry, and show how to use these profiles to decompose a bounded sequence of functions $f_k\in H^1(\R\times\T^3)$. In other words, we prove the analogue of Keraani's theorem \cite{Ker} in $\R\times\T^3$. A similar construction was used recently by the authors in \cite{IoPaSt}.

\item Compactness argument. In section \ref{proofthm} we use induction on energy and mass and a compactness argument to prove Theorem \ref{Main1}.
\end{enumerate}
This is essentially the standard framework in which large-data critical dispersive problems have been analyzed in recent years, starting with the work of Kenig--Merle \cite{KeMe}.

There are, however, significant difficulties in our semiperiodic setting, mainly caused by the presence of a large set of trapped geodesics.\footnote{ This difficulty is specific to Schr\"{o}dinger evolutions, for which very high frequency solutions propagate at very high speeds. } On one hand, the family of scale-invariant Strichartz estimates is more limited in our setting than in Euclidean spaces. This leads to difficulties in the small-data case, which have been recently resolved by Herr--Tataru--Tzvetkov \cite{HeTaTz}, \cite{HeTaTz2} using refined function spaces. On the other hand, some high frequency solutions can spend a significant (and unrelated to their frequency) amount of time in a small region, thereby interacting for a long time with middle frequency solutions, as is manifested by looking at linear solutions of the form
\begin{equation}\label{solu}
u(x,t)=e^{it\Delta}[\chi(x_1)e^{ik_0x'}],\qquad k_0\gg 1.
\end{equation}
This complicates significantly the large-data theory, in particular the construction of approximate solutions of the nonlinear flow, since local smoothing type estimates clearly fail for such solutions.

These considerations motivate our choice of working in $\R\times\T^3$. In this semiperiodic setting, the set of trapped geodesics is included in a codimension $1$ subbundle of the tangent space. Because of this we can still recover some amount of dispersion in the nonperiodic direction, which is the main ingredient in some of the key estimates such as Lemma \ref{locsmo} and Lemma \ref{step2}. At the same time, ``bad'' solutions such as those in \eqref{solu} are sufficiently far from saturating the Sobolev inequality,\footnote{This is relevant because of the ``critical'' nature of our equation.} which ultimately allows us to discard the effect of their interactions.

The strategy used here differs from the one used by the authors for $\mathbb{T}^3$ in \cite{IoPa} in that it gives a more precise construction of the approximate solution and the control on the error relies on the fact that the a priori problematic interactions do not modify much the location of the spectrum.

\section{Preliminaries}\label{prelim}In this section we summarize our notations and collect several lemmas that are used in the rest of the paper.

We define the Fourier transform on $\mathbb{R}\times\mathbb{T}^3$ as follows
\begin{equation*}
\left(\mathcal{F}f\right)(\xi):=\int_{\mathbb{R}\times\mathbb{T}^3}f(x)e^{-ix\cdot\xi}dx,
\end{equation*}
where $\xi=(\xi_1,\xi_2,\xi_3,\xi_4)\in\mathbb{R}\times\mathbb{Z}^3$. We also note the Fourier inversion formula
\begin{equation*}
f(x)=c\sum_{(\xi_2,\xi_3,\xi_4)\in\mathbb{Z}^3}\int_{\xi_1\in\mathbb{R}}\left(\mathcal{F}f\right)(\xi)e^{ix\cdot\xi}d\xi_1.
\end{equation*}
We define the Schr\"odinger propagator $e^{it\Delta}$ by
\begin{equation*}
\left(\mathcal{F}e^{it\Delta}f\right)(\xi):=e^{-it|\xi|^2}\left(\mathcal{F}f\right)(\xi).
\end{equation*}

We now define the Littlewood-Paley projections. We fix $\eta^{1}:\mathbb{R}\to[0,1]$ a smooth even function with $\eta^1(y)=1$ if $|y|\leq 1$ and $\eta^1(y)=0$ if $|y|\geq 2$. Let $\eta^4:\mathbb{R}^4\to[0,1]$, $\eta^4(\xi):=\eta^1(\xi_1)^2\eta^1(\xi_2)^2\eta^1(\xi_3)^2\eta^1(\xi_4)^2$. We define the Littlewood-Paley projectors $P_{\le N}$ and $P_N$ for $N=2^j\ge 1$ a dyadic integer by
\begin{equation*}
\begin{split}
\mathcal{F}\left(P_{\le N}f\right)\left(\xi\right)&:=\eta^4(\xi/N)\left(\mathcal{F}f\right)(\xi),\qquad\xi\in \R\times\Z^3,\\
P_1f&:=P_{\le 1}f,\\
P_Nf&:=P_{\le N}f-P_{\le N/2}f\quad\hbox{if}\quad N\ge 2.
\end{split}
\end{equation*}
For any $a\in (0,\infty)$ we define
\begin{equation*}
P_{\leq a}:=\sum_{N\in 2^{\Z_+},\,N\leq a}P_N,\qquad P_{> a}:=\sum_{N\in 2^{\Z_+},\,N>a}P_N.
\end{equation*}
We will also consider frequency projections on cubes. For $C=[-\frac{1}{2},\frac{1}{2})^4$ and $z\in\mathbb{Z}^4$, we define the (sharp) projection on $C_z=z+C$, $P_{C_z}$ by
\begin{equation*}
\left(\mathcal{F}P_{C_z}f\right)(\xi):=\mathbf{1}_{C_z}(\xi)\left(\mathcal{F}f\right)(\xi)
\end{equation*}
for $\mathbf{1}_{C_z}$ the characteristic function of $C_z$.

{\bf{Function spaces.}} The strong spaces are the same as the one used by Herr-Tataru-Tzvetkov \cite{HeTaTz,HeTaTz2}. Namely
\begin{equation*}
\begin{split}
\Vert u\Vert_{X^s(\mathbb{R})}&:=\left(\sum_{z\in\mathbb{Z}^4}\langle z\rangle^{2s}\Vert P_{C_z}u\Vert_{U^2_\Delta(\mathbb{R},L^2)}^2\right)^\frac{1}{2},\\
\Vert u\Vert_{Y^s(\mathbb{R})}&:=\left(\sum_{z\in\mathbb{Z}^4}\langle z\rangle^{2s}\Vert P_{C_z}u\Vert_{V^2_\Delta(\mathbb{R},L^2)}^2\right)^\frac{1}{2},
\end{split}
\end{equation*}
where we refer to \cite{HeTaTz,HeTaTz2} for a description of the spaces $U^p_\Delta$ and $V^p_\Delta$ and of their properties. Note in particular that
\begin{equation*}
X^1(\R)\hookrightarrow Y^1(\R)\hookrightarrow L^\infty(\mathbb{R},H^1).
\end{equation*}

For intervals $I\subset\mathbb{R}$, we define $X^1(I)$ in the usual way as restriction norms, thus
\begin{equation*}
X^1(I):=\{u\in C(I:H^1):\Vert u\Vert_{X^1(I)}:=\sup_{J\subseteq I,\,|J|\leq 1}[\inf\{\Vert v\Vert_{X^1(\mathbb{R})}:v_{\vert J}=u\}]<\infty\}.
\end{equation*}
The norm controling the inhomogeneous term on an interval $I=(a,b)$ is then defined as
\begin{equation}\label{NNorm}
\Vert h\Vert_{N(I)}:=\Big \|\int_a^te^{i(t-s)\Delta}h(s)ds\Big\|_{X^1(I)}.
\end{equation}

We also need a weaker norm
\begin{equation*}
\begin{split}
\Vert u\Vert_{Z(I)}=\sup_{J\subseteq I,\vert J\vert\le 1}\Vert(\sum_NN^2\Vert P_Nu(t)\Vert_{L^4(\mathbb{R}\times\mathbb{T}^3)}^4)^\frac{1}{4}\Vert_{L^4_t(J)}.
\end{split}
\end{equation*}
A consequence of Strichartz estimates from Proposition \ref{Stric2} below is that
\begin{equation*}
\Vert u\Vert_{Z(I)}\lesssim \Vert u\Vert_{X^1(I)},
\end{equation*}
thus $Z$ is indeed a weaker norm. The purpose of this norm is that it is fungible and still controls the global evolution, as will be manifest from the local theory in Section \ref{localwp}.

{\bf Definition of solutions}.
Given an interval $I\subseteq\R$, we call $u\in C(I:H^1(\mathbb{R}\times\mathbb{T}^3))$ a strong solution of \eqref{eq1} if $u\in X^1(I)$ and $u$ satisfies that for all $t,s\in I$,
\begin{equation*}
u(t)=e^{i(t-s)\Delta}u(s)-i\int_s^te^{i(t-t^\prime)\Delta}\left(u(t^\prime)\vert u(t^\prime)\vert^2\right)dt^\prime.
\end{equation*}

{\bf{Dispersive estimates on $\R\times\T^3$.}} We first recall the dispersion estimate in $\mathbb{R}^d$
\begin{equation*}
\Vert e^{it\Delta_{\mathbb{R}^d}}f\Vert_{L^\infty(\mathbb{R}^d)}\lesssim \vert t\vert^{-\frac{d}{2}}\Vert f\Vert_{L^1(\mathbb{R}^d)}.
\end{equation*}
Using also the unitarity of $e^{it\Delta_{\mathbb{R}^d}}$ on $L^2(\mathbb{R}^d)$ one can obtain the Euclidean Strichartz estimates. However, in our context, such estimates do not hold. Instead, we obtain the following:

\begin{proposition}\label{Stric2}
(Strichartz estimates on $\R\times\mathbb{T}^3$) Let
\begin{equation}\label{no0}
p_1:=18/5.
\end{equation}
Then, for any $p> p_1$, $N\geq 1$, and $f\in L^2(\R\times\T^3)$
\begin{equation}\label{no0.001}
\|e^{it\Delta}P_Nf\|_{L^p(\R\times\mathbb{T}^3\times[-1,1])}\lesssim_p N^{2-6/p}\|f\|_{L^2(\R\times\mathbb{T}^3)}.
\end{equation}
\end{proposition}

We prove this estimate in section \ref{lastsection}, using the circle method as in \cite{Bo2}. The scale-invariant inequality \eqref{no0.001} clearly holds for $p=\infty$; as explained in \cite{HeTaTz2}, it is important to prove this inequality for some exponent $p<4$, which is consistent with the restriction \eqref{no0}.

In addition, we will need the following dispersive bound, which is a consequence of the estimate \eqref{no4}: if $t\in\R$ and $N\geq 1$ then
\begin{equation}\label{Stric3}
\|e^{it\Delta}P_N f\|_{L^\infty(\R\times\T^3)}\lesssim N^3|t|^{-1/2}\|f\|_{L^1(\R\times\T^3)}.
\end{equation}

From Proposition \ref{Stric2}, we deduce the following lemma
\begin{lemma}
Let $p>p_1$, then, for any cube $C$ of size $N$, there holds that
\begin{equation}\label{UpEst}
\Vert P_Cu\Vert_{L^p(\mathbb{R}\times\mathbb{T}^3\times[-1,1])}\lesssim N^{2-\frac{6}{p}}\Vert P_C u\Vert_{U^p_\Delta(\mathbb{R},L^2)}.
\end{equation}
\end{lemma}

{\bf{Additional estimates.}} We will need the following two Sobolev-type estimates:

\begin{lemma}\label{precsob}
For $f\in H^1(\mathbb{R}\times\mathbb{T}^3)$, there holds that
\begin{equation}\label{PrecSob}
\Vert f\Vert_{L^4(\mathbb{R}\times\mathbb{T}^3)}\lesssim \left(\sup_N N^{-1}\Vert P_Nf\Vert_{L^\infty(\mathbb{R}\times\mathbb{T}^3)}\right)^\frac{1}{2}\|f\|_{H^1(\R\times\T^3)}^\frac{1}{2}.
\end{equation}
\end{lemma}

\begin{proof}[Proof of Lemma \ref{precsob}] Decomposing into different frequencies, ordering these frequencies, remarking that the two highest frequencies have to be comparable and estimating the two low-frequency terms in $L^\infty$ and the two high frequency terms in $L^2$, we get,
\begin{equation*}
\begin{split}
\Vert f\Vert_{L^4}^4&\lesssim\sum_{N_1\sim N_2\ge N_3\ge N_4}\int_{\mathbb{R}\times\mathbb{T}^3}\vert P_{N_1}fP_{N_2}fP_{N_3}fP_{N_4}f\vert dx\\
&\lesssim \left(\sup_N N^{-1}\Vert P_Nf\Vert_{L^\infty}\right)^2\sum_{N_1\sim N_2\ge N_3\ge N_4}N_3N_4\Vert P_{N_1}f\Vert_{L^2}\Vert P_{N_2}f\Vert_{L^2}\\
&\lesssim\left(\sup_N N^{-1}\Vert P_Nf\Vert_{L^\infty}\right)^2\sum_{N_1\sim N_2}N_1^2\Vert P_{N_1}f\Vert_{L^2}\Vert P_{N_2}f\Vert_{L^2}
\end{split}
\end{equation*}
which gives \eqref{PrecSob}.
\end{proof}

Given $\delta\in(0,1)$ we define the operator $\widetilde{P}_\delta^1$ on $L^2(\R\times\T^3)$,
\begin{equation}\label{sobop}
\mathcal{F}(\widetilde{P}_\delta^1 f)(\xi):=\sum_{N\ge 1}\mathcal{F}(P_Nf)(\xi)\cdot(1-\eta^1)(\xi_1/(\delta N)).
\end{equation}
Also, given a function $f:\R\times\T^3\times I\to\mathbb{C}$ we define
\begin{equation}\label{grad}
|\nabla^1 f|(x,t):=|f(x,t)|+\sum_{j=1}^4|\partial_jf(x,t)|.
\end{equation}

\begin{lemma}\label{locsmo}
Assume $\psi\in H^1(\R\times\T^3)$ satisfies
\begin{equation}\label{ml1}
\|\psi\|_{H^1(\R\times\T^3)}\leq 1,\qquad \sup_{K\geq 1,\,|t|\leq 1,\,x\in\mathbb{R}\times\T^3}K^{-1}|P_Ke^{it\Delta}\psi(x)|\leq\delta,
\end{equation}
for some $\delta\in (0,1]$. Then, for any $R>0$ there is $C(R)\geq 1$ such that
\begin{equation}\label{ml2}
N\|\,|\nabla^1 e^{it\Delta}\widetilde{P}_\delta^1\psi|\,\|_{L^2_{x,t}(\{(x,t)\in\R\times\T^3\times[-1,1]:|x-x_0|\leq RN^{-1},\,|t-t_0|\leq R^2N^{-2}\})}\leq C(R)\delta^{1/100}
\end{equation}
for any $N\geq 1$, any $t_0\in\mathbb{R}$, and any $x_0\in\R\times\T^3$.
\end{lemma}

\begin{proof}[Proof of Lemma \ref{locsmo}] We may assume $R=1$ and $x_0=0$. It follows from \eqref{ml1} that for any $K\geq 1$ and $t\in[-1,1]$
\begin{equation*}
\|P_Ke^{it\Delta}\psi\|_{L^\infty(\R\times\T^3)}\lesssim \delta K,\qquad \|P_Ke^{it\Delta}\psi\|_{L^4(\R\times\T^3)}\lesssim 1,
\end{equation*}
therefore, by interpolation,
\begin{equation*}
\|P_Ke^{it\Delta}\psi\|_{L^{8}(\R\times\T^3)}\lesssim \delta^{1/2} K^{1/2}.
\end{equation*}
Thus, for any $K\geq 1$ and $t\in[-1,1]$
\begin{equation*}
\|\,|\nabla^1(P_Ke^{it\Delta}\psi)|\,\|_{L^{8}(\R\times\T^3)}\lesssim \delta^{1/2} K^{3/2},
\end{equation*}
which shows that, for any $K\geq 1$, $N\geq 1$, and $t_0\in\mathbb{R}$
\begin{equation}\label{ml3}
N\|\,|\nabla^1(P_Ke^{it\Delta}\widetilde{P}_\delta^1\psi)|\,\|_{L^2_{x,t}(\{(x,t)\in\R\times\T^3\times[-1,1]:|x|\leq N^{-1},\,|t-t_0|\leq N^{-2}\})}\lesssim \delta^{1/2} K^{3/2}N^{-3/2}.
\end{equation}
We will prove below that for any $N\geq 1$ and $K\geq \delta^{-5}N$
\begin{equation}\label{ml4}
N\|\,|\nabla^1(P_Ke^{it\Delta}\widetilde{P}_\delta^1\psi)|\,\|_{L^2_{x,t}(\{(x,t)\in\R\times\T^3\times[-1,1]:|x|\leq N^{-1},\,|t-t_0|\leq N^{-2}\})}\lesssim N^{1/2}(\delta K)^{-1/2}.
\end{equation}
The desired bound \eqref{ml2} follows from \eqref{ml3} and \eqref{ml4}.

To prove the local smoothing bound \eqref{ml4} it suffices to prove the stronger bound
\begin{equation}\label{m15}
\sup_{x_1\in\mathbb{R}}\|P_Ke^{it\Delta}\widetilde{P}_\delta^1\phi\|_{L^2_{x',t}(\T^3\times\R)}\lesssim (\delta K)^{-1/2}\|\phi\|_{L^2(\R\times\T^3)}
\end{equation}
for any $\phi\in L^2(\R\times\T^3)$, $\delta\in(0,1]$, and $K\geq 1$. In view of the Fourier inversion formula it suffices to prove that
\begin{equation*}
\begin{split}
\Big\|\int_{\R\times\Z^3}m(\xi_1,\xi')e^{ix'\cdot\xi'}e^{-it\xi_1^2}e^{-it|\xi'|^2}&(1-\eta^1)(\xi_1/(\delta K))\,d\xi_1d\xi'\Big\|_{L^2_{x',t}(\T^3\times\R)}\\
&\lesssim (\delta K)^{-1/2}\|m\|_{L^2(\R\times\Z^3)}
\end{split}
\end{equation*}
for any $m\in L^2(\R\times\Z^3)$, $\delta\in(0,1]$, and $K\geq 1$. This is a direct consequence of the Plancherel identity.
\end{proof}

\section{Local well-posedness and stability theory}\label{localwp}

The purpose of this section is to work out the local theory results that allow us to connect nearby intervals of nonlinear evolution. A consequence of Proposition \ref{Stric2} and \cite[Theorem 1.1]{HeTaTz2} is that the Cauchy problem for \eqref{eq1} is locally well-posed. However, here we want slightly more precise results.

We want to control the evolution in terms of the $Z$-norm. It turns out that it is more convenient to consider an intermediate quantity first,
\begin{equation*}
\|u\|_{Z^\prime(I)}:=\|u\|_{Z(I)}^{3/4}\|u\|_{X^1(I)}^{1/4}.
\end{equation*}

\begin{lemma}\label{lem1}
Assuming $|I|\leq 1$, $P_{N_i}u_i=u_i$ and $N_1\ge N_2$, there holds that
\begin{equation}\label{NLEst}
\begin{split}
\Vert u_1u_2\Vert_{L^2_{x,t}(\mathbb{R}\times\mathbb{T}^3\times I)}&\lesssim \left(\frac{N_2}{N_1}+\frac{1}{N_2}\right)^\kappa\Vert u_1\Vert_{Y^0(I)}\Vert u_2\Vert_{Z^\prime(I)}
\end{split}
\end{equation}
for some $\kappa>0$.
\end{lemma}

\begin{proof}[Proof of Lemma \ref{lem1}] For any cube $C$ centered at $\xi_0\in\mathbb{Z}^4$ and of size $N_2$, using Strichartz estimates we find that
\begin{equation*}
\begin{split}
\Vert (P_Cu_1)u_2\Vert_{L^2(\mathbb{R}\times\mathbb{T}^3\times I)}
&\lesssim \Vert P_Cu_1\Vert_{L^4(\mathbb{R}\times\mathbb{T}^3\times I)}\Vert u_2\Vert_{L^4(\mathbb{R}\times\mathbb{T}^3\times I)}\\
&\lesssim \Vert P_Cu_1\Vert_{U^4_\Delta(I,L^2)}\left(N_2^{\frac{1}{2}}\Vert u_2\Vert_{L^4(\mathbb{R}\times\mathbb{T}^3\times I)}\right)\\
&\lesssim \Vert P_Cu_1\Vert_{Y^0(I)}\left(N_2^{\frac{1}{2}}\Vert u_2\Vert_{L^4(\mathbb{R}\times\mathbb{T}^3\times I)}\right).
\end{split}
\end{equation*}
Using othogonality and the summability properties of $Y^0(I)$, we get that
\begin{equation*}
\begin{split}
\Vert u_1u_2\Vert_{L^2(\mathbb{R}\times\mathbb{T}^3\times I)}^2&\lesssim \sum_C\Vert P_Cu_1\Vert_{Y^0(I)}^2\left(N_2^{\frac{1}{2}}\Vert u_2\Vert_{L^4(\mathbb{R}\times\mathbb{T}^3\times I)}\right)^2\lesssim \Vert u_1\Vert_{Y^0(I)}^2\Vert u_2\Vert_{Z(I)}^2.
\end{split}
\end{equation*}

Independently, using Proposition \ref{Stric2}, we also have from \cite[Proposition 2.8]{HeTaTz2} that
\begin{equation}\label{pla100}
\begin{split}
\Vert u_1u_2\Vert_{L^2(\mathbb{R}\times\mathbb{T}^3\times I)}
&\lesssim \left(\frac{N_2}{N_1}+\frac{1}{N_2}\right)^{\delta_0}\Vert u_1\Vert_{Y^0(I)}\Vert u_2\Vert_{Y^1(I)}
\end{split}
\end{equation}
for some $\delta_0>0$. The desired bound \eqref{NLEst} follows from these two estimates and the inclusion $X^1(I)\hookrightarrow Y^1(I)$.
\end{proof}

Now we can obtain our main nonlinear estimate:

\begin{lemma}\label{NLEst2}
For $u_k\in X^1(I)$, $k=1,2,3$, $|I|\leq 1$, the estimate
\begin{equation*}
\Vert \Pi_{i=1}^3\tilde{u}_k\Vert_{N(I)}\lesssim \sum_{\{i,j,k\}=\{1,2,3\}}\Vert u_i\Vert_{X^1(I)}\Vert u_j\Vert_{Z^\prime(I)}\Vert u_k\Vert_{Z^\prime(I)}
\end{equation*}
holds true, where $\tilde{u}_k=u_k$ or $\tilde{u}_k=\overline{u}_k$.
\end{lemma}

\begin{proof}[Proof of Lemma \ref{NLEst2}] We decompose
\begin{equation*}
\begin{split}
\tilde{u}_1\tilde{u}_2\tilde{u}_3&=\sum_{N_1\geq 1}P_{N_1}\tilde{u}_1\cdot P_{\leq N_1}\tilde{u}_2\cdot P_{\leq N_1}\tilde{u}_3\\
&+\sum_{N_2\geq 2}P_{\leq N_2/2}\tilde{u}_1\cdot P_{N_2}\tilde{u}_2\cdot P_{\leq N_2}\tilde{u}_3+\sum_{N_3\geq 2}P_{\leq N_3/2}\tilde{u}_1\cdot P_{\leq N_3/2}\tilde{u}_2\cdot P_{N_3}\tilde{u}_3.
\end{split}
\end{equation*}
By symmetry, it suffices to prove that
\begin{equation}\label{refin}
\Big\|\sum_{N_1\geq 1}P_{N_1}\tilde{u}_1\cdot P_{\leq 2^{10}N_1}\tilde{u}_2\cdot P_{\leq 2^{10}N_1}\tilde{u}_3\Big\|_{N(I)}\lesssim \|u_1\|_{X^1(I)}\|u_2\|_{Z'(I)}\|u_3\|_{Z'(I)}.
\end{equation}

From \cite[Proposition 2.10]{HeTaTz2}, it suffices to prove the multilinear estimate
\begin{equation*}
\begin{split}
\Big\vert\int_{\mathbb{R}\times\mathbb{T}^3\times I}u_0\cdot \sum_{N_1\geq 1}&P_{N_1}\tilde{u}_1\cdot P_{\leq 2^{10}N_1}\tilde{u}_2\cdot P_{\leq 2^{10}N_1}\tilde{u}_3\,dxdt\Big\vert\\
&\lesssim \Vert u_0\Vert_{Y^{-1}(I)}\Vert u_1\Vert_{X^1(I)}\Vert u_2\Vert_{Z^\prime(I)}\Vert u_3\Vert_{Z^\prime(I)}
\end{split}
\end{equation*}
for any $u_0\in Y^{-1}(I)$. To prove this, we decompose into dyadic shells
\begin{equation*}
u_k=\sum_{N_k}P_{N_k}u_k
\end{equation*}
and after relabeling, we need to estimate
\begin{equation*}
S=\sum_{\mathcal{N}}\Vert P_{N_1}\tilde{u}_1P_{N_3}\tilde{u}_3\Vert_{L^2_{x,t}(\mathbb{R}\times\mathbb{T}^3\times I)}\Vert P_{N_0}\tilde{u}_0P_{N_2}\tilde{u}_2\Vert_{L^2_{x,t}(\mathbb{R}\times\mathbb{T}^3\times I)}
\end{equation*}
where $\mathcal{N}$ is the set of $4$-tuples $(N_0,N_1,N_2, N_3)$ such that
\begin{equation*}
N_1\sim \max(N_2,N_0)\ge N_2\ge N_3.
\end{equation*}
We first consider the case $N_0\sim N_1$. In this case, using \eqref{NLEst}, summing first in $N_3$, then in $N_2$ and finally in $N_1$, we compute that
\begin{equation*}
\begin{split}
S_1&\lesssim\sum_{\mathcal{N}_1}\left(\frac{N_3}{N_1}+\frac{1}{N_3}\right)^\kappa\left(\frac{N_2}{N_0}+\frac{1}{N_2}\right)^\kappa \Vert P_{N_0}u_0\Vert_{Y^0(I)}\Vert P_{N_1}u_1\Vert_{Y^0(I)}\Vert P_{N_2}u_2\Vert_{Z^\prime(I)}\Vert P_{N_3}u_3\Vert_{Z^\prime(I)}\\
&\lesssim \Vert u_0\Vert_{Y^{-1}(I)} \Vert u_1\Vert_{X^1(I)}\Vert u_2\Vert_{Z^\prime(I)}\Vert u_3\Vert_{Z^\prime(I)}.
\end{split}
\end{equation*}
Now, we consider the case $N_0\le N_2$. In this case, using \eqref{UpEst} and \eqref{NLEst}, we see that
\begin{equation*}
\begin{split}
S_2&\lesssim\sum_{N_1\sim N_2\ge N_0,N_3}\Vert \left(P_{N_1}\tilde{u}_1\right)\left(P_{N_3}\tilde{u}_3\right)\Vert_{L^2_{x,t}(\mathbb{R}\times\mathbb{T}^3\times I)}\Vert \left(P_{N_2}\tilde{u}_2\right)\left(P_{N_0}\tilde{u}_0\right)\Vert_{L^2_{x,t}(\mathbb{R}\times\mathbb{T}^3\times I)}\\
&\lesssim \sum_{N_1\sim N_2\ge N_0,N_3}\left(\frac{N_3}{N_1}+\frac{1}{N_3}\right)^\kappa\Vert P_{N_1}u_1\Vert_{Y^0(I)}\Vert u_3\Vert_{Z^\prime(I)}\Vert P_{N_2}u_2\Vert_{L^4(\mathbb{R}\times\mathbb{T}^3\times I)}\Vert P_{N_0}u_0\Vert_{L^4(\mathbb{R}\times\mathbb{T}^3\times I)}\\
&\lesssim \Vert u_3\Vert_{Z^\prime(I)}\sum_{N_1\sim N_2\ge N_0}N_1^{-\frac{3}{2}}\Vert P_{N_1}u_1\Vert_{Y^1(I)}\left(N_2^\frac{1}{2}\Vert P_{N_2}u_2\Vert_{L^4(\mathbb{R}\times\mathbb{T}^3\times I)}\right)N_0^\frac{1}{2}\Vert P_{N_0}u_{0}\Vert_{U^4_\Delta(I,L^2)}.\\
\end{split}
\end{equation*}
We use the embedding $Y^0\hookrightarrow U^4_\Delta(I,L^2)$ to obtain, with Schur's lemma,
\begin{equation*}
\begin{split}
S_2&\lesssim \Vert u_3\Vert_{Z^\prime(I)}\Vert u_2\Vert_{Z^\prime(I)}\sum_{N_1\ge N_0}\left(\frac{N_0}{N_1}\right)^{\frac{3}{2}}\Vert P_{N_1}u_1\Vert_{Y^1(I)}\Vert P_{N_0}u_0\Vert_{Y^{-1}(I)}\\
&\lesssim \Vert u_3\Vert_{Z^\prime(I)}\Vert u_2\Vert_{Z^\prime(I)}\Vert u_1\Vert_{X^1(I)}\Vert u_0\Vert_{Y^{-1}(I)}.
\end{split}
\end{equation*}
This gives the desired estimate \eqref{refin} and finishes the proof.
\end{proof}

With these two estimates, we can state our local existence result and a criterion for global existence.

\begin{proposition}\label{LocTheory}
Assume that $E>0$ is fixed. There exists $\delta_0=\delta_0(E)$ such that if
\begin{equation}\label{CondForLocEx}
\Vert e^{it\Delta}u_0\Vert_{Z^\prime(I)}<\delta
\end{equation}
for some $\delta\le\delta_0$, some interval $0\in I$ with $\vert I\vert\le 1$ and some function $u_0\in H^1(\mathbb{R}\times\mathbb{T}^3)$ satisfying
$\Vert u_0\Vert_{H^1}\le E$, then there exists a unique strong solution to \eqref{eq1} $u\in X^1(I)$ such that $u(0)=u_0$. Besides,
\begin{equation}\label{UAlmostLin}
\Vert u-e^{it\Delta}u_0\Vert_{X^1(I)}\le \delta^\frac{5}{3}.
\end{equation}
\end{proposition}

\begin{proof}[Proof of Proposition \ref{LocTheory}]
We proceed by a standard fixed point argument.
For $E,a>0$, we consider the space
$$\mathcal{S}=\{u\in X^1(I):\Vert u\Vert_{X^1(I)}\le 2E,\Vert u\Vert_{Z^\prime(I)}\le a\}$$
and the mapping
\begin{equation*}
\Phi(v)=e^{it\Delta}u_0-i\int_0^te^{i(t-s)\Delta}\left(v(s)\vert v(s)\vert^2\right)ds.
\end{equation*}
We then see that, for $u,v\in \mathcal{S}$,
\begin{equation}\label{Contract}
\begin{split}
\Vert \Phi(u)-\Phi(v)\Vert_{X^1(I)}&\lesssim \left(\Vert u\Vert_{X^1(I)}+\Vert v\Vert_{X^1(I)}\right)\left(\Vert u\Vert_{Z^\prime(I)}+\Vert v\Vert_{Z^\prime(I)}\right)\Vert u-v\Vert_{X^1(I)}\\
&\lesssim Ea\Vert u-v\Vert_{X^1(I)}.
\end{split}
\end{equation}
Similarly, using Lemma \ref{NLEst2}, we also obtain that
\begin{equation*}
\begin{split}
\Vert \Phi(u)\Vert_{X^1(I)}&\le \Vert \Phi(0)\Vert_{X^1(I)}+C\Vert\Phi(v)-\Phi(0)\Vert_{X^1(I)}\le \Vert u_0\Vert_{H^1}+CEa^2,\\
\Vert \Phi(u)\Vert_{Z^\prime(I)}&\le \Vert \Phi(0)\Vert_{Z^\prime(I)}+C\Vert \Phi(u)-\Phi(0)\Vert_{X^1(I)}\le \delta+CEa^2.
\end{split}
\end{equation*}
Now, we choose $a=2\delta$ and we let $\delta_0=\delta_0(E)$ be small enough. We then see that $\Phi$ is a contraction on $\mathcal{S}$. Hence it possesses a unique fixed point $u$. Finally, using \eqref{NLEst} and taking $\delta_0$ smaller if necessary,
\begin{equation*}
\Vert u-e^{it\Delta}u_0\Vert_{X^1}\lesssim \Vert u\Vert_{X^1}\Vert u\Vert_{Z^\prime}^2\le \delta^\frac{5}{3}.
\end{equation*}
In order to justify the uniqueness in $X^1(I)$, assume that $u,v\in X^1(I)$ satisfy $u(0)=v(0)$ and choose $E=\max(\Vert u\Vert_{X^1(I)},\Vert v\Vert_{X^1(I)})$. Then there exists a possibly smaller open interval $0\in J$ such that $u,v\in\mathcal{S}_J$. By uniqueness of the fixed point in $\mathcal{S}$, $u_{\vert J}=v_{\vert J}$ and hence $\{u=v\}$ is closed and open in $I$.
This finishes the proof.
\end{proof}

\begin{lemma}[Z-norm controls the global existence]\label{BlowUpCriterion}
Assume that $I\subset\mathbb{R}$ is a bounded open interval, that $u$ is a strong solution of \eqref{eq1} and that
\begin{equation*}
\Vert u\Vert_{L^\infty_t(I,H^1)}\le E.
\end{equation*}
Then, if
\begin{equation*}
\Vert u\Vert_{Z(I)}<+\infty
\end{equation*}
there exists an open interval $J$ with $\overline{I}\subset J$ such that $u$ can be extended to a strong solution of \eqref{eq1} on $J$. In particular, if $u$ blows up in finite time, then $u$ blows up in the $Z$-norm.
\end{lemma}

\begin{proof}[Proof of Lemma \ref{BlowUpCriterion}] It suffices to consider the case $I=(0,T)$.
Choose $\varepsilon>0$ sufficiently small and  $T_1\ge T-1$ such that
\begin{equation*}
\Vert u\Vert_{Z(T_1,T)}\le\varepsilon.
\end{equation*}
Now, let
\begin{equation*}
h(s)=\Vert e^{i(t-T_1)\Delta}u(T_1)\Vert_{Z^\prime(T_1,T_1+s)}.
\end{equation*}
Clearly, $h$ is a continuous function of $s$ satisfying $h(0)=0$. In addition, using Proposition \ref{LocTheory}, as long as $h(s)\le\delta_0$, we have that
\begin{equation*}
\Vert u(t)-e^{i(t-T_1)\Delta}u(T_1)\Vert_{X^1(T_1,T_1+s)}\le h(s)^\frac{5}{3}
\end{equation*}
and in particular, from Duhamel formula,
\begin{equation*}
\begin{split}
\Vert e^{i(t-T_1)\Delta}u(T_1)\Vert_{Z(T_1,T_1+s)}&\le \Vert u\Vert_{Z(T_1,T_1+s)}+C\Vert u(t)-e^{i(t-T_1)\Delta}u(T_1)\Vert_{X^1(T_1,T_1+s)}\\
&\le \varepsilon+Ch(s)^\frac{5}{3}.
\end{split}
\end{equation*}
On the other hand, by definition, there holds that
\begin{equation*}
\begin{split}
h(s)&\le \Vert e^{i(t-T_1)\Delta}u(T_1)\Vert_{Z(T_1,T_1+s)}^\frac{3}{4}\Vert e^{i(t-T_1)\Delta}u(T_1)\Vert_{X^1(T_1,T_1+s)}^\frac{1}{4}\\
&\le\left(\varepsilon+Ch(s)^\frac{5}{3}\right)^\frac{3}{4}\Vert u(T_1)\Vert_{H^1}^\frac{1}{4}.
\end{split}
\end{equation*}
Now, if we take $\varepsilon>0$ sufficiently small, we see by continuity that $h(s)\le C\sqrt{\varepsilon}E^\frac{1}{4}\leq\delta_0/2$ for all $s\le T-T_1$. Consequently, there exists a larger time $T_2>T$ such that
\begin{equation*}
\Vert e^{i(t-T_1)\Delta}u(T_1)\Vert_{Z^\prime(T_1,T_2)}\leq\frac{3}{4}\delta_0.
\end{equation*}
Applying Proposition \ref{LocTheory}, we see that $u$ can be extended to $(0,T_2)$.
\end{proof}

Finally, we conclude with a stability result.

\begin{proposition}[Stability]\label{Stabprop}
Assume $I$ is an open bounded interval, $\rho\in[-1,1]$, and $\widetilde{u}\in X^1(I)$ satisfies the approximate  Schr\"{o}dinger equation
\begin{equation}\label{ANLS}
(i\partial_t+\Delta)\widetilde{u}=\rho\widetilde{u}|\widetilde{u}|^2+e\quad\text{ on }\mathbb{R}\times\mathbb{T}^3\times I.
\end{equation}
Assume in addition that
\begin{equation}\label{ume}
\|\widetilde{u}\|_{Z(I)}+\|\widetilde{u}\|_{L^\infty_t(I,H^1(\mathbb{R}\times\mathbb{T}^3))}\leq M,
\end{equation}
for some $M\in[1,\infty)$. Assume $t_0 \in I$ and $u_0\in H^1(\mathbb{R}\times\mathbb{T}^3)$ is such that the smallness condition
\begin{equation}\label{safetycheck}
\|u_0 - \widetilde{u}(t_0)\|_{H^1(\mathbb{R}\times\mathbb{T}^3)}+\| e\|_{N(I)}\leq \eps
\end{equation}
holds for some $0 < \eps < \eps_1$, where $\eps_1\leq 1$ is a small constant $\eps_1 = \eps_1(M) > 0$.

Then there exists a strong solution $u\in X^1(I)$ of the Schr\"{o}dinger equation
\begin{equation}\label{ANLS2}
(i\partial_t+\Delta)u=\rho u|u|^2
\end{equation}
 such that $u(t_0)=u_0$ and
\begin{equation}\label{output}
\begin{split}
\| u \|_{X^1(I)}+\|\widetilde{u}\|_{X^1(I)}&\leq C(M),\\
\| u - \widetilde u \|_{X^1(I)}&\leq C(M)\eps.
 \end{split}
\end{equation}
\end{proposition}

\begin{proof}[Proof of Proposition \ref{Stabprop}] Without loss of generality, we may assume that $\vert I\vert\le 1$. We proceed in several steps.

First we need the following variation on the local existence result. We claim that
there exists $\delta_1=\delta_1(M)$ such that if, for some interval $J\ni t_0$,
\begin{equation*}
\begin{split}
\Vert e^{i(t-t_0)\Delta}\tilde{u}(t_0)\Vert_{Z^\prime(J)}+\Vert e\Vert_{N(J)}&\le\delta_1,\\
\end{split}
\end{equation*}
then there exists a unique solution $\tilde{u}$ of \eqref{ANLS} on $J$ and
\begin{equation*}
\Vert \tilde{u}-e^{i(t-t_0)\Delta}\tilde{u}(t_0)\Vert_{X^1(J)}\le \Vert e^{i(t-t_0)\Delta}\tilde{u}(t_0)\Vert_{Z^\prime(J)}^\frac{5}{3}+2\Vert e\Vert_{N(J)}.
\end{equation*}
The proof is very similar to the proof of Proposition \ref{LocTheory} and is omitted.

Now we claim that there exists $\varepsilon_1=\varepsilon_1(M)$ such that if the inequalities
\begin{equation}\label{AssumpZ1}
\begin{split}
\Vert e\Vert_{N(I_k)}&\le\varepsilon_1\\
\Vert \tilde{u}\Vert_{Z(I_k)}&\le\varepsilon\le\varepsilon_1,
\end{split}
\end{equation}
hold on $I_k=(T_k,T_{k+1})$, then
\begin{equation}\label{SmallZnorm}
\begin{split}
\Vert e^{i(t-T_k)\Delta}\tilde{u}(T_k)\Vert_{Z^\prime(I_k)}&\le C(1+M)\left(\varepsilon+\Vert e\Vert_{N(I_k)}\right)^\frac{3}{4},\\
\Vert \tilde{u}\Vert_{Z^\prime(I_k)}&\le C(1+M)\left(\varepsilon+\Vert e\Vert_{N(I_k)}\right)^\frac{3}{4}.
\end{split}
\end{equation}

Define $h(s)=\Vert e^{i(t-T_k)\Delta}\tilde{u}(T_k)\Vert_{Z^\prime(T_k,T_k+s)}$. Let $J_k=[T_k,T^\prime)\subset I_k$ be the largest interval such that
$h(s)\le\delta_1/2$, where $\delta_1=\delta_1(M)$ is as in the claim above. Then, on the one hand, we see from Duhamel Formula that
\begin{equation*}
\begin{split}
\Vert e^{i(t-T_k)\Delta}\tilde{u}(T_k)\Vert_{Z(T_k,T_k+s)}&\le \Vert \tilde{u}\Vert_{Z(T_k,T_k+s)}+\Vert \tilde{u}-e^{i(t-T_k)\Delta}\tilde{u}(T_k)\Vert_{X^1(T_k,T_k+s)}\\
&\le \varepsilon+h(s)^\frac{5}{3}+2\Vert e\Vert_{N(I_k)}.
\end{split}
\end{equation*}
On the other hand, we also have that
\begin{equation*}
\begin{split}
h(s)&\le \Vert e^{i(t-T_k)\Delta}\tilde{u}(T_k)\Vert_{Z(T_k,T_k+s)}^\frac{3}{4}\Vert e^{i(t-T_k)\Delta}\tilde{u}(T_k)\Vert_{X^1(T_k,T_k+s)}^\frac{1}{4}\\
&\le \left(\varepsilon+h(s)^\frac{5}{3}+2\Vert e\Vert_{N(I_k)}\right)^\frac{3}{4}M^\frac{1}{4}\\
&\le C(1+M)\left(\varepsilon+\Vert e\Vert_{N(I_k)}\right)^\frac{3}{4}+C(1+M)h(s)^{\frac{5}{4}}.
\end{split}
\end{equation*}
The claim \eqref{SmallZnorm} follows provided that $\varepsilon_1$ is chosen small enough.

Now we consider an interval $I_k=(T_k,T_{k+1})$ on which we assume that
\begin{equation}\label{AssumptSmallness}
\begin{split}
\Vert e^{i(t-T_k)\Delta}\tilde{u}(T_k)\Vert_{Z^\prime(I_k)}&\le\varepsilon\le\varepsilon_0\\
\Vert \tilde{u}\Vert_{Z^\prime(I_k)}&\le \varepsilon\le \varepsilon_0\\
\Vert e\Vert_{N(I_k)}&\le \varepsilon_0
\end{split}
\end{equation}
for some constant $\varepsilon_0$ sufficiently small. We can control $\tilde{u}$ on $I_k$ as follows
\begin{equation*}
\begin{split}
\Vert \tilde{u}\Vert_{X^1(I_k)}&\le \Vert e^{i(t-T_k)\Delta}\tilde{u}(T_k)\Vert_{X^1(I_k)}+\Vert \tilde{u}-e^{i(t-T_k)\Delta}\tilde{u}(T_k)\Vert_{X^1(I_k)}\le M+1.
\end{split}
\end{equation*}
Now, we let $u$ be an exact strong solution of \eqref{ANLS2} defined on an interval $J_u$ such that
\begin{equation}\label{DefOfAk}
a_k=\Vert \tilde{u}(T_k)-u(T_k)\Vert_{H^1(\mathbb{R}\times\mathbb{T}^3)}\le\varepsilon_0
\end{equation}
and we let $J_k=[T_k,T_k+s]\cap I_k\cap J_u$ be the maximal interval such that
\begin{equation}\label{AssumptionStabAdded1}
\Vert \omega\Vert_{Z^\prime(J_k)}\le 10C\varepsilon_0\le 1/(10(M+1)),
\end{equation}
where $\omega:=u-\widetilde{u}$. Such an interval exists and is nonempty since $s\mapsto\Vert \omega\Vert_{Z^\prime(T_k,T_k+s)}$ is finite and continuous on $J_u$ and vanishes for $s=0$. Then, we see that $\omega=u-\tilde{u}$ solves
\begin{equation*}
\left(i\partial_t+\Delta\right)\omega=\rho\left((\tilde{u}+\omega)\vert \tilde{u}+\omega\vert^2-\tilde{u}\vert\tilde{u}\vert^2\right)-e
\end{equation*}
and consequently, using Proposition \ref{NLEst2}, we get that
\begin{equation}\label{estimJk}
\begin{split}
\Vert \omega\Vert_{X^1(J_k)}&\le \Vert e^{i(t-T_k)\Delta}(u(T_k)-\tilde{u}(T_k))\Vert_{X^1(J_k)}+\Vert (\tilde{u}+\omega)\vert \tilde{u}+\omega\vert^2-\tilde{u}\vert\tilde{u}\vert^2\Vert_{N(J_k)}+\Vert e\Vert_{N(J_k)}\\
&\le \Vert u(T_k)-\tilde{u}(T_k)\Vert_{H^1}+C\big(\Vert \tilde{u}\Vert_{X^1(J_k)}\Vert \tilde{u}\Vert_{Z^\prime(J_k)}\Vert \omega\Vert_{X^1(J_k)}\\
&+\Vert \tilde{u}\Vert_{X^1(J_k)}\Vert \omega\Vert_{X^1(J_k)}\Vert\omega\Vert_{Z^\prime(J_k)}+\Vert \omega\Vert_{X^1(J_k)}\Vert \omega\Vert_{Z^\prime(J_k)}^2+\Vert e\Vert_{N(J_k)}\big)\\
&\le \Vert u(T_k)-\tilde{u}(T_k)\Vert_{H^1}+C\big(\varepsilon_0\Vert \omega\Vert_{X^1(J_k)}+\Vert e\Vert_{N(I_k)}\big),
\end{split}
\end{equation}
where the last line holds due to \eqref{AssumptionStabAdded1}. Consequently,
\begin{equation*}
\Vert \omega\Vert_{Z^\prime(J_k)}\le C\Vert \omega\Vert_{X^1(J_k)}\le 4C\left(\Vert u(T_k)-\tilde{u}(T_k)\Vert_{H^1}+\Vert e\Vert_{N(I_k)}\right)\le 8C\varepsilon_0
\end{equation*}
provided that $\varepsilon_0$ is small enough. Consequently $J_k=I_k\cap J_u$ and \eqref{estimJk} holds on $I_k\cap J_u$. Therefore, it follows from Lemma \ref{BlowUpCriterion} that the solution $u$ can be extended to the entire interval $I_k$, and the bounds \eqref{AssumptionStabAdded1} and \eqref{estimJk} hold with $J_k=I_k$.
\medskip

Now we can finish the proof. We take $\varepsilon_2(M)<\varepsilon_1(M)$ sufficiently small and split $I$ into $N=O(\Vert \tilde{u}\Vert_{Z(I)}/\varepsilon_2)^4$ intervals such that
\begin{equation*}
\Vert \tilde{u}\Vert_{Z(I_k)}\le \varepsilon_2,\qquad \Vert e\Vert_{N(I_k)}\le\kappa\varepsilon_2.
\end{equation*}
Then, on each interval, we have that \eqref{AssumpZ1} hold, so that we also have from \eqref{SmallZnorm} that \eqref{AssumptSmallness} holds. As a consequence, the bounds \eqref{AssumptionStabAdded1} and \eqref{estimJk} both hold on each interval $I_k$. The conclusion of the proposition follows.
\end{proof}

\section{Euclidean approximations}\label{Eucl}

In this section we prove precise estimates showing how to compare Euclidean and semiperiodic solutions of both linear and nonlinear Schr\"{o}dinger equations. Such a comparison is meaningful only in the case of rescaled data that concentrate at a point. We follow closely the arguments in \cite[Section 4]{IoPaSt}, but provide all the details in our case for the sake of completeness.

We fix a spherically-symmetric function $\eta\in C^\infty_0(\mathbb{R}^4)$ supported in the ball of radius $2$ and equal to $1$ in the ball of radius $1$. Given $\phi\in \dot{H}^1(\mathbb{R}^4)$ and a real number $N\geq 1$ we define
\begin{equation}\label{rescaled}
\begin{split}
Q_N\phi\in H^1(\mathbb{R}^4),\qquad &(Q_N\phi)(x)=\eta(x/N^{1/2})\phi(x),\\
\phi_N\in H^1(\mathbb{R}^4),\qquad &\phi_N(x)=N(Q_N\phi)(Nx),\\
f_{N}\in H^1(\R\times\T^3),\qquad &f_{N}(y)=\phi_N(\Psi^{-1}(y)),
\end{split}
\end{equation}
where $\Psi:\{x\in\mathbb{R}^4:|x|<1\}\to O_0\subseteq \R\times\T^3$, $\Psi(x)=x$. Thus $Q_N\phi$ is a compactly supported\footnote{This modification is useful to avoid the contribution of $\phi$ coming from the Euclidean infinity, in a uniform way depending on the scale $N$.} modification of the profile $\phi$, $\phi_N$ is an $\dot{H}^1$-invariant rescaling of $Q_N\phi$, and $f_{N}$ is the function obtained by transferring $\phi_N$ to a neighborhood of $0$ in $\R\times\T^3$. We define also
\begin{equation*}
E_{\mathbb{R}^4}(\phi)=\frac{1}{2}\int_{\mathbb{R}^4}|\nabla_{\R^4}\phi|^2\,dx+\frac{1}{4}\int_{\mathbb{R}^4}|\phi|^4\,dx.
\end{equation*}

We will use the main theorem of \cite{RV}, in the following form.

\begin{theorem}\label{MainThmEucl}
Assume $\psi\in\dot{H}^1(\mathbb{R}^4)$. Then there is a unique global solution $v\in C(\mathbb{R}:\dot{H}^1(\mathbb{R}^4))$ of the initial-value problem
\begin{equation}\label{clo3}
(i\partial_t+\Delta_{\R^4})v=v|v|^2,\qquad v(0)=\psi,
\end{equation}
and
\begin{equation}\label{clo4}
\|\,|\nabla_{\R^4} v|\,\|_{(L^\infty_tL^2_x\cap L^2_tL^4_x)(\mathbb{R}^4\times\mathbb{R})}\leq \widetilde{C}(E_{\mathbb{R}^4}(\psi)).
\end{equation}
Moreover this solution scatters in the sense that there exists $\psi^{\pm\infty}\in\dot{H}^1(\mathbb{R}^4)$ such that
\begin{equation}\label{EScat}
\Vert v(t)-e^{it\Delta}\psi^{\pm\infty}\Vert_{\dot{H}^1(\mathbb{R}^4)}\to 0
\end{equation}
as $t\to\pm\infty$. Besides, if $\psi\in H^5(\mathbb{R}^4)$ then $v\in C(\mathbb{R}:H^5(\mathbb{R}^4))$ and
\begin{equation*}
\sup_{t\in\mathbb{R}}\|v(t)\|_{H^5(\mathbb{R}^3)}\lesssim_{\|\psi\|_{H^5(\mathbb{R}^4)}}1.
\end{equation*}
\end{theorem}

Our first result in this section is the following lemma:

\begin{lemma}\label{step1}
Assume $\phi\in\dot{H}^1(\mathbb{R}^4)$, $T_0\in(0,\infty)$, and $\rho\in\{0,1\}$ are given, and define $f_{N}$ as in \eqref{rescaled}. Then the following conclusions hold:

(i) There is $N_0=N_0(\phi,T_0)$ sufficiently large such that for any $N\geq N_0$ there is a unique solution $U_{N}\in C((-T_0N^{-2},T_0N^{-2}):H^1(\R\times\T^3))$ of the initial-value problem
\begin{equation}\label{clo5}
(i\partial_t+\Delta)U_N=\rho U_N|U_N|^2,\qquad U_N(0)=f_N.
\end{equation}
Moreover, for any $N\geq N_0$,
\begin{equation}\label{clo6}
\|U_N\|_{X^1(-T_0N^{-2},T_0N^{-2})}\lesssim_{E_{\mathbb{R}^4}(\phi)}1.
\end{equation}

(ii) Assume $\varepsilon_1\in(0,1]$ is sufficiently small (depending only on $E_{\mathbb{R}^4}(\phi)$), $\phi'\in H^5(\mathbb{R}^4)$, and $\|\phi-\phi'\|_{\dot{H}^1(\mathbb{R}^4)}\leq\varepsilon_1$. Let $v'\in C(\mathbb{R}:H^5)$ denote the solution of the initial-value problem
\begin{equation*}
(i\partial_t+\Delta_{\R^4})v'=\rho v'|v'|^2,\qquad v'(0)=\phi'.
\end{equation*}
For $R,N\geq 1$ we define
\begin{equation}\label{clo9}
\begin{split}
&v'_R(x,t)=\eta(x/R)v'(x,t),\qquad\,\,\qquad (x,t)\in\mathbb{R}^4\times(-T_0,T_0),\\
&v'_{R,N}(x,t)=Nv'_R(Nx,N^2t),\qquad\quad\,(x,t)\in\mathbb{R}^4\times(-T_0N^{-2},T_0N^{-2}),\\
&V_{R,N}(y,t)=v'_{R,N}(\Psi^{-1}(y),t)\qquad\quad\,\, (y,t)\in\R\times\T^3\times(-T_0N^{-2},T_0N^{-2}).
\end{split}
\end{equation}
Then there is $R_0\geq 1$ (depending on $T_0$ and $\phi'$ and $\varepsilon_1$) such that, for any $R\geq R_0$,
\begin{equation}\label{clo18}
\limsup_{N\to\infty}\|U_N-V_{R,N}\|_{X^1(-T_0N^{-2},T_0N^{-2})}\lesssim_{E_{\mathbb{R}^4}(\phi)}\varepsilon_1.
\end{equation}
\end{lemma}

\begin{proof}[Proof of Lemma \ref{step1}] All of the constants in this proof are allowed to depend on $E_{\mathbb{R}^4}(\phi)$ (in fact on the large constant $\widetilde{C}(E_{\mathbb{R}^4}(\phi))$ in \eqref{clo4}); for simplicity of notation we will not track this dependence explicitly. Using Theorem \ref{MainThmEucl}
\begin{equation}\label{clo7}
\begin{split}
&\|\nabla_{\R^4} v'\|_{(L^\infty_tL^2_x\cap L^2_tL^4_x)(\mathbb{R}^4\times\mathbb{R})}\lesssim 1,\\
&\sup_{t\in\mathbb{R}}\|v'(t)\|_{H^5(\mathbb{R}^4)}\lesssim_{\|\phi'\|_{H^5(\mathbb{R}^4)}}1.
\end{split}
\end{equation}
We will prove that for any $R_0$ sufficiently large there is $N_0$ such that $V_{R_0,N}$ is an almost-solution of the equation \eqref{clo5}, for any $N\geq N_0$. We will then apply Proposition \ref{Stabprop} to upgrade this to an exact solution of the initial-value problem \eqref{clo5} and prove the lemma.

Let
\begin{equation*}
\begin{split}
e_R(x,t):&=[(i\partial_t+\Delta_{\R^4})v'_R-\rho v'_R|v'_R|^2](x,t)=\rho(\eta(x/R)-\eta(x/R)^3)v'(x,t)|v'(x,t)|^2\\
&+R^{-2}v'(x,t)(\Delta_{\R^4}\eta)(x/R)+2R^{-1}\sum_{j=1}^4\partial_jv'(x,t)\partial_j\eta(x/R).
\end{split}
\end{equation*}
Since $|v'(x,t)|\lesssim_{\|\phi'\|_{H^5(\mathbb{R}^4)}}1$, see \eqref{clo7}, it follows that
\begin{equation*}
\begin{split}
|e_R(x,t)|&+\sum_{k=1}^4|\partial_ke_R(x,t)|\\
&\lesssim_{\|\phi'\|_{H^5(\mathbb{R}^4)}}\mathbf{1}_{[R,2R]}(|x|)\cdot\big[|v'(x,t)|+\sum_{k=1}^3|\partial_kv'(x,t)|+\sum_{k,j=1}^3|\partial_k\partial_jv'(x,t)|\big].
\end{split}
\end{equation*}
Therefore
\begin{equation}\label{clo10}
\lim_{R\to\infty}\|\,|e_R|+|\nabla_{\R^4} e_R|\,\|_{L^2_tL^2_x(\mathbb{R}^4\times(-T_0,T_0))}=0.
\end{equation}
Letting
\begin{equation*}
e_{R,N}(x,t):=[(i\partial_t+\Delta_{\R^4})v'_{R,N}-\rho v'_{R,N}|v'_{R,N}|^2](x,t)=N^3e_R(Nx,N^2t),
\end{equation*}
it follows from \eqref{clo10} that there is $R_0\geq 1$ such that, for any $R\geq R_0$ and $N\geq 1$,
\begin{equation}\label{clo11}
\|\,|e_{R,N}|+|\nabla_{\R^4} e_{R,N}|\,\|_{L^1_tL^2_x(\mathbb{R}^4\times(-T_0N^{-2},T_0N^{-2}))}\leq\varepsilon_1.
\end{equation}

With $V_{R,N}(y,t)=v'_{R,N}(\Psi^{-1}(y),t)$ as in \eqref{clo9} and $N\geq 10R$, let
\begin{equation}\label{clo13}
E_{R,N}(y,t):=[(i\partial_t+\Delta)V_{R,N}-\rho V_{R,N}|V_{R,N}|^2](y,t)=e_{R,N}(\Psi^{-1}(y),t).
\end{equation}
Using \eqref{clo11}, it follows that for any $R_0$ sufficiently large there is $N_0$ such that for any $N\geq N_0$
\begin{equation}\label{clo15}
\|\,|\nabla^1 E_{R_0,N}|\,\|_{L^1_tL^2_x(\R\times\T^3\times(-T_0N^{-2},T_0N^{-2}))}\leq 2\varepsilon_1.
\end{equation}

To verify the hypothesis \eqref{ume} of Proposition \ref{Stabprop}, we estimate for $N$ large enough, using \eqref{clo7}
\begin{equation}\label{clo16}
\sup_{t\in(-T_0N^{-2},T_0N^{-2})}\|V_{R_0,N}(t)\|_{H^1(\R\times\T^3)}\leq \sup_{t\in(-T_0N^{-2},T_0N^{-2})}\|v'_{R_0,N}(t)\|_{H^1(\mathbb{R}^4)}\lesssim 1.
\end{equation}
In addition, using Littlewood--Paley theory and \eqref{clo7}, for $N$ large enough
\begin{equation}\label{clo16.5}
\begin{split}
\|V_{R_0,N}\|_{Z(-T_0N^{-2},T_0N^{-2})}&\lesssim\|(1-\Delta)^{1/4}V_{R_0,N}\|_{L^4(\R\times\T^3\times(-T_0N^{-2},T_0N^{-2}))}\\
&\lesssim \|(1-\Delta)^{1/2}V_{R_0,N}\|_{L^4_tL^{8/3}_x(\R\times\T^3\times(-T_0N^{-2},T_0N^{-2}))}\\
&\lesssim \|\,|v'_{R_0,N}|+|\nabla_{\R^4} v'_{R_0,N}|\,\|_{L^4_tL^{8/3}_x(\R^4\times(-T_0N^{-2},T_0N^{-2}))}\\
&\lesssim 1.
\end{split}
\end{equation}

Finally, to verify the inequality on the first term in \eqref{safetycheck} we estimate, for $R_0,N$ large enough,
\begin{equation}\label{clo17}
\begin{split}
\|f_N-V_{R_0,N}(0)&\|_{H^1(\R\times\mathbb{T}^3)}\lesssim \|\phi_N-v'_{R_0,N}(0)\|_{\dot{H}^1(\mathbb{R}^4)}\lesssim\|Q_N\phi-v'_{R_0}(0)\|_{\dot{H}^1(\mathbb{R}^4)}\\
&\lesssim \|Q_N\phi-\phi\|_{\dot{H}^1(\mathbb{R}^4)}+\|\phi-\phi'\|_{\dot{H}^1(\mathbb{R}^4)}+\|\phi'-v'_{R_0}(0)\|_{\dot{H}^1(\R^4)}\lesssim \varepsilon_1.
\end{split}
\end{equation}
The conclusion of the lemma follows from Proposition \ref{Stabprop}, provided that $\varepsilon_1$ is fixed sufficiently small depending on $E_{\mathbb{R}^4}(\phi)$.
\end{proof}

As a consequence, we have the following key lemma:

\begin{lemma}\label{step2}
Assume $\psi\in \dot{H}^1(\mathbb{R}^4)$, $\varepsilon>0$, $I\subseteq \mathbb{R}$ is an interval, and
\begin{equation}\label{clo20}
\|\,|\nabla_{\R^4}(e^{it\Delta}\psi)|\,\|_{L^{4}_tL^{8/3}_x(\mathbb{R}^4\times I)}\leq\varepsilon.
\end{equation}
For $N\geq 1$ we define, as before,
\begin{equation*}
(Q_N\psi)(x)=\eta(x/N^{1/2})\psi(x),\,\,\psi_N(x)=N(Q_N\psi)(Nx),\,\,\widetilde{\psi}_N(y)=\psi_N(\Psi^{-1}(y)).
\end{equation*}
Then there is $N_1=N_1(\psi,\varepsilon)$ such that, for any $N\geq N_1$,
\begin{equation}\label{clo21}
\|e^{it\Delta}\widetilde{\psi}_N\|_{Z(N^{-2}I)}\lesssim\varepsilon.
\end{equation}
\end{lemma}

\begin{proof}[Proof of Lemma \ref{step2}] As before, the implicit constants may depend on $E_{\mathbb{R}^4}(\psi)$. We may assume that $\psi\in C^\infty_0(\mathbb{R}^4)$.

We show first that we can fix $T_1=T_1(\psi,\varepsilon)$ such that, for any $N\geq 1$, we have the extinction bound
\begin{equation}\label{extinct}
\|e^{it\Delta}\widetilde{\psi}_N\|_{Z(\mathbb{R}\setminus(-T_1N^{-2},T_1N^{-2}))}\lesssim\varepsilon.
\end{equation}
Using the dispersive estimate \eqref{Stric3}, for any $t\neq 0$ and $M\in\{1,2,4,\ldots\}$
\begin{equation*}
\|P_Me^{it\Delta}\widetilde{\psi}_N\|_{L^\infty(\R\times\T^3)}\lesssim M^3|t|^{-1/2}\|\widetilde{\psi}_N\|_{L^1(\R\times\T^3)}\lesssim_\psi M^3|t|^{-1/2}N^{-3}.
\end{equation*}
In addition, for any $t\neq 0$, $M\in\{1,2,4,\ldots\}$, and $p\in\{0,100\}$
\begin{equation*}
\|P_Me^{it\Delta}\widetilde{\psi}_N\|_{L^2(\R\times\T^3)}\lesssim M^{-2p}\|(1-\Delta)^p\widetilde{\psi}_N\|_{L^2(\R\times\T^3)}\lesssim_\psi M^{-2p}N^{2p-1}.
\end{equation*}
Therefore, for any $t\neq 0$ and $M\in\{1,2,4,\ldots\}$,
\begin{equation*}
\|P_Me^{it\Delta}\widetilde{\psi}_N\|_{L^6(\R\times\T^3)}\lesssim_\psi |t|^{-1/3}N^{-1/3}\cdot\min[M/N,N^{60}/M^{60}],
\end{equation*}
which shows that
\begin{equation}\label{extinct2}
\|P_Me^{it\Delta}\widetilde{\psi}_N\|_{L^6(\R\times\T^3\times[\mathbb{R}\setminus(-T_1N^{-2},T_1N^{-2})])}\lesssim_\psi T_1^{-1/6}\cdot\min[M/N,N^{60}/M^{60}].
\end{equation}
Using Proposition \ref{Stric2}, it follows that
\begin{equation*}
\|P_Me^{it\Delta}\widetilde{\psi}_N\|_{L^p(\R\times\T^3\times I)}\lesssim_\psi M^{2-6/p}N^{-1}
\end{equation*}
for $p=19/5$ and any interval $I\subseteq\R$, $|I|\leq 1$. By interpolation between the last two bounds, for any interval $I\subseteq\mathbb{R}\setminus(-T_1N^{-2},T_1N^{-2})$, $|I|\leq 1$,
\begin{equation*}
\|P_Me^{it\Delta}\widetilde{\psi}_N\|_{L^4(\R\times\T^3\times I)}\lesssim_\psi M^{-1/2}T_1^{-\delta}\min[M/N,N/M],
\end{equation*}
for some $\delta>0$. The desired bound \eqref{extinct} follows.

The bound \eqref{clo21} on the remaining interval $N^{-2}I\cap(-T_1N^{-2},T_1N^{-2})$ follows from Lemma \ref{step1} (ii) with $\rho=0$. This completes the proof of the lemma.
\end{proof}

We conclude this section with a lemma describing nonlinear solutions of the initial-value problem \eqref{eq1.1} corresponding to data concentrating at a point. In view of the profile analysis in the next section, we need to consider slightly more general data. Given $f\in L^2(\R\times\T^3)$, $t_0\in\mathbb{R}$, and $x_0\in\R\times\T^3$ we define
\begin{equation*}
\begin{split}
&(\pi_{x_0}f)(x):=f(x-x_0),\\
&(\Pi_{t_0,x_0})f(x)=(e^{-it_0\Delta}f)(x-x_0)=(\pi_{x_0}e^{-it_0\Delta}f)(x).
\end{split}
\end{equation*}
As in \eqref{rescaled}, given $\phi\in\dot{H}^1(\mathbb{R}^4)$ and $N\geq 1$, we define
\begin{equation*}
T_N\phi(x):=N\widetilde{\phi}(N\Psi^{-1}(x))\qquad\text{ where }\qquad\widetilde{\phi}(y):=\eta(y/N^{1/2})\phi(y),
\end{equation*}
and observe that
\begin{equation*}
T_N:\dot{H}^1(\mathbb{R}^4)\to H^1(\R\times\T^3)\text{ is a linear operator with }\|T_N\phi\|_{H^1(\R\times\T^3)}\lesssim \|\phi\|_{\dot{H}^1(\mathbb{R}^4)}.
\end{equation*}
Let $\widetilde{\mathcal{F}}_e$ denote the set of renormalized Euclidean frames
\begin{equation*}
\begin{split}
\widetilde{\mathcal{F}}_e:=\{(N_k,t_k,x_k)_{k\geq 1}:&\,N_k\in[1,\infty),\,t_k\in[-1,1],\,x_k\in\R\times\T^3,\\
&\,N_k\to\infty,\text{ and }t_k=0 \text{ or }N_k^2|t_k|\to\infty\}.
\end{split}
\end{equation*}

\begin{lemma}\label{GEForEP}
Assume that $(N_k,t_k,x_k)_k\in\widetilde{\mathcal{F}}_e$, $\phi\in\dot{H}^1(\mathbb{R}^4)$, and let $U_k(0)=\Pi_{t_k,x_k}(T_{N_k}\phi)$.

(i) For $k$ large enough (depending only on $\phi$) there is a nonlinear solution $U_k\in X^1(-2,2)$ of the initial-value problem \eqref{eq1.1} and
\begin{equation}\label{ControlOnZNormForEP}
\Vert U_k\Vert_{X^1(-2,2)}\lesssim_{E_{\mathbb{R}^4}(\phi)}1.
\end{equation}

(ii) There exists a Euclidean solution $u\in C(\mathbb{R}:\dot{H}^1(\mathbb{R}^4))$ of
\begin{equation}\label{EEq}
\left(i\partial_t+\Delta_{\R^4}\right)u=u\vert u\vert^2
\end{equation}
with scattering data $\phi^{\pm\infty}$ defined as in \eqref{EScat} such that the following holds, up to a subsequence:
for any $\varepsilon>0$, there exists $T(\phi,\varepsilon)$ such that for all $T\ge T(\phi,\varepsilon)$ there exists $R(\phi,\varepsilon,T)$ such that for all $R\ge R(\phi,\varepsilon,T)$, there holds that
\begin{equation}\label{ProxyEuclHyp}
\Vert U_k-\tilde{u}_k\Vert_{X^1(\vert t-t_k\vert\le TN_k^{-2})}\le\varepsilon,
\end{equation}
for $k$ large enough, where
\begin{equation*}
(\pi_{-x_k}\tilde{u}_k)(x,t)=N_k\eta(N_k\Psi^{-1}(x)/R)u(N_k\Psi^{-1}(x),N_k^2(t-t_k)).
\end{equation*}
In addition, up to a subsequence,
\begin{equation}\label{ScatEuclSol}
\Vert U_k(t)-\Pi_{t_k-t,x_k}T_{N_k}\phi^{\pm\infty}\Vert_{X^1(\{\pm(t-t_k)\geq TN_k^{-2}\}\cap (-2,2))}\le \varepsilon,
\end{equation}
for $k$ large enough (depending on $\phi,\varepsilon,T,R$).
\end{lemma}

\begin{proof}[Proof of Lemma \ref{GEForEP}] Clearly, we may assume that $x_k=0$.

We have two cases. If $t_k=0$ for any $k$ then the lemma follows from Lemma \ref{step1} and Lemma \ref{step2}: we let $u$ be the nonlinear Euclidean solution of \eqref{EEq} with $u(0)=\phi$ and notice that for any $\delta>0$ there is $T(\phi,\delta)$ such that
\begin{equation*}
\|\nabla_{\R^4} u\|_{L^3_{x,t}(\mathbb{R}^4\times\{|t|\geq T(\phi,\delta)\})}\leq\delta.
\end{equation*}
The bound \eqref{ProxyEuclHyp} follows for any fixed $T\geq T(\phi,\delta)$ from Lemma \ref{step1}. Assuming $\delta$ is sufficiently small and $T$ is sufficiently large (both depending on $\phi$ and $\varepsilon$), the bound \eqref{ScatEuclSol} then follow from Corollary \ref{step2} (which guarantees smallness of $\mathbf{1}_{\pm}(t)\cdot e^{it\Delta}U_k(\pm N_k^{-2}T(\phi,\delta))$ in $Z(-2,2)$) and Lemma \ref{LocTheory}.

Otherwise, if $\lim_{k\to\infty}N_k^2|t_k|=\infty$, we may assume by symmetry that $N_k^2t_k\to+\infty$. Then we let $u$ be the solution of
\eqref{EEq} such that
\begin{equation*}
\Vert\nabla_{\R^4}\left(u(t)-e^{it\Delta_{\R^4}}\phi\right)\Vert_{L^2(\mathbb{R}^4)}\to0
\end{equation*}
as $t\to-\infty$ (thus $\phi^{-\infty}=\phi$).
We let $\tilde{\phi}=u(0)$ and  apply the conclusions of the lemma to the frame $(N_k,0,0)_k\in\mathcal{F}_e$ and $V_k(s)$, the solution of \eqref{eq1} with initial data $V_k(0)=T_{N_k}\tilde{\phi}$. In particular, we see from the fact that $N_k^2t_k\to+\infty$ and \eqref{ScatEuclSol} that
\begin{equation*}
\Vert V_k(-t_k)-\Pi_{t_k,0}T_{N_k}\phi\Vert_{H^1(\R\times\T^3)}\to 0
\end{equation*}
as $k\to\infty$. Then, using Proposition \ref{Stabprop}, we see that
\begin{equation*}
\Vert U_k-V_k(\cdot-t_k)\Vert_{X^1(-2,2)}\to 0
\end{equation*}
as $k\to\infty$, and we can conclude by inspecting the behavior of $V_k$. This ends the proof.
\end{proof}

\section{Profile decompositions}\label{profiles}

In this section we show that given a bounded sequence of functions $f_k\in H^1(\R\times\T^3)$ we can construct suitable {\it{profiles}} and express the sequence in terms of these profiles. The statements and the arguments in this section are very similar to those in \cite[Section 5]{IoPaSt}. See also \cite{Ker} for the original proofs of Keraani in the Euclidean geometry.

As before, given $f\in L^2(\R\times\T^3)$, $t_0\in\mathbb{R}$, and $x_0\in\R\times\T^3$ we define
\begin{equation}\label{PI}
\begin{split}
&(\pi_{x_0}f)(x):=f(x-x_0),\\
&(\Pi_{t_0,x_0})f(x)=(e^{-it_0\Delta}f)(x-x_0)=(\pi_{x_0}e^{-it_0\Delta}f)(x).
\end{split}
\end{equation}
As in \eqref{rescaled}, given $\phi\in\dot{H}^1(\mathbb{R}^4)$ and $N\geq 1$, we define
\begin{equation}\label{TN}
T_N\phi(x):=N\widetilde{\phi}(N\Psi^{-1}(x))\qquad\text{ where }\qquad\widetilde{\phi}(y):=\eta(y/N^{1/2})\phi(y),
\end{equation}
and observe that
\begin{equation}\label{TN2}
T_N:\dot{H}^1(\mathbb{R}^4)\to H^1(\R\times\T^3)\text{ is a linear operator with }\|T_N\phi\|_{H^1(\R\times\T^3)}\lesssim \|\phi\|_{\dot{H}^1(\mathbb{R}^4)}.
\end{equation}

The following is our main definition.

\begin{definition}\label{DefPro}

\begin{enumerate}

\item We define a frame to be a sequence $\mathcal{F}=(N_k,t_k,x_k)_k$ with $N_k\ge 1$, $t_k\in [-1,1]$, $x_k\in\mathbb{R}\times\mathbb{T}^3$ such that either $N_k=1$ and $t_k=0$, (Scale-$1$ frame), or $N_k\to+\infty$ (Euclidean frame). We denote $\mathcal{F}_1$ the set of Scale-1 frames and $\mathcal{F}_e$ the set of Euclidean frames. We say that two frames $(N_k,t_k,x_k)_k$ and $(M_k,s_k,y_k,)_k$ are orthogonal if
\begin{equation*}
\lim_{k\to+\infty} \left(\left\vert \ln\frac{N_k}{M_k}\right\vert+N_k^2\vert t_k-s_k\vert+N_k\vert x_k-y_k\vert\right)=+\infty.\end{equation*}
Two frames that are not orthogonal are called equivalent.

\item If $\mathcal{O}=(N_k,t_k,x_k)_k$ is a Euclidean frame and if $\phi\in \dot{H}^1(\mathbb{R}^4)$, we define the Euclidean profile associated to $(\phi,\mathcal{O})$ as the sequence $\widetilde{\phi}_{\mathcal{O}_k}$
\begin{equation*}
\widetilde{\phi}_{\mathcal{O}_k}(x):=\Pi_{t_k,x_k}(T_{N_k}\phi).
\end{equation*}

\item
If $\mathcal{O}_k=(1,0,x_k)_k$ is a Scale-1 frame and $\varphi\in H^1(\mathbb{R}\times\mathbb{T}^3)$, we define the associated Scale-1 profile by
\begin{equation*}
\widetilde{\varphi}_{\mathcal{O}_k}(x):=\varphi(x-x_k).
\end{equation*}

\end{enumerate}
\end{definition}

The following lemma summarizes some of the basic properties of profiles associated to equivalent/orthogonal frames. Its proof is very similar to the proof of Lemma 5.4 in \cite{IoPaSt}, and is omitted.

\begin{lemma}(Equivalence of frames)\label{EquivFrames}

(i) If $\mathcal{O}$ and $\mathcal{O}^\prime$ are equivalent Euclidean profiles (resp Scale-1 profiles), then, there exists an isometry of $\dot{H}^1(\mathbb{R}^4)$ (resp of $H^1(\mathbb{R}\times\mathbb{T}^3)$), $T$ such that for any profile $\widetilde{\psi}_{\mathcal{O}^\prime_k}$, up to a subsequence there holds that
\begin{equation}\label{equiv}
\limsup_{k\to+\infty}
\Vert \widetilde{T\psi}_{\mathcal{O}_k}-\widetilde{\psi}_{\mathcal{O}^\prime_k}\Vert_{H^1(\mathbb{R}\times\mathbb{T}^3)}=0.
\end{equation}

(ii) If $\mathcal{O}$ and $\mathcal{O}^\prime$ are orthogonal frames and $\widetilde{\psi}_{\mathcal{O}_k}$, $\widetilde{\varphi}_{\mathcal{O}^\prime_k}$ are corresponding profiles, then, up to a subsequence,
\begin{equation*}
\begin{split}
\lim_{k\to+\infty}\langle \widetilde{\psi}_{\mathcal{O}_k},\widetilde{\varphi}_{\mathcal{O}^\prime_k}\rangle_{H^1\times H^1(\mathbb{R}\times\mathbb{T}^3)}&=0,\\
\lim_{k\to+\infty}\langle |\widetilde{\psi}_{\mathcal{O}_k}|^2,|\widetilde{\varphi}_{\mathcal{O}^\prime_k}|^2\rangle_{L^2\times L^2(\mathbb{R}\times\mathbb{T}^3)}&=0.
\end{split}
\end{equation*}

(iii) If $\mathcal{O}$ is a Euclidean frame and $\widetilde{\psi}_{\mathcal{O}_k}$, $\widetilde{\varphi}_{\mathcal{O}_k}$ are two profiles corresponding to $\mathcal{O}$, then
\begin{equation*}
\begin{split}
&\lim_{k\to+\infty}\left(\Vert\widetilde{\psi}_{\mathcal{O}_k}\Vert_{L^2}+\Vert\widetilde{\varphi}_{\mathcal{O}_k}\Vert_{L^2}\right)=0,\\
&\lim_{k\to+\infty}\langle \widetilde{\psi}_{\mathcal{O}_k},\widetilde{\varphi}_{\mathcal{O}_k}\rangle_{H^1\times H^1(\mathbb{R}\times\mathbb{T}^3)}=\langle \psi,\varphi\rangle_{\dot{H}^1\times\dot{H}^1(\mathbb{R}^4)}.
\end{split}
\end{equation*}

(iv) If $\mathcal{O}$ is a Scale-1 frame and $\widetilde{\psi}_{\mathcal{O}_k}$, $\widetilde{\varphi}_{\mathcal{O}_k}$ are two profiles corresponding to the same frame $\mathcal{O}$, then
\begin{equation*}
\lim_{k\to+\infty}\langle \widetilde{\psi}_{\mathcal{O}_k},\widetilde{\varphi}_{\mathcal{O}_k}\rangle_{H^1\times H^1(\mathbb{R}\times\mathbb{T}^3)}=\langle \psi,\varphi\rangle_{H^1\times H^1(\mathbb{R}\times\mathbb{T}^3)}.
\end{equation*}
\end{lemma}

\begin{definition}\label{absent}
We say that a sequence of functions $\{f_k\}_k\subseteq H^1(\R\times\T^3)$ is absent from a frame $\mathcal{O}$ if, up to a subsequence, for every profile $\psi_{\mathcal{O}_k}$ associated to $\mathcal{O}$,
\begin{equation*}
\int_{\mathbb{R}\times\mathbb{T}^3}\left(f_k\overline{\widetilde{\psi}}_{\mathcal{O}_k}+\nabla f_k\nabla\overline{\widetilde{\psi}}_{\mathcal{O}_k}\right)dx\to0
\end{equation*}
as $k\to+\infty$.
\end{definition}

Note in particular that a profile associated to a frame $\mathcal{O}$ is absent from any frame orthogonal to $\mathcal{O}$.

The following Lemma is the core of this section.

\begin{lemma}\label{ProfileDec1}
Consider $\{f_k\}_k$ a sequence of functions in $H^1(\mathbb{R}\times\mathbb{T}^3)$ satisfying
\begin{equation}\label{FkBounded}
\limsup_{k\to+\infty}\Vert f_k\Vert_{H^1(\mathbb{R}\times\mathbb{T}^3)}\lesssim E
\end{equation}
and fix $\delta>0$. There exists $N\lesssim \delta^{-2}$ profiles $\widetilde{\psi}^\alpha_{\mathcal{O}^\alpha_k}$ associated to pairwise orthogonal frames $\mathcal{O}^\alpha$, $\alpha=1\dots N,$ such that, after extracting a subsequence,
\begin{equation}
f_k=\sum_{1\le \alpha\le N}\widetilde{\psi}^\alpha_{\mathcal{O}^\alpha_k}+R_k
\end{equation}
where $R_k$ is absent from the frames $\mathcal{O}^\alpha$ and is small in the sense that
\begin{equation}\label{smallness}
\sup_{N\ge 1,\,|t|\leq 1,\,x\in\mathbb{R}\times\mathbb{T}^3}N^{-1}\left\vert \left(e^{it\Delta}P_NR_k\right)(x)\right\vert\le\delta.
\end{equation}
Besides, we also have the following orthogonality relations
\begin{equation}\label{Orthogonality}
\begin{split}
\Vert f_k\Vert_{L^2}^2&=\sum_{\beta}\Vert \psi^\beta\Vert_{L^2}^2+\Vert R_k\Vert_{L^2}^2+o_k(1),\\
\Vert \nabla f_k\Vert_{L^2}^2&=\sum_{\beta}\Vert \nabla\psi^\beta\Vert_{L^2}^2+\sum_\gamma\Vert\nabla_{\R^4}\psi^\gamma\Vert_{L^2(\mathbb{R}^4)}^2+\Vert\nabla R_k\Vert_{L^2}^2+o_k(1),
\end{split}
\end{equation}
where $\{\alpha\}=\{\beta\}\cup\{\gamma\}$, $\beta$ corresponding to Scale-1 frames and $\gamma$ to Euclidian frames.
\end{lemma}

\begin{proof}[Proof of Lemma \ref{ProfileDec1}] We will split this proof into several parts.

(I) Extraction of a frame. For a sequence $\{f_k\},$ define the following functional:
\begin{equation*}
\Lambda(\{f_k\})=\limsup_{k\to+\infty}\sup_{N\ge 1,\,|t|\leq 1,\,x\in\mathbb{R}\times\mathbb{T}^3}N^{-1}\left\vert \left(e^{it\Delta}P_Nf_k\right)(x)\right\vert.
\end{equation*}
We claim that if $\Lambda(\{f_k\})\geq \delta$, then there exists a frame $\mathcal{O}$ and an associated profile $\widetilde{\psi}_{\mathcal{O}_k}$ satisfying
\begin{equation}\label{claim11}
\limsup_{k\to+\infty}\Vert\widetilde{\psi}_{\mathcal{O}_k}\Vert_{H^1}\lesssim 1
\end{equation}
and
\begin{equation}\label{claim12}
\limsup_{k\to+\infty}\langle f_k,\widetilde{\psi}_{\mathcal{O}_k}\rangle_{H^1\times H^{1}}\ge\frac{\delta}{2}.
\end{equation}
In addition, if $f_k$ was absent from a family of frames $\mathcal{O}^\alpha$, then $\mathcal{O}$ is orthogonal to all the previous frames $\mathcal{O}^\alpha$.

Let us prove the claim above. By assumption, up to extracting a subsequence, there exists a sequence $(N_k,t_k,x_k)_k$ such that, for all $k$
\begin{equation*}
\begin{split}
\frac{\delta}{2}&\le N_k^{-1}\left\vert \left(e^{it_k\Delta}P_{N_k}f_k\right)(x_k)\right\vert=\left\vert \langle N_k^{-1}e^{it_k\Delta}P_{N_k}f_k,\delta_{x_k}\rangle_{\mathcal{D}\times\mathcal{D}^\prime}\right\vert\\
&\le\left\vert \langle f_k,N_k^{-1}e^{-it_k\Delta}P_{N_k}\delta_{x_k}\rangle_{H^1\times H^{-1}}\right\vert.
\end{split}
\end{equation*}
Now, first assume that $N_k$ remains bounded, then, up to a subsequence, one may assume that $N_k\to N_\infty$ and $t_k\to t_\infty$. In this case, we define $\mathcal{O}=(1,0,x_k)_k$ and
\begin{equation*}
\psi=(1-\Delta)^{-1}N_\infty^{-1}e^{-it_\infty\Delta}P_{N_\infty}\delta_0.
\end{equation*}
Inequality \eqref{claim12} is clear since
\begin{equation*}
(N,t,y)\mapsto N^{-1}(1-\Delta)^{-1}e^{-it\Delta}P_{N}\delta_0
\end{equation*}
is a strongly continuous function. The uniform $H^1$ bound in \eqref{claim11} is clear from the definition.

Now, we assume that $N_k\to+\infty$ and we define the Euclidean frame $\mathcal{O}=(N_k,t_k,x_k)_k$ and the function
\begin{equation*}
\psi=\mathcal{F}^{-1}_{\R^4}\left(|\xi|^{-2}[\eta^4(\xi)-\eta^4(2\xi)]\right)\in H^1(\mathbb{R}^4),\quad\xi=(\xi_1,\xi_2,\xi_3,\xi_4)\in\mathbb{R}^4.
\end{equation*}
Using the Poisson summation formula, it is not hard to prove that
\begin{equation*}
\lim_{k\to+\infty}\|(1-\Delta)T_{N_k}\psi-N_k^{-1}P_{N_k}\delta_0\|_{L^{4/3}(\R\times\T^3)}=0.
\end{equation*}
Thus $\|(1-\Delta)T_{N_k}\psi-N_k^{-1}P_{N_k}\delta_0\|_{H^{-1}(\R\times\T^3)}\to 0$ and we conclude that
\begin{equation*}
\frac{\delta}{2}\lesssim\left\vert\langle f_k,N_k^{-1}e^{-it_k\Delta}P_{N_k}\delta_{x_k}\rangle_{H^1\times H^{-1}}\right\vert\lesssim \left\vert\langle f_k,(1-\Delta)\widetilde{\psi}_{\mathcal{O}_k}\rangle_{H^1\times H^{-1}}\right\vert.
\end{equation*}
Changing $\psi$ by $e^{i\theta}\psi$ to make the scalar product real valued, we obtain \eqref{claim12}.

The last claim about orthogonality of $\mathcal{O}$ with $\mathcal{O}^\alpha$ follows from Lemma \ref{EquivFrames} and the existence of a nonzero scalar product in \eqref{claim12}.

(II) Now that we have selected a frame, we can select the localization of our sequence in this frame as a linear profile.

First, if the frame selected above $\mathcal{O}=(1,0,x_k)_k$ was a Scale-$1$ frame, we consider the sequence
\begin{equation*}
g_k(x):=f_k(x+x_k)=\pi_{-x_k}f_k.
\end{equation*}
This is a bounded sequence in $H^1(\mathbb{R}\times\mathbb{T}^3)$, thus, up to considering a subsequence, we can assume that it converges weakly to $\varphi\in H^1(\mathbb{R}\times\mathbb{T}^3)$. We then define the profile corresponding to $\mathcal{O}$ as $\widetilde{\varphi}_{\mathcal{O}_k}$. By its definition and \eqref{FkBounded}, $\varphi$ has norm smaller than $E$. Besides, we also have that
\begin{equation*}
\begin{split}
\frac{\delta}{2}&\lesssim\lim_{k\to+\infty}\langle f_k,\psi(\cdot-x_k)\rangle_{H^1\times H^1}\lesssim\lim_{k\to+\infty}\langle g_k,\psi\rangle_{H^1\times H^1}=\langle \varphi,\psi\rangle_{H^1\times H^1}.
\end{split}
\end{equation*}
Consequently, we get that
\begin{equation}\label{NonzeroProfile}
\Vert\varphi\Vert_{H^1}\gtrsim\delta.
\end{equation}
We also observe that since $g_k-\varphi$ weakly converges to $0$ in $H^1$, there holds that
\begin{equation}\label{AdditionOfEnergies}
\begin{split}
\Vert A f_k\Vert_{L^2}^2&=\Vert A g_k\Vert_{L^2}^2=\Vert A(g_k-\varphi)\Vert_{L^2}^2-\Vert A\varphi\Vert_{L^2}^2+o_k(1)\\
&=\Vert A(f_k-\varphi_{\mathcal{O},k})\Vert_{L^2}^2-\Vert A\varphi\Vert_{L^2}^2+o_k(1)
\end{split}
\end{equation}
for $A=1$ or $A=\nabla$.

The situation if $\mathcal{O}$ is a Euclidean frame is similar. In this case, for $R>0$ and $k$ large enough, we consider
\begin{equation*}
\varphi^R_k(y)=N_k^{-1}\eta^4(y/R)\left(\Pi_{-t_k,-x_k}f_k\right)(\Psi(y/N_k)),\quad y\in\mathbb{R}^4.
\end{equation*}
This is a sequence of functions bounded in $\dot{H}^1(\mathbb{R}^4)$, uniformly in $R$. We can thus extract a subsequence that converges weakly to a function $\varphi^R\in\dot{H}^1(\mathbb{R}^4)$ satisfying
\begin{equation*}
\Vert \varphi^R\Vert_{\dot{H}^1(\mathbb{R}^4)}\lesssim 1.
\end{equation*}
Hence, extracting a further subsequence, we may assume that $\varphi^R\rightharpoonup\varphi\in\dot{H}^1(\mathbb{R}^4)$ and by uniqueness of the weak limit, we see that for every $R$,
\begin{equation*}
\varphi^R(x)=\eta^4(x/R)\varphi(x).
\end{equation*}
Now, let $\gamma\in C^\infty_0(\mathbb{R}^4)$ be supported in $B(0,R/2)\subset\mathbb{R}^4$ and remark that, for $k$ large enough,
\begin{equation*}
\begin{split}
\langle f_k,\widetilde{\gamma}_{\mathcal{O}_k}\rangle_{H^1\times H^1(\mathbb{R}\times\mathbb{T}^3)}&=\langle \Pi_{-t_k,-x_k}f_k,T_{N_k}\gamma\rangle_{H^1\times H^1(\mathbb{R}\times\mathbb{T}^3)}\\
&=\langle \varphi^R,\gamma\rangle_{\dot{H}^1\times\dot{H}^1(\mathbb{R}^4)}+o_k(1)\\
&=\langle \varphi,\gamma\rangle_{\dot{H}^1\times\dot{H}^1(\mathbb{R}^4)}+o_k(1).
\end{split}
\end{equation*}
Form this and \eqref{claim11}, \eqref{claim12}, we conclude that
\begin{equation}\label{NonzeroProfile2}
\Vert\varphi\Vert_{\dot{H}^1(\mathbb{R}^4)}\gtrsim\delta
\end{equation}
and that
\begin{equation*}
g_k=f_k-\widetilde{\varphi}_{\mathcal{O}_k}
\end{equation*}
is absent from the frame $\mathcal{O}$. Now, similarly to \eqref{AdditionOfEnergies} and using Lemma \ref{EquivFrames}, we see that
\begin{equation*}
\begin{split}
\Vert g_k\Vert_{L^2}^2&=\Vert f_k\Vert_{L^2}+o_k(1)\\
\Vert \nabla g_k\Vert_{L^2}^2&=\Vert \nabla f_k\Vert_{L^2}^2+\Vert\nabla\varphi\Vert_{L^2(\mathbb{R}^4)}^2-2\langle \nabla f_k,\nabla \widetilde{\varphi}_{\mathcal{O}_k}\rangle_{L^2\times L^2}\\
&=\Vert \nabla f_k\Vert_{L^2}^2-\Vert\nabla\varphi\Vert_{L^2(\mathbb{R}^4)}^2+o_k(1).
\end{split}
\end{equation*}

(III) Now, we can finish the proof of Lemma \ref{ProfileDec1}. We define $f^0_k=f_k$. For $\alpha\ge 0$, while $\Lambda(\{f^\alpha_k\}_k)>\delta$ we proceed as follows: applying the two steps above, we get a frame $\mathcal{O}^\alpha$ and an associated profile $\widetilde{\varphi}^\alpha_{\mathcal{O}^\alpha_k}$. We then let
\begin{equation*}
f^{\alpha+1}_k=f^\alpha_k-\widetilde{\varphi}^\alpha_{\mathcal{O}^\alpha_k}.
\end{equation*}
We remark that $f^{\alpha+1}$ is absent from $\mathcal{O}^\alpha$ by definition. By induction and Lemma \ref{EquivFrames}, it is also absent from all the frames $\mathcal{O}^\beta$, $\beta\le\alpha$. Similarly, $\mathcal{O}^\alpha$ is orthogonal to all previous frames $\mathcal{O}^\beta$, $\beta\le \alpha-1$ and by induction, all frames $\mathcal{O}^\beta, \beta\le\alpha$ are pairwise orthogonal. Using \eqref{AdditionOfEnergies} inductively, we also obtain that
\begin{equation*}
\Vert A f_k\Vert_{L^2}^2=\sum_{\beta\le\alpha}\Vert A\widetilde{\varphi}_{\mathcal{O}^\beta_k}\Vert_{L^2}^2+\Vert Af^{\alpha+1}_k\Vert_{L^2}^2+o_k(1).
\end{equation*}
for $A=1$ or $A=\nabla$. In particular using also \eqref{FkBounded} and \eqref{NonzeroProfile}, \eqref{NonzeroProfile2}, we see that this procedure stops after $O(\delta^{-2})$ steps. At this points, it remains to set $R_k=f^{\alpha_{end}+1}_k$ where $\alpha_{end}$ is the last index. This finishes the proof.
\end{proof}

Now, we can deduce the complete profile decomposition

\begin{proposition}\label{PD}
Consider $\{f_k\}_k$ a sequence of functions in $H^1(\mathbb{R}\times\mathbb{T}^3)$ satisfying
\begin{equation}\label{FkBoundedPD}
\limsup_{k\to+\infty}\Vert f_k\Vert_{H^1(\mathbb{R}\times\mathbb{T}^3)}\lesssim E.
\end{equation}
There exists a sequence of profiles $\widetilde{\psi}^\alpha_{\mathcal{O}^\alpha_k}$ associated to pairwise orthogonal frames $\mathcal{O}^\alpha$ such that, after extracting a subsequence, for every $J\ge 0$
\begin{equation}\label{DecompositionPD}
f_k=\sum_{1\le \alpha\le J}\widetilde{\psi}^\alpha_{\mathcal{O}^\alpha_k}+R_k^J
\end{equation}
where $R_k^J$ is absent from the frames $\mathcal{O}^\alpha$, $\alpha\le J$ and is small in the sense that
\begin{equation}\label{smallnessPD}
\limsup_{J\to+\infty}\limsup_{k\to+\infty}\big[\sup_{N\ge 1,\,|t|\leq 1,\,x\in\mathbb{R}\times\mathbb{T}^3}N^{-1}\left\vert \left(e^{it\Delta}P_NR_k^J\right)(x)\right\vert\big]=0.
\end{equation}
Besides, we also have the following orthogonality relations
\begin{equation}\label{OrthogonalityPD}
\begin{split}
&\Vert f_k\Vert_{L^2}^2=\sum_{\beta}\Vert \psi^\beta\Vert_{L^2}^2+\Vert R_k^J\Vert_{L^2}^2+o_k(1),\\
&\Vert \nabla f_k\Vert_{L^2}^2=\sum_{\beta\le J}\Vert \nabla\psi^\beta\Vert_{L^2}^2+\sum_{\gamma\le J}\Vert\nabla_{\R^4}\psi^\gamma\Vert_{L^2(\mathbb{R}^4)}^2+\Vert\nabla R_k^J\Vert_{L^2}^2+o_k(1),\\
&\lim_{J\to+\infty}\limsup_{k\to+\infty}\left\vert\Vert f_k\Vert_{L^4}^4-\sum_{\alpha\le J}\Vert\widetilde{\varphi}^\alpha_{\mathcal{O}^\alpha_k}\Vert_{L^4}^4\right\vert=0,
\end{split}
\end{equation}
where $\{\alpha\}=\{\beta\}\cup\{\gamma\}
$, $\beta$ corresponding to Scale-1 frames and $\gamma$ to Euclidean frames.
\end{proposition}

\begin{proof}[Proof of Proposition \ref{PD}]
We let $\delta_l=2^{-l}$ and we apply inductively Lemma \ref{ProfileDec1} to get the sequence of frames and profiles $\mathcal{O}^\alpha$, $\widetilde{\varphi}^\alpha_{\mathcal{O}^\alpha_k}$. We only need to prove the last equality in \eqref{OrthogonalityPD}. Using Lemma \ref{EquivFrames} (ii), we see that for $\alpha\ne\beta$,
\begin{equation*}
\langle |\widetilde{\varphi}^\alpha_{\mathcal{O}^\alpha_k}|,|\widetilde{\varphi}^\beta_{\mathcal{O}^\beta_k}|^3\rangle_{L^2\times L^2}
\le \left[\langle |\widetilde{\varphi}^\alpha_{\mathcal{O}^\alpha_k}|^2,|\widetilde{\varphi}^\beta_{\mathcal{O}^\beta_k}|^2\rangle_{L^2\times L^2}\right]^\frac{1}{2}\Vert\widetilde{\varphi}^\be_{\mathcal{O}^\beta_k}\Vert_{L^4}^2\le o_k(1).
\end{equation*}
Therefore, for any fixed $J$, there holds that
\begin{equation*}
\Vert \sum_{\alpha\le J}\widetilde{\varphi}^\alpha_{\mathcal{O}^\alpha_k}\Vert_{L^4}^4-\sum_{\alpha\le J}\Vert\widetilde{\varphi}^\alpha_{\mathcal{O}^\alpha_k}\Vert_{L^4}^4\lesssim_J o_k(1).
\end{equation*}
On the other hand, using that $f_k$ is bounded and $R_k^J$ are bounded in $L^4$ uniformly in $J$,
\begin{equation*}
\limsup_{k\to+\infty}\left\vert \Vert f_k\Vert_{L^4}^4-\Vert\sum_{\alpha\le J}\widetilde{\varphi}^\alpha_{\mathcal{O}^\alpha_k}\Vert_{L^4}^4\right\vert \lesssim_E \limsup_{k\to+\infty}\Vert R^J_k\Vert_{L^4}.
\end{equation*}
Since $f_k$ is bounded in $L^4$, combining
\eqref{PrecSob}, \eqref{smallnessPD} and \eqref{OrthogonalityPD}, we see that
\begin{equation*}
\lim_{J\to+\infty}\limsup_{k\to+\infty}\Vert R^J_k\Vert_{L^4}=0.
\end{equation*} This ends the proof.
\end{proof}

\section{Proof of the main theorem}\label{proofthm}

From Lemma \ref{BlowUpCriterion}, we see that to prove Theorem \ref{Main1}, it suffices to prove that solutions remain bounded in $Z$ on intervals of length at most $1$.
To obtain this, we induct on the level
\begin{equation*}
L(u)=M(u)+E(u).
\end{equation*}

Define
\begin{equation*}
\Lambda(L)=\sup\{\Vert u\Vert_{Z(I)},E(u)+M(u)\le L,\vert I\vert\le 1\}
\end{equation*}
where the supremum is taken over all strong solutions of
\eqref{eq1} of level less than or equal to $L$ and all intervals $I$ of length $\vert I\vert\le 1$. We also define
\begin{equation}\label{Lmax}
L_{max}=\sup\{L: \Lambda(L)<+\infty\}.
\end{equation}

In subsection \ref{ProofOfKP}, we prove the following key proposition:
\begin{proposition}\label{KeyProp}
Assume that $L_{max}<\infty$. Let $u_k$ be a sequence of strong solutions of \eqref{eq1} on intervals $I_k$ and $t_k$ a sequence of times in $I_k$ satisfying
\begin{equation}\label{CondForComp}
E(u_k)+M(u_k)\to L_{max},\quad \min(\Vert u_k\Vert_{Z(T_k,t_k)},\Vert u_k\Vert_{Z(t_k,T^k)})\to+\infty
\end{equation}
for some $T_k\le t_k\le T^k$, $T^k-T_k\le1$.
Then, there exists a sequence $x_k\in\mathbb{R}\times\mathbb{T}^3$ and $v\in H^1(\mathbb{R}\times\mathbb{T}^3)$ such that, up to a subsequence,
\begin{equation}\label{Compactness}
u_k(\cdot-x_k,t_k)\to v
\end{equation}
strongly in $H^1$.
\end{proposition}
Assuming this proposition, we will prove the following which implies Theorem \ref{Main1}.
\begin{theorem}\label{LmaxThem}
$L_{max}=+\infty$. In particular solutions of \eqref{eq1} exist globally on $\mathbb{R}\times\mathbb{T}^3$ and Theorem \ref{Main1} is proved.
\end{theorem}

\begin{proof}[Proof of Theorem \ref{LmaxThem}] We assume that $L_{max}<+\infty$ and we will derive a contradiction.
We proceed in several steps.

(I) Existence of a critical element. By definition, we may find $u_k$ and intervals $J_k=(T_k,T^k)$ of length less than $1$ such that
\begin{equation*}
E(u_k)+M(u_k)\to L_{max},\quad \Vert u_k\Vert_{Z(J_k)}\ge 4k.
\end{equation*}
We can then clearly find $t_k\in J_k$ such that \eqref{CondForComp} is satisfied.
Now $u_k,t_k$ satisfy the hypothesis of Proposition \ref{KeyProp}. By extracting a subsequence, we may assume that $T^k-t_k\to b$ and $t_k-T_k\to a$. We can thus find a sequence $x_k$ such that \eqref{Compactness} holds. We then define $U$ to be the solution of \eqref{eq1} with initial data $U(0)=v$.

By strong convergence, there holds that $E(U)+M(U)=E(v)+L(v)=L_{max}$. Besides, let $I=(-T_-,T^+)$ be the maximal time of existence of $U$. We claim that $I$ is bounded. Indeed, $I\subset (-a,b)$, for if, for example, $(0,b+2\delta)\subset I$, then
\begin{equation*}
\Vert U\Vert_{Z(0,b+\delta)}=M<+\infty
\end{equation*}
and applying Lemma \ref{Stabprop} with $\tilde{u}(x,t)=v$, $u(x,t)=u_k(x-x_k,t_k+t)$, we get from \eqref{output} that, for all $k$ large enough,  $u_k$ is bounded in $Z(t_k,T^k)$, which contradicts \eqref{CondForComp}. From Lemma \ref{BlowUpCriterion}, we thus conclude that
$\Vert U\Vert_{Z(0,T^+)}=+\infty$.

(II) Nonexistence of a critical element. Now we will derive a contradiction. Indeed, take $t_k<T^+$, $t_k\to T^+$. By what we saw before, there holds that
\begin{equation}\label{Contrad}
\Vert U\Vert_{Z(0,t_k)}\to+\infty,\quad \Vert U\Vert_{Z(t_k,T^+)}=+\infty.
\end{equation}
Besides, $E(U)+M(U)=L_{max}$. Consequently, we may apply Proposition \ref{KeyProp} to get that there exists $x_k$ such that $U(\cdot-x_k,t_k)\to v_\infty$. Now, let $V$ be the solution of \eqref{eq1} with initial data $v_\infty$. By local wellposedness, there exists $\delta>0$ such that
\begin{equation}\label{ConsOfLWP}
\Vert V\Vert_{Z(-2\delta,2\delta)}\le 1.
\end{equation}
Now, take $k$ sufficiently large so that $t_k\ge T^+-\delta$. Applying the stability proposition, we see that
\begin{equation*}
\Vert U\Vert_{Z(t_k,T^+)}\le \Vert V\Vert_{Z(0,\delta)}+o_k(1)\le 2
\end{equation*}
for $k$ sufficiently large. This contradicts \eqref{Contrad} and finishes the proof.
\end{proof}

\subsection{Proof of Proposition \ref{KeyProp}}\label{ProofOfKP} Without loss of generality, we may assume that $t_k=0$. We apply Proposition \ref{PD} to the sequence $\{u_k(0)\}_k$ which is indeed bounded in $H^1(\mathbb{R}\times\mathbb{T}^3)$. This way we obtain, for all $J$,
\begin{equation*}
u_k(0)=\sum_{1\le\alpha\le J}\widetilde{\psi}_{\mathcal{O}^\alpha_k}^\alpha+R^J_k.
\end{equation*}

We first consider the remainder. Using H\"older's inequality and Strichartz estimates \eqref{no0.001}, we observe that
\begin{equation}\label{Rem}
\begin{split}
\Vert e^{it\Delta}R^J_k\Vert_{Z(-1,1)}^4
&=\sum_NN^2\Vert P_Ne^{it\Delta}R^J_k\Vert_{L^4_{x,t}(\mathbb{R}\times\mathbb{T}^3\times[-1,1])}^4\\
&\lesssim \sum_N\left(N^{-1}\Vert P_Ne^{it\Delta}R^J_k\Vert_{L^\infty_{x,t}}\right)^\frac{1}{5}\left(N^\frac{11}{19}\Vert P_Ne^{it\Delta}R^J_k\Vert_{L^\frac{19}{5}_{x,t}}\right)^\frac{19}{5}\\
&\lesssim \left(\sup_NN^{-1}\Vert P_Ne^{it\Delta}R^J_k\Vert_{L^\infty_{x,t}}\right)^\frac{1}{5}\sum_N \left(N\Vert P_NR^J_k\Vert_{L^2}\right)^\frac{19}{5}\\
&\lesssim \left(\sup_NN^{-1}\Vert P_Ne^{it\Delta}R^J_k\Vert_{L^\infty_{x,t}}\right)^\frac{1}{5}\Vert R^J_k\Vert_{H^1}^\frac{19}{5}.
\end{split}
\end{equation}

\medskip

{\bf Case I}: There are no profiles. Then, taking $J$ sufficiently large, we see from \eqref{Rem} that
\begin{equation*}
\begin{split}
\Vert e^{it\Delta}u_k(0)\Vert_{Z(-1,1)}&=\Vert e^{it\Delta}R^J_k\Vert_{Z(-1,1)}\le \delta_0/2
\end{split}
\end{equation*}
for $k$ sufficiently large,
where $\delta_0=\delta_0(L_{max})$ is given in Proposition \ref{LocTheory}. Then we see from Proposition \ref{LocTheory} that $u_k$ can be extended on $(-1,1)$ and that
\begin{equation*}
\lim_{k\to+\infty}\Vert u_k\Vert_{Z(-1,1)}\le \delta_0
\end{equation*}
which contradicts \eqref{CondForComp}.

\medskip

Hence, we see that there exists at least one profile. Using Lemma \ref{EquivFrames} and passing to a subsequence, we may renormalize every Euclidean profile, that is, up to passing to an equivalent profile, we may assume that for every Euclidean frame $\mathcal{O}^\alpha$, $\mathcal{O}^\alpha\in\widetilde{\mathcal{F}}_e$.
Besides, using Lemma \ref{EquivFrames} and passing to a subsequence once again, we may assume that for every $\alpha\ne\beta$,
either $N^\alpha_k/N^\beta_k+N^\beta_k/N^\alpha_k\to+\infty$ as $k\to+\infty$ or $N^\alpha_k=N^\beta_k$ for all $k$ and in this case, either $t^\alpha_k=t^\beta_k$ as $k\to+\infty$ or $(N^\alpha_k)^2\vert t^\alpha_k-t^\beta_k\vert \to+\infty$ as $k\to+\infty$.

Now for every linear profile $\widetilde{\psi}^\alpha_{\mathcal{O}^\alpha_k}$, we define the associated nonlinear profile $U^\alpha_k$ as the maximal solution of \eqref{eq1} with initial data $U^\alpha_k(0)=\widetilde{\psi}^\alpha_{\mathcal{O}^\alpha_k}$. We can give a more precise description of each nonlinear profile.
\begin{enumerate}
\item If $\mathcal{O}^\alpha=(1,0,x_k^\alpha)_k$ is a Scale-1 frame, then letting $W^\alpha$ be the maximal strong solution with initial data $\psi^\alpha\in H^1(\mathbb{R}\times\mathbb{T}^3)$, we have that
\begin{equation*}
U^\alpha_k(t,x)=W^\alpha(t,x-x^\alpha_k).
\end{equation*}
\item If $\mathcal{O}^\alpha$ is a Euclidean frame, this is given by Lemma \ref{GEForEP}.
\end{enumerate}

From \eqref{OrthogonalityPD} we see that
\begin{equation}\label{SumOfL}
\begin{split}
&L(\alpha):=\lim_{k\to+\infty}\left(E(\widetilde{\psi}^\alpha_{\mathcal{O}^\alpha_k})+M(\widetilde{\psi}^\alpha_{\mathcal{O}^\alpha_k})\right)\in(0,L_{max}],\\
&\lim_{J\to+\infty}\big[\sum_{1\le\alpha\le J}L(\alpha)+\lim_{k\to+\infty}L(R_k^J)\big]\le L_{max}.
\end{split}
\end{equation}
The numbers $L(\alpha)$ and $\lim_{k\to+\infty}L(R_k^J)$ are all well defined up to taking a subsequence. Up to relabelling the profiles, we can assume that for all $\alpha$, $L(1)\ge L(\alpha)$.

\medskip

{\bf Case IIa:} $L(1)=L_{max}$ and there is only one Euclidean profile, that is
\begin{equation*}
u_k(0)=\widetilde{\psi}_{\mathcal{O}_k}+o_k(1)
\end{equation*}
in $H^1$ (see \eqref{SumOfL}), where $\mathcal{O}$ is a Euclidean frame. In this case, since from \eqref{ControlOnZNormForEP} the corresponding nonlinear profile $U_k$ satisfies
\begin{equation*}
\Vert U_k\Vert_{Z(-1,1)}\lesssim_{E_{\mathbb{R}^4}(\psi)} 1\quad\text{and}\quad\lim_{k\to +\infty}\Vert U_k(0)-u_k(0)\Vert_{H^1}\to 0
\end{equation*}
we may use Proposition \ref{Stabprop} to deduce that
\begin{equation*}
\Vert u_k\Vert_{Z(-1,1)}\lesssim \Vert u_k\Vert_{X^1(-1,1)}\lesssim_{L_{max}} 1
\end{equation*}
which contradicts \eqref{CondForComp}.

\medskip

{\bf Case IIb:} $L(1)=L_{max}$ and, using again \eqref{SumOfL} we have that
\begin{equation*}
u_k(0)=\widetilde{\psi}_{\mathcal{O}_k}+o_k(1)\quad\text{ in }H^1,
\end{equation*}
where $\mathcal{O}$ is a Scale-1 frame. This is precisely conclusion \eqref{Compactness}.

\medskip

{\bf Case III:}
$L(1)<L_{max}$. Then, we have that there exists $\eta>0$ such that for all $\alpha$, $L(\alpha)<L_{max}-\eta$. In this case, we have that all nonlinear profiles are global and satisfy, for any $k,\alpha$ (after extracting a subsequence)
\begin{equation*}
\Vert U^\alpha_k\Vert_{Z(-2,2)}\le 2\Lambda(L_{max}-\eta/2)\lesssim 1,
\end{equation*}
where from now on all the implicit constants are allowed to depend on $\Lambda(L_{max}-\eta/2)$. Using Proposition \ref{Stabprop} it follows that
\begin{equation}\label{BddX1}
\Vert U^\alpha_k\Vert_{X^1(-1,1)}\lesssim 1.
\end{equation}

For $J,k\geq 1$ we define
\begin{equation*}
U^J_{prof,k}:=\sum_{\al=1}^J U^\al_k.
\end{equation*}
We show first that there is a constant $Q\lesssim 1$ such that
\begin{equation}\label{bi1}
\Vert U^J_{prof,k}\Vert_{X^1(-1,1)}^2+\sum_{\al=1}^J\|U^\al_k\|_{X^1(-1,1)}^2\leq Q^2,
\end{equation}
uniformly in $J$, for all $k$ sufficiently large. Indeed, a simple fixed point argument as in section \ref{localwp} shows that there exists $\delta_0>0$ such that if
\begin{equation*}
\Vert \phi\Vert_{H^1(\mathbb{R}\times\mathbb{T}^3)}\le\delta_0
\end{equation*}
then the unique strong solution of \eqref{eq1} with initial data $\phi$ is global and satisfies
\begin{equation}\label{SmalldataCCL}
\begin{split}
\Vert u\Vert_{X^1(-2,2)}&\le 2\delta\quad\text{and}\quad\Vert u-e^{it\Delta}\phi\Vert_{X^1(-2,2)}\lesssim \Vert\phi\Vert_{H^1(\mathbb{R}\times\mathbb{T}^3)}^3.
\end{split}
\end{equation}
From \eqref{SumOfL}, we know that there are only finitely many profiles such that $L(\alpha)\geq\delta_0/2$. Without loss of generality, we may assume that for all $\alpha\ge A$, $L(\alpha)\leq\delta_0$. Using \eqref{OrthogonalityPD}, \eqref{BddX1}, and \eqref{SmalldataCCL} we then see that
\begin{equation*}
\begin{split}
&\Vert U^J_{prof,k}\Vert_{X^1(-1,1)}=\Vert \sum_{1\le\alpha\le J}U^\alpha_k\Vert_{X^1(-1,1)}\\
&\le\sum_{1\le\alpha\le A}\Vert U^\alpha_k\Vert_{X^1(-1,1)}+\Vert \sum_{A\le\alpha\le J}(U^\alpha_k-e^{it\Delta}U^\alpha_k(0))\Vert_{X^1(-1,1)}+\Vert e^{it\Delta}\sum_{A\le\alpha\le J}U^\alpha_k(0)\Vert_{X^1(-1,1)}\\
&\lesssim A+\sum_{A\le\alpha\le J}L(\alpha)^\frac{3}{2}+\Vert\sum_{A\le\alpha\le J}U^\alpha_k(0)\Vert_{H^1}\\
&\lesssim 1.
\end{split}
\end{equation*}
The bound on $\sum_{\al=1}^J\|U^\al_k\|_{X^1(-1,1)}^2$ is similar (in fact easier), which gives \eqref{bi1}.

We define $F(z)=\vert z\vert^2z$ and
\begin{equation*}
F^\prime(G)u:=2G\overline{G}u+G^2\overline{u}.
\end{equation*}
For fixed $B$ and $J$, we define $g^{B,J}_k$ to be the solution of the initial value problem
\begin{equation}\label{bi2}
\left(i\partial_t+\Delta\right)g-F^\prime(U^B_{prof,k})g=0,\qquad g(0)=R_k^J.
\end{equation}
Using \eqref{bi1}, dividing the interval $(-1,1)$ into finitely many subintervals, and using the stability theory in section \ref{localwp} (in particular Lemma \ref{NLEst2}) in each subinterval, it follows that there is a constant $Q^\prime\ge Q$ independent of $J$ and $B$ such that, for all $k\ge k_0(J,B)$ the solution $g^{B,J}_k$ is well-defined on $(-1,1)$ and
\begin{equation}\label{BoundOng}
\Vert g^{B,J}_k\Vert_{X^1(-1,1)}\le Q^\prime.
\end{equation}

For a fixed $A$ to be defined soon, we will consider the approximate solution
\begin{equation}\label{DefOfAppSol}
U^{app,J}_k:=U^A_{prof,k}+g^{A,J}_k+U^{J,A}_{prof,k}\qquad\text{ where }\qquad U^{J,A}_{prof,k}:=\sum_{\al=A+1}^JU^\al_k.
\end{equation}
We have from \eqref{bi1} and \eqref{BoundOng} that for every choice of $A$, $J$, $A\le J$, there holds that
\begin{equation*}
U^{app,J}_k(0)=u_k(0)\qquad\text{ and }\qquad\Vert U^{app,J}_k\Vert_{X^1(-1,1)}\le 4Q',
\end{equation*}
for all $k\ge k_0(J)$ large enough. The stability Proposition \ref{Stabprop} with $M\approx Q'$ provides $\epsilon_1=\epsilon_1(M)\leq 1/(Q'+K)$, $K$ sufficiently large. We now choose $A$ such that
\begin{equation}\label{bi3}
\Vert U^{J,A}_{prof,k}\Vert_{X^1(-1,1)}^2+\sum_{A<\al\leq J}\|U_k^\al\|_{X^1(-1,1)}^2\le\epsilon_1^{10}\qquad\text{ for any }J\geq A\text{ and }k\text{ sufficiently large}.
\end{equation}
This is possible, by an argument similar to the argument in the proof of \eqref{bi1}. This choice fixes the constant $A$.

To continue, we need two lemmas which are proved in the next subsections.

\begin{lemma}\label{lemm1}
Assume $\al\neq\be\in\{1,\ldots,A\}$. For every sequence of functions $g_k$ which is uniformly bounded in $X^1$, we have
\begin{equation*}
\limsup_{k\to+\infty}\Vert \tilde{U}^\alpha_k \tilde{U}^\beta_kg_k\Vert_{N(-1,1)}=0\qquad\text{ where }\tilde{U}^\alpha_k\in\{U^\al_k,\overline{U^\al_k}\},\,\,\tilde{U}^\beta_k\in\{U^\be_k,\overline{U^\be_k}\}.
\end{equation*}
\end{lemma}

\begin{lemma}\label{lemm2}
With the notation above
\begin{equation*}
\limsup_{J\to+\infty}\limsup_{k\to+\infty}\|g_k^{A,J}\|_{Z(-1,1)}=0.
\end{equation*}
\end{lemma}

Assuming these lemmas we can complete the proof of Proposition \ref{KeyProp}. Indeed, with the choice of approximate solution as in \eqref{DefOfAppSol}, we compute
\begin{equation*}
\begin{split}
e^J_k&=\left(i\partial_t+\Delta\right)U^{app,J}_k-F(U^{app,J}_k)\\
&=-F(U^{A}_{prof,k}+g^{A,J}_k+U^{J,A}_{prof,k})+\sum_{1\le\alpha\le J}F(U^\alpha_k)+F^\prime(U^{A}_{prof,k})g^{A,J}_k\\
&=-F(U^{A}_{prof,k}+g^{A,J}_k+U^{J,A}_{prof,k})+F(U^A_{prof,k}+g^{A,J}_k)+F(U^{J,A}_{prof,k})\\
&-F(U^A_{prof,k}+g^{A,J}_k)+F(U^A_{prof,k})+F^\prime(U^{A}_{prof,k})g^{A,J}_k\\
&-F(U^A_{prof,k})+\sum_{1\le\alpha\le A}F(U^\alpha_k)\\
&-F(U^{J,A}_{prof,k})+\sum_{A<\alpha\le J}F(U^\alpha_k)\\
&=I+II+III+IV.
\end{split}
\end{equation*}

We first see from Lemma \eqref{NLEst2} and the bounds \eqref{bi1}, \eqref{BoundOng}, and \eqref{bi3} that
\begin{equation*}
\begin{split}
\Vert I\Vert_{N(-1,1)}&\lesssim \left(\Vert U^{A}_{prof,k}+g^{A,J}_k\Vert_{X^1(-1,1)}+\Vert U^{J,A}_{prof,k}\Vert_{X^1(-1,1)}\right)^2\Vert U^{J,A}_{prof,k}\Vert_{X^1(-1,1)}\le \epsilon_1/10,
\end{split}
\end{equation*}
for $k$ large enough. Now we estimate $II$ as follows\footnote{The choice of $g^{A,J}_k$ satisfying \eqref{bi2} is precisely to achieve a cancellation in this term, i.e. to avoid estimating products such as $(U^A_{prof,k})^2\overline{g^{A,J}_k}$ and $|U^A_{prof,k}|^2g^{A,J}_k$. Such products, which are linear in $g^{A,J}_k$, cannot be made small in $N(-1,1)$ by letting $J,k\to+\infty$.}
\begin{equation*}
\begin{split}
\Vert II\Vert_{N(-1,1)}&\lesssim \Vert \overline{U^A_{prof,k}}(g^{A,J}_k)^2\Vert_{N(-1,1)}+\Vert U^A_{prof,k}g^{A,J}_k\overline{g^{A,J}_k}\Vert_{N(-1,1)}+\Vert (g^{A,J}_k)^2\overline{g^{A,J}_k}\Vert_{N(-1,1)}\\
&\lesssim \Vert U^A_{prof,k}\Vert_{X^1(-1,1)}\Vert g^{A,J}_k\Vert_{X^1(-1,1)}\Vert g^{A,J}_k\Vert_{Z^\prime(-1,1)}+\Vert g^{A,J}_k\Vert_{X^1(-1,1)}\Vert g^{A,J}_k\Vert_{Z^\prime(-1,1)}^2\\
&\lesssim (Q^\prime)^2\Vert g^{A,J}_k\Vert_{Z^\prime(-1,1)}.
\end{split}
\end{equation*}
We now turn to $III$, which we estimate using Lemma \ref{lemm1}
\begin{equation*}
\Vert III\Vert_{N(-1,1)}\lesssim_A\sum_{1\le\alpha\ne\beta\le A,1\le\gamma\le A}\Vert \widetilde{U}^\alpha_k\widetilde{U}^\beta_k \widetilde{U}^\gamma_k\Vert_{N(-1,1)}\le\epsilon_1/10
\end{equation*}
provided that $k$ is sufficiently large, where $\widetilde{U}^\al_k\in\{U^\al_k,\overline{U^\al_k}\}$. We estimate $IV$ using \eqref{bi3}
\begin{equation*}
\begin{split}
\Vert IV\Vert_{N(-1,1)}&\lesssim \Vert F(U^{J,A}_{prof,k})\Vert_{N(-1,1)}+\sum_{A<\alpha\le J}\Vert F(U^\alpha_k)\Vert_{N(-1,1)}\\
&\lesssim \Vert U^{J,A}_{prof,k}\Vert_{X^1(-1,1)}^3+\sum_{A<\alpha\le J}\Vert U^\alpha_k\Vert_{X^1(-1,1)}^3\\
&\le\epsilon_1/10.
\end{split}
\end{equation*}

Adding these estimates we obtain
\begin{equation*}
\Vert e^J_k\Vert_{N(-1,1)}\le \epsilon_1/2+C(Q^\prime)^2\Vert g^{A,J}_k\Vert_{Z^\prime(-1,1)}.
\end{equation*}
Using Lemma \ref{lemm2}, we can find $J$ sufficiently large such that for all $k\ge k_0(J)$ we have
\begin{equation*}
C(Q^\prime)^2\Vert g^{A,J}_k\Vert_{Z^\prime(-1,1)}\le\epsilon_1/10.
\end{equation*}
For $k$ sufficiently large this shows that
\begin{equation*}
\Vert e^J_k\Vert_{N(-1,1)}<\epsilon_1.
\end{equation*}
We can now apply Proposition \ref{Stabprop} to conclude that $u_k$ is defined on $(-1,1)$ and that
\begin{equation*}
\limsup_{k\to+\infty}\Vert u_k\Vert_{Z(-1,1)}<+\infty.
\end{equation*}
But this contradicts \eqref{CondForComp}.

\subsection{Proof of Lemma \ref{lemm1}} We may assume $\|g_k\|_{X^1(-1,1)}\leq 1$ for any $k$. For simplicity of notation we will also assume that $\tilde{U}_k^\al=U_k^\al$ and $\tilde{U}_k^\be=U_k^\be$. All of our estimates use only the bounds in Lemma \ref{NLEst2}, the spaces $X^1,Z',N$ and spacetime norms such as $L^p_tL^q_x$, which are not sensitive to complex conjugation.

Let $I=(-1,1)$. If $\mathcal{O}^\alpha$ and $\mathcal{O}^\beta$ are Scale-1 frames, then the proof is easy.
Indeed, given $\theta>0$ there is $R_\theta=R_{\theta,\psi^\alpha,\psi^\beta}\gg 1$ such that we may decompose, for $\gamma\in\{\al,\be\}$,
\begin{equation}\label{deco1}
\begin{split}
&U^\gamma_k=\pi_{x_k}V^{\gamma,\theta}+\pi_{x_k}\rho^{\gamma,\theta}=\omega^{\gamma,\theta}_k+\rho^{\gamma,\theta}_k,\\
&\|\rho^{\gamma,\theta}\|_{X^1(I)}+\|(1-P_{\leq R_\theta})\omega^{\gamma,\theta}_k\|_{X^1(I)}\leq\theta,\qquad \|V^{\gamma,\theta}\|_{X^1(I)}\lesssim 1,\\
&|D_x^mV^{\gamma,\theta}|\leq R_\theta\mathbf{1}_{\mathcal{S}_k^{\gamma,\theta}},\qquad 0\leq|m|\leq 6,
\end{split}
\end{equation}
where
\begin{equation*}
S^{\gamma,\theta}_k:=\{(x,t)\in\R\times\T^3\times(-1,1):|x-x^\gamma_k|\leq R_\theta\}.
\end{equation*}
From Lemma \ref{NLEst2}, we then have that
\begin{equation*}
\Vert \rho^{\alpha,\theta}_kU^\beta_kg_k\Vert_{N(I)}+\Vert V^{\alpha,\theta}_k\rho^{\beta,\theta}_kg_k\Vert_{N(I)}\lesssim \theta.
\end{equation*}
Independently, since $\vert x^\alpha_k-x^\beta_k\vert\to+\infty$, it follows from the description \eqref{deco1} that
\begin{equation*}
V^{\alpha,\theta}_kV^{\beta,\theta}_kg_k\equiv 0
\end{equation*}
for $k$ sufficiently large. This completes the proof of the lemma in this case.

Now we assume that $\mathcal{O}^\gamma$ is a Euclidean frame, $\gamma\in\{\al,\be\}$. Using Lemma \ref{GEForEP}, and then applying also Lemma \ref{step1} and Lemma \ref{step2}, it follows that for any $\theta>0$ there are $T_\theta=T_{\theta,\psi^\alpha,\psi^\beta}\gg 1$ and $R_\theta=R_{\theta,\psi^\alpha,\psi^\beta}\gg T_\theta$ such that for all $k$ sufficiently large we can decompose
\begin{equation}\label{DecNP}
\begin{split}
&U^\gamma_k=V^{\gamma,\theta}_k+\rho^{\gamma,\theta}_k=\omega^{\gamma,\theta,-\infty}_k+\omega^{\gamma,\theta}_k+\omega^{\gamma,\theta,+\infty}_k+\rho^{\gamma,\theta}_k,\\
&\Vert \omega^{\gamma,\theta,\pm\infty}_k\Vert_{Z^\prime(I)}+\Vert\rho^{\gamma,\theta}_k\Vert_{X^1(I)}+\|(1-P_{\leq R_\theta N_k^\gamma}+P_{\leq R_\theta^{-1} N_k^\gamma})\omega_k^{\gamma,\theta}\|_{X^1(I)}\le\theta,\\
&\Vert \omega^{\gamma,\theta,\pm\infty}_k\Vert_{X^1(I)}+\Vert\omega^{\gamma,\theta}_k\Vert_{X^1(I)}\lesssim 1,\\
&\vert D_x^m\omega_k^{\gamma,\theta}\vert\le R_\theta (N^\gamma_k)^{1+|m|}\mathbf{1}_{\mathcal{S}^\gamma_\theta},\qquad 0\leq|m|\leq 6,\\
&\omega_k^{\gamma,\theta,\pm\infty}=\mathbf{1}_{\{\pm(t-t^\gamma_k)\ge T_{\theta}(N^\gamma_k)^{-2}\}}[e^{i(t-t_k^\gamma)\Delta}\pi_{x_k^\gamma}T_{N_k^\gamma}(\phi^{\gamma,\theta,\pm\infty})],\qquad\|\phi^{\gamma,\theta,\pm\infty}\|_{\dot{H}^1(\R^4)}\lesssim 1,\\
&\phi^{\gamma,\theta,\pm\infty}=P_{\le R_\theta}(\phi^{\gamma,\theta,\pm\infty}),\qquad \|\phi^{\gamma,\theta,\pm\infty}\|_{L^1(\R^4)}\leq R_\theta,
\end{split}
\end{equation}
where
\begin{equation*}
\mathcal{S}^{\gamma,\theta}_k:=\{(x,t)\in\R\times\T^3\times(-1,1):\vert t-t^\gamma_k\vert\le T_{\theta}(N^\gamma_k)^{-2},\,\vert x-x^\gamma_k\vert\le R_{\theta}(N^\gamma_k)^{-1}\}.
\end{equation*}

Assuming that $\mathcal{O}^\alpha$ is a Euclidean frame, we fix $\theta>0$ and decompose $U^\alpha_k$ as in \eqref{DecNP}; similarly we decompose $U^\beta_k$ as in \eqref{DecNP} if $\mathcal{O}^\beta$ is a Euclidean frame and decompose $U^\beta_k$ as in \eqref{deco1} if $\mathcal{O}^\beta$ is a Scale-1 frame. Notice that the two decompositions agree if we set $\omega^{\beta,\theta}_k=V_k^{\beta,\theta}$ and $\omega^{\beta,\theta,\pm\infty}_k=0$ if $\mathcal{O}^\beta$ is a Scale-1 frame. For simplicity of notation, we omit from now on the dependence in $\theta$.

Any term involving $\rho^\alpha_k$ can be dealt with directly, using Lemma \ref{NLEst2}, for example
\begin{equation*}
\begin{split}
\Vert \rho^{\alpha}_kU^\beta_kg_k\Vert_{N(I)}&\lesssim \Vert \rho^\alpha_k\Vert_{X^1(I)}\Vert U^\beta_k\Vert_{X^1(I)}\Vert g_k\Vert_{X^1(I)}\lesssim \theta.
\end{split}
\end{equation*}
Similarly, we observe that since $\omega^{\alpha,\pm\infty}_k$, $\omega^{\beta,\pm\infty}_k$ are small in $Z^\prime$, for $k$ large enough
\begin{equation*}
\begin{split}
\Vert &\omega^{\alpha,\pm\infty}_k\omega^{\beta,\pm\infty}_kg_k\Vert_{N(I)}\\
&\lesssim \left(\Vert\omega^{\alpha,\pm\infty}_k\Vert_{Z^\prime(I)}+\Vert\omega^{\beta,\pm\infty}_k\Vert_{Z^\prime(I)}\right)\left(\Vert \omega^{\alpha,\pm\infty}_k\Vert_{X^1(I)}+\Vert \omega^{\beta,\pm\infty}_k\Vert_{X^1(I)}\right)\Vert g_k\Vert_{X^1(I)}\\
&\lesssim\theta.
\end{split}
\end{equation*}
To estimate the remaining terms, we may assume from now on that $N^\al_k\geq N^\be_k$.

We first estimate $\omega^\alpha_k\omega^\beta_kg_k$. In view of the support properties of $\omega_k^\alpha$ and $\omega_k^\beta$ in \eqref{DecNP} and the orthogonality of $\mathcal{O}^\al$ and $\mathcal{O}^\beta$, we have
\begin{equation*}
\omega_k^\alpha\omega_k^\beta g_k=0
\end{equation*}
if $N_k^\alpha=N_k^\beta\to\infty$, for $k$ large enough. Otherwise $N^\alpha_k/N^\beta_k\to+\infty$ and we estimate
\begin{equation*}
\begin{split}
\Vert \omega^\alpha_k\omega^\beta_kg_k\Vert_{L^1_t(I,H^1)}
&\lesssim \Vert \omega^\alpha_k\omega^\beta_k\Vert_{L^1_tL^\infty_x}\Vert\,|\nabla^1 g_k|\,\Vert_{L^\infty_t(I,L^2)}+\Vert\nabla(\omega^\alpha_k\omega^\beta_k)\Vert_{L^\frac{6}{5}_tL^3_x}\Vert g_k\Vert_{L^6_{t,x}},
\end{split}
\end{equation*}
where $|\nabla^1 g_k |$ is defined in \eqref{grad}. Using \eqref{DecNP} we estimate, for $k$ large enough,
\begin{equation*}
\begin{split}
&\Vert \omega^\alpha_k\omega^\beta_k\Vert_{L^1_tL^\infty_x}\le\Vert \omega^\alpha_k\Vert_{L^1_tL^\infty_x}\Vert \omega^\beta_k\Vert_{L^\infty_{x,t}}\lesssim_{\theta} N^\beta_k/N^\alpha_k,\\
&\Vert (\nabla\omega^\alpha_k)\omega^\beta_k\Vert_{L^\frac{6}{5}_tL^3_x}\le \Vert \nabla\omega^\alpha_k\Vert_{L^\frac{6}{5}_tL^3_x}\Vert \omega^\beta_k\Vert_{L^\infty_{t,x}}\lesssim_{\theta} N^\beta_k(N^\alpha_k)^2\Vert \mathbf{1}_{\mathcal{S}^\alpha_k}\Vert_{L^\frac{6}{5}_tL^3_x}\lesssim_{\theta} N^\beta_k/N^\alpha_k,\\
&\Vert \omega^\alpha_k\nabla \omega^\beta_k\Vert_{L^\frac{6}{5}_tL^3_x}\le \Vert \omega^\alpha_k\Vert_{L^\frac{6}{5}_tL^3_x}\Vert \nabla \omega^\beta_k\Vert_{L^\infty_{t,x}}\lesssim_{\alpha,\beta,\theta} (N^\beta_k)^2N^\alpha_k\Vert \mathbf{1}_{\mathcal{S}^\alpha_k}\Vert_{L^\frac{6}{5}_tL^3_x}\lesssim_{\theta} (N^\beta_k/N^\alpha_k)^2.
\end{split}
\end{equation*}
Therefore
\begin{equation*}
\|\omega_k^\alpha\omega_k^\beta g_k\|_{N(I)}\leq\theta\qquad\text{ for }k\text{ large enough },
\end{equation*}
as desired.

We estimate now $\omega^\alpha_k\omega^{\beta,\pm\infty}_kg_k$ as before
\begin{equation}\label{ca2}
\Vert \omega^\alpha_k\omega^{\beta,\pm\infty}_kg_k\Vert_{L^1_t(I,H^1)}\lesssim \Vert \omega^\alpha_k\omega^{\beta,\pm\infty}_k\Vert_{L^1_tL^\infty_x}\Vert\,|\nabla^1 g_k|\,\Vert_{L^\infty_t(I,L^2)}+\Vert\nabla(\omega^\alpha_k\omega^{\beta,\pm\infty}_k)\Vert_{L^\frac{6}{5}_tL^3_x}\Vert g_k\Vert_{L^6_{t,x}}.
\end{equation}
If $\lim_{k\to\infty}N^\alpha_k/N^\beta_k=+\infty$ then we estimate
\begin{equation*}
\|\omega_k^\alpha\omega_k^{\beta,\pm\infty} g_k\|_{N(I)}\lesssim_\theta N^\be_k/N^\al_k\qquad\text{ for }k\text{ sufficiently large},
\end{equation*}
as before, since the bounds $\Vert \omega^\beta_k\Vert_{L^\infty_{x,t}}\lesssim_\theta N_k^\beta$ and $\Vert \nabla \omega^\beta_k\Vert_{L^\infty_{t,x}}\lesssim_\theta (N_k^\beta)^2$ still hold. Other\-wise, if $N^\al_k=N^\be_k\to\infty$ and $(N_k^\alpha)^2\vert t^\alpha_k-t^\beta_k\vert\to+\infty$, using \eqref{Stric3} and the description \eqref{DecNP} we see that, for $k$ large enough,
\begin{equation*}
\begin{split}
\Vert \omega^\alpha_k\omega^{\beta,\pm\infty}_k\Vert_{L^1_t(I,L^\infty_x)}&\le  \Vert \omega^\alpha_k\Vert_{L^1_t(I,L^\infty_x)}\Vert\omega^{\beta,\pm\infty}_k\Vert_{L^\infty_{x,t}(\{(N_k^\alpha)^2\vert t-t^\alpha_k\vert\le R_{\theta}\})}\lesssim_{\theta} (N_k^\alpha)^{-1} \vert t^\beta_k-t^\alpha_k\vert^{-\frac{1}{2}},
\end{split}
\end{equation*}
and
\begin{equation*}
\begin{split}
\Vert \nabla(\omega^\alpha_k\omega^{\beta,\pm\infty}_k)\Vert_{L^\frac{6}{5}_t(I,L^3_x)}&\le  \Vert \nabla \omega^\alpha_k\Vert_{L^\frac{6}{5}(I,L^3_x)}\Vert\omega^{\beta,\pm\infty}_k\Vert_{L^\infty(\{(N_k^\alpha)^2\vert t-t^\alpha_k\vert\le R_{\theta}\})}\\
&+\Vert \omega^\alpha_k\Vert_{L^\frac{6}{5}(I,L^3_x)}\Vert\nabla\omega^{\beta,\pm\infty}_k\Vert_{L^\infty(\{(N_k^\alpha)^2\vert t-t^\alpha_k\vert\le R_{\theta}\})}\\
&\lesssim_{\theta} (N_k^\alpha)^{-1}\vert t^\beta_k-t^\alpha_k\vert^{-\frac{1}{2}}.
\end{split}
\end{equation*}
Finally, if $N^\alpha_k\vert x^\alpha_k-x^\beta_k\vert\to+\infty$ and $t^\alpha_k=t^\beta_k, N^\al_k=N^\be_k$, then we use Lemma \ref{step1} (ii) with $\rho=0$. It follows easily that
\begin{equation*}
\Vert \omega^\alpha_k\omega^{\beta,\pm\infty}_kg_k\Vert_{N(I)}\to 0
\end{equation*}
as $k\to \infty$, as desired

Finally, it remains to prove that
\begin{equation}\label{Deco9}
\limsup_{k\to\infty}\|\omega^{\alpha,\pm\infty}_k\omega^\beta_kg_k\|_{N(I)}\lesssim\theta\qquad\text{ if }\qquad N^\alpha_k/N^\beta_k\to+\infty.
\end{equation}
We write
\begin{equation}\label{Deco9.5}
\begin{split}
\omega^{\alpha,\pm\infty}_k\omega^\beta_kg_k&=\sum_{A\geq 2^{10}}P_A\omega^{\alpha,\pm\infty}_k\cdot P_{\leq 2^{-10}A}\omega^\beta_k\cdot P_{\leq 2^{-10}A}g_k\\
&+\sum_{B\geq 1}P_{\leq 2^9B}\omega^{\alpha,\pm\infty}_k\cdot P_B\omega^\beta_k\cdot P_{\leq B}g_k\\
&+\sum_{C\geq 2}P_{\leq 2^9C}\omega^{\alpha,\pm\infty}_k\cdot P_{\leq C/2}\omega^\beta_k\cdot P_{C}g_k.
\end{split}
\end{equation}
It follows from \eqref{refin} that
\begin{equation*}
\Big\|\sum_{B\geq 2}P_{\leq 2^9B}\omega^{\alpha,\pm\infty}_k\cdot P_B\omega^\beta_k\cdot P_{\leq B}g_k\Big\|_{N(I)}\lesssim \|\omega^{\alpha,\pm\infty}_k\|_{Z'(I)}\|\omega^\beta_k\|_{X^1(I)}\|g_k\|_{Z'(I)}
\end{equation*}
and
\begin{equation*}
\Big\|\sum_{C\geq 2}P_{\leq 2^9C}\omega^{\alpha,\pm\infty}_k\cdot P_{\leq C/2}\omega^\beta_k\cdot P_{C}g_k\Big\|\lesssim \|\omega^{\alpha,\pm\infty}_k\|_{Z'(I)}\|\omega^\beta_k\|_{Z'(I)}\|g_k\|_{X^1(I)}.
\end{equation*}
Recall from \eqref{DecNP} that $\|\omega^{\alpha,\pm\infty}_k\|_{Z'(I)}\leq\theta$ and $\|\omega^\beta_k\|_{X^1(I)}+\|g_k\|_{X^1(I)}\lesssim 1$. In view of the decomposition \eqref{Deco9.5} and the last two bounds, it remains to prove that
\begin{equation}\label{Deco10}
\limsup_{k\to\infty}\Big\|\sum_{A\geq 2^{10}}P_A\omega^{\alpha,\pm\infty}_k\cdot P_{\leq 2^{-10}A}\omega^\beta_k\cdot P_{\leq 2^{-10}A}g_k\Big\|_{N(I)}\lesssim\theta\qquad\text{ if }\qquad N^\alpha_k/N^\beta_k\to+\infty.
\end{equation}

It follows from the description of $\omega_k^{\al,\pm\infty}$ in \eqref{DecNP} that there is $Q_\theta=Q_{\theta,\phi^{\al,\pm\infty}}\in 2^{\mathbb{Z}_+}$ sufficiently large and $\theta'=\theta'_{\theta,\phi^{\al,\pm\infty}}$ sufficiently small such that, for $k$ large enough,
\begin{equation*}
\begin{split}
\|P_{\leq Q_\theta^{-1} N_k^\al}(T_{N_k^\al}&\phi^{\al,\pm\infty})\|_{H^1(\R\times\T^3)}+\|(1-P_{\leq Q_\theta N_k^\al})(T_{N_k^\al}\phi^{\al,\pm\infty})\|_{H^1(\R\times\T^3)}\\
&+\|(1-\widetilde{P}_{\theta'}^1)(T_{N_k^\al}\phi^{\al,\pm\infty})\|_{H^1(\R\times\T^3)}\leq\theta,
\end{split}
\end{equation*}
where the operator $\widetilde{P}_{\theta'}^1$ is defined in \eqref{sobop}. Therefore
\begin{equation*}
\|P_{\leq Q_\theta^{-1} N_k^\al}\omega^{\al,\pm\infty}_k\|_{X^1(I)}+\|(1-P_{\leq Q_\theta N_k^\al})\omega^{\al,\pm\infty}_k\|_{X^1(I)}+\|(1-\widetilde{P}_{\theta'}^1)\omega_k^{\al,\pm\infty})\|_{X^1(I)}\lesssim\theta.
\end{equation*}
Using again \eqref{refin}, for \eqref{Deco10} it suffices to prove that if $N^\alpha_k/N^\beta_k\to+\infty$ then
\begin{equation}\label{Deco11}
\limsup_{k\to\infty}\sum_{Q_\theta^{-1}\leq A/N_k^\al\leq Q_\theta}\|P_A\widetilde{P}_{\theta'}^1\omega^{\alpha,\pm\infty}_k\cdot P_{\leq 2^{-10}A}\omega^\beta_k\cdot P_{\leq 2^{-10}A}g_k\|_{N(I)}\lesssim\theta.
\end{equation}

For $A\in [Q_\theta^{-1}N_k^\al,Q_\theta N_k^\al]$ and $k$ large enough we estimate, using \cite[Proposition 2.10]{HeTaTz2},
\begin{equation*}
\begin{split}
&\|P_A\widetilde{P}_{\theta'}^1\omega^{\alpha,\pm\infty}_k\cdot P_{\leq 2^{-10}A}\omega^\beta_k\cdot P_{\leq 2^{-10}A}g_k\|_{N(I)}\\
&\lesssim \sup_{\|f\|_{Y^{-1}(I)=1}}\Big|\int_{\R\times\T^3\times I}f\cdot P_A\widetilde{P}_{\theta'}^1\omega^{\alpha,\pm\infty}_k\cdot P_{\leq 2^{-10}A}\omega^\beta_k\cdot P_{\leq 2^{-10}A}g_k\,dxdt\Big|\\
&\lesssim \|P_A\widetilde{P}_{\theta'}^1\omega^{\alpha,\pm\infty}_k\cdot P_{\leq 2^{-10}A}\omega^\beta_k\|_{L^2(\R\times\T^3\times I)}\cdot\sup_{\|f\|_{Y^{-1}(I)=1}}\sup_{A'\sim A}\|P_{A'}f\cdot P_{\leq 2^{-10}A}g_k\|_{L^2(\R\times\T^3\times I)}.
\end{split}
\end{equation*}
Using \eqref{NLEst}
\begin{equation*}
\sup_{\|f\|_{Y^{-1}(I)=1}}\sup_{A'\sim A}\|P_{A'}f\cdot P_{\leq 2^{-10}A}g_k\|_{L^2(\R\times\T^3\times I)}\lesssim A.
\end{equation*}
Using \eqref{m15} and \eqref{DecNP}
\begin{equation*}
\begin{split}
\|P_A\widetilde{P}_{\theta'}^1\omega^{\alpha,\pm\infty}_k\cdot P_{\leq 2^{-10}A}\omega^\beta_k\|_{L^2(\R\times\T^3\times I)}&\lesssim \|P_A\widetilde{P}_{\theta'}^1\omega^{\alpha,\pm\infty}_k\|_{L^\infty_{x_1}L^2_{x',t}}\|\omega^\be_k\|_{L^2_{x_1}L^\infty_{x',t}}\\
&\lesssim_\theta (\theta'A)^{-1/2}A^{-1}(N_k^\beta)^{1/2}.
\end{split}
\end{equation*}
It follows from the last three estimates that
\begin{equation*}
\|P_A\widetilde{P}_{\theta'}^1\omega^{\alpha,\pm\infty}_k\cdot P_{\leq 2^{-10}A}\omega^\beta_k\cdot P_{\leq 2^{-10}A}g_k\|_{N(I)}\lesssim_\theta (N^k_\beta/N_k^\alpha)^{1/2}
\end{equation*}
and \eqref{Deco11} follows. This completes the proof of the lemma.

\subsection{Proof of Lemma \ref{lemm2}} Recall that we defined nonlinear profiles $U_k^\alpha$, $\al,k\in\{1,2,\ldots\}$ with $\|U_k^\alpha\|_{X^1(-2,2)}\lesssim 1$. We also fixed a number $A\lesssim 1$ and defined
\begin{equation*}
U^A_{prof,k}=\sum_{1\leq\al\leq A}U_k^\al.
\end{equation*}
Letting $F'(G)u=2G\overline{G}u+G^2\overline{u}$, for $J\geq A$ and $k\in\mathbb{Z}^\ast_+$ we define $g_k^{A,J}$ as the solution on $(-1,1)$ of the linear initial-value problem
\begin{equation}\label{pla1}
(i\partial_t+\Delta)g_k^{A,J}-F'(U^A_{prof,k})g_k^{A,J}=0,\qquad g_k^{A,J}(0)=R_k^J.
\end{equation}
We know that $g_k^{A,J}$ is well defined on $(-1,1)$ and
\begin{equation}\label{pla2}
\|g_k^{A,J}\|_{X^1(-1,1)}\lesssim 1\qquad\text{ uniformly in }J\text{ and }k.
\end{equation}
We would like to prove that
\begin{equation}\label{pla3}
\limsup_{J\to+\infty}\limsup_{k\to+\infty}\|g_k^{A,J}\|_{Z(-1,1)}=0.
\end{equation}
Recall that
\begin{equation}\label{pla4}
\begin{split}
&\|R_k^J\|_{H^1}\lesssim 1\qquad\text{ for any }J,k\geq 1,\\
&\limsup_{J\to+\infty}\limsup_{k\to+\infty}\,\sup_{N\geq 1}N^{-1}\|e^{it\Delta}P_NR_k^J\|_{L^\infty_{x,t}(\R\times\T^3\times(-1,1))}=0.
\end{split}
\end{equation}

We will make repeated use of the following simple claim, which is consequence of the stability theory on section \ref{localwp}:

{\bf{Claim:}} If $u,u',v,v',h$ satisfy
\begin{equation*}
\Vert u\Vert_{X^1(-1,1)}+\Vert u'\Vert_{X^1(-1,1)}+\Vert v\Vert_{X^1(-1,1)}+\Vert v'\Vert_{X^1(-1,1)}\le C_1,\qquad\Vert h\Vert_{N(-1,1)}<+\infty
\end{equation*}
then, for every $u_0\in H^1$, the solution of the linear initial-value problem
\begin{equation*}
\left(i\partial_t+\Delta\right)g-uvg-u'v'\overline{g}=h,\qquad h(0)=u_0
\end{equation*}
exists globally in $X^1(-1,1)$ and satisfies
\begin{equation}\label{LinearClaim}
\Vert h\Vert_{X^1(-1,1)}\lesssim_{C_1}\Vert u_0\Vert_{H^1}+\Vert h\Vert_{N(-1,1)}.
\end{equation}

We prove \eqref{pla3} in several steps.

(I) For every fixed $\theta>0$, define $h^{J,\theta}_k$ as the solution to the initial value problem
\begin{equation}\label{pla5}
\left(i\partial_t+\Delta\right)h_k^{J,\theta}-\sum_{1\le\alpha\le A}F^\prime(\omega^{\alpha,\theta}_k)h_k^{J,\theta}=0,\qquad h(0)=R^J_k
\end{equation}
where $\omega_k^{\alpha,\theta}$ are defined as in \eqref{deco1} or \eqref{DecNP}, with the convention, as in the proof of Lemma \ref{lemm1}, $\omega_k^{\alpha,\theta}=V_k^{\alpha,\theta}$ if $\mathcal{O}^\al$ is a Scale-1 frame. Then, we have that
\begin{equation}\label{pla6}
\Vert h^{J,\theta}_k\Vert_{X^1(-1,1)}\lesssim 1
\end{equation}
uniformly in $k,J,\theta$, which is a consequence of \eqref{LinearClaim}. We will prove that
\begin{equation}\label{Simpl}
\Vert g^{A,J}_k-h^{J,\theta}_k\Vert_{X^1}\lesssim\theta
\end{equation}
for all $k$ sufficiently large. In order to do this, we let $\delta=h^{J,\theta}_k-g^{A,J}_k$ and compute that
\begin{equation*}
\begin{split}
&\left(i\partial_t+\Delta\right)\delta= F^\prime(U^A_{prof})\delta-\big[F^\prime(U^A_{prof,k})-\sum_{1\le\alpha\le A}F^\prime(\omega^{\alpha,\theta}_k)\big]h^{J,\theta}_k\\
&=F^\prime(U^A_{prof})\delta-\big[F^\prime(U^A_{prof,k})-\sum_{1\le\alpha\le A}F^\prime(U^\al_k)\big]h^{J,\theta}_k+\big[\sum_{1\le\alpha\le A}F^\prime(\omega^{\alpha,\theta}_k)-F^\prime(U^\al_k)\big]h^{J,\theta}_k.
\end{split}
\end{equation*}
Consequently, using \eqref{LinearClaim},
\begin{equation*}
\begin{split}
\|\delta\|_{X^1(-1,1)}&\lesssim \Big\|\big[F^\prime(U^A_{prof,k})-\sum_{1\le\alpha\le A}F^\prime(U^\al_k)\big]h^{J,\theta}_k\Big\|_{N(-1,1)}\\
&+\Big\|\big[\sum_{1\le\alpha\le A}F^\prime(\omega^{\alpha,\theta}_k)-F^\prime(U^\al_k)\big]h^{J,\theta}_k\Big\|_{N(-1,1)}.
\end{split}
\end{equation*}
Using Lemma \ref{lemm1} and \eqref{pla6}
\begin{equation*}
\limsup_{k\to+\infty}\Big\|\big[F^\prime(U^A_{prof,k})-\sum_{1\le\alpha\le A}F^\prime(U^\al_k)\big]h^{J,\theta}_k\Big\|_{N(-1,1)}=0.
\end{equation*}
Using \eqref{pla6}, the bounds in the descriptions \eqref{deco1} and \eqref{DecNP}, and Lemma \ref{NLEst2}
\begin{equation*}
\Big\|\big[\sum_{1\le\alpha\le A}F^\prime(\omega^{\alpha,\theta}_k)-F^\prime(U^\al_k)\big]h^{J,\theta}_k\Big\|_{N(-1,1)}\lesssim\theta
\end{equation*}
for $k$ large enough. The bound \eqref{Simpl} follows. It remains to prove that, for any $\theta$ fixed,
\begin{equation}\label{pla8}
\limsup_{J\to+\infty}\limsup_{k\to+\infty}\|h_k^{J,\theta}\|_{Z(-1,1)}\lesssim\theta.
\end{equation}

(II) For $\kappa>0$ a small dyadic number, define $\widetilde{P}^1_\kappa$ as in Lemma \ref{locsmo} and let $\widetilde{P}_\kappa^2:=1-\widetilde{P}_\kappa^1$. For every $J,k$, we consider $\sigma^{J,\theta,\kappa}_k$ the solution of the linear initial-value problem
\begin{equation}\label{pla10}
\left(i\partial_t+\Delta\right)\sigma^{J,\theta,\kappa}_k-\sum_{1\le\alpha\le A}F^\prime(\omega^{\alpha,\theta}_k)\sigma^{J,\theta,\kappa}_k=0,\qquad \sigma^{J,\theta,\kappa}_k(0)=(\widetilde{P}^1_{\kappa}+\widetilde{P}_\kappa^2P_{\leq\kappa^{-1}})R^J_k.
\end{equation}
As before, the solution $\sigma^{J,\theta,\kappa}_k$ is well-defined on $(-1,1)$ and $\|\sigma^{J,\theta,\kappa}_k\|_{X^1(-1,1)}\lesssim 1$. Using \eqref{LinearClaim} we see that
\begin{equation*}
\Vert \sigma^{J,\theta,\kappa}_k-e^{it\Delta}(\widetilde{P}^1_{\kappa}+\widetilde{P}_\kappa^2P_{\leq\kappa^{-1}})R^J_k\Vert_{X^1(-1,1)}\lesssim \sum_{1\le\alpha\le A}\Vert F^\prime(\omega^{\alpha,\theta}_k)(e^{it\Delta}(\widetilde{P}^1_{\kappa}+\widetilde{P}_\kappa^2P_{\leq\kappa^{-1}})R^J_k)\Vert_{N(-1,1)}.
\end{equation*}
For $\al\in\{1,\ldots,A\}$ we estimate, using Lemma \ref{locsmo}, Lemma \ref{precsob}, and the bounds in \eqref{deco1} and \eqref{DecNP},
\begin{equation*}
\begin{split}
\Vert F^\prime(\omega^{\alpha,\theta}_k)(e^{it\Delta}\widetilde{P}^1_{\kappa}R^J_k)\Vert_{N(-1,1)}&\lesssim\|(\omega_k^{\al,\theta})^2\|_{L^2_tL^\infty_x}\|\mathbf{1}_{\mathcal{S}_k^{\alpha,\theta}}\cdot|\nabla^1(e^{it\Delta}\widetilde{P}^1_{\kappa}R^J_k)|\,\|_{L^2_{x,t}}\\
&+\|\omega_k^{\al,\theta}|\nabla\omega_k^{\al,\theta}|\,\|_{L^1_tL^4_x}\|e^{it\Delta}\widetilde{P}^1_{\kappa}R^J_k\|_{L^\infty_tL^4_x}\\
&\lesssim_\theta \kappa^{1/100}
\end{split}
\end{equation*}
provided that $J,k$ sufficiently large. Moreover
\begin{equation*}
\|F'(\omega_k^{\al,\theta})(e^{it\Delta}\widetilde{P}^2_{\kappa}P_{\leq\kappa^{-1}}R^J_k)\Vert_{N(-1,1)}\lesssim \|\,|\nabla^1(\omega_k^{\al,\theta})^2|\,\|_{L^1_tL^2_x}\|\,|\nabla^1(e^{it\Delta}\widetilde{P}^2_{\kappa}P_{\leq\kappa^{-1}}R^J_k)|\,\|_{L^{\infty}_{x,t}}\lesssim\kappa
\end{equation*}
provided that $J,k$ sufficiently large. Independently, as in \eqref{Rem} and recalling \eqref{pla4}
\begin{equation*}
\limsup_{J\to+\infty}\limsup_{k\to+\infty}\Vert e^{it\Delta}(\widetilde{P}^1_{\kappa}+\widetilde{P}_\kappa^2P_{\leq\kappa^{-1}})R^J_k\Vert_{Z(-1,1)}=0.
\end{equation*}
Consequently, we see that for every fixed $\kappa>0$,
\begin{equation}\label{Estimsigma}
\limsup_{J\to+\infty}\limsup_{k\to+\infty}\Vert \sigma^{J,\theta,\kappa}_k\Vert_{Z(-1,1)}\lesssim_{\theta}\kappa^{1/100}.
\end{equation}

(III) It remains to prove that the solution $f_k^{J,\theta,\kappa}\in X^1(-1,1)$ of the initial-value problem
\begin{equation}\label{pla20}
\left(i\partial_t+\Delta\right)f^{J,\theta,\kappa}_k-\sum_{1\le\alpha\le A}F^\prime(\omega^{\alpha,\theta}_k)f^{J,\theta,\kappa}_k=0,\qquad f^{J,\theta,\kappa}_k(0)=R^{J,\kappa}_k:=\widetilde{P}_\kappa^2(1-P_{\leq\kappa^{-1}})R^J_k,
\end{equation}
satisfies, for $\kappa\leq\kappa(\theta)$ sufficiently small and $J\geq A$
\begin{equation}\label{pla21}
\limsup_{k\to+\infty}\|f^{J,\theta,\kappa}_k\|_{Z(-1,1)}\lesssim\theta.
\end{equation}

For dyadic numbers $1\leq M\leq N$ we define
\begin{equation}\label{pla26}
\begin{split}
p_{N,M}(\xi)&:=\left[\eta^4(\xi/N)-\eta^4(2\xi/N)\right][\eta^1(\xi_1/(2M))-\eta^1(\xi_1/M)]\text{ if }M\geq 2,\\
p_{N,1}(\xi)&:=\left[\eta^4(\xi/N)-\eta^4(2\xi/N)\right]\eta^1(\xi_1/2)\\
\mathcal{F}(P_{N,M}f)(\xi)&:=p_{N,M}(\xi)(\mathcal{F}f)(\xi),\qquad P_{N,\leq a}:=\sum_{1\leq M\leq \min(N,a)}P_{N,M}.
\end{split}
\end{equation}
Given $\rho\in(0,1]$ and dyadic numbers $1\leq M\leq N$ let
\begin{equation}\label{pla20.5}
c_{\rho,N,M}:=\Big[\Big(\frac{M^2+\rho^2N^2}{N^2}\Big)+(1+\rho^2 N)^{-1}\Big]^{\delta_0/100},
\end{equation}
where $\delta_0\in(0,1/100]$ is the constant in \eqref{pla100}. For functions $f\in X^1(I)$, $I\subseteq (-1,1)$, and $g\in H^1(\R\times\T^3)$ let
\begin{equation}\label{pla23}
\begin{split}
&\|f\|_{\widetilde{X}^1_\rho(I)}:=\Big\|\sum_{N\geq 1}\sum_{1\leq M\leq N}c_{\rho,N,M}P_{N,M}f\Big\|_{X^1(I)}\\
&\|g\|_{\widetilde{H}^1_\rho(\R\times\T^3)}:=\Big\|\sum_{N\geq 1}\sum_{1\leq M\leq N}c_{\rho,N,M}P_{N,M}g\Big\|_{H^1(\R\times\T^3)}.
\end{split}
\end{equation}
Clearly, for any $\rho\in(0,1]$,
\begin{equation}\label{pla24}
\begin{split}
&\rho^{\delta_0/50}\|f\|_{X^1(I)}\lesssim\|f\|_{\widetilde{X}^1_\rho(I)}\lesssim \|f\|_{X^1(I)},\\
&\rho^{\delta_0/50}\|g\|_{H^1(\R\times\T^3)}\lesssim\|g\|_{\widetilde{H}^1_\rho(\R\times\T^3)}\lesssim \|g\|_{H^1(\R\times\T^3)}.
\end{split}
\end{equation}
On the other hand, for any $\rho\in(0,1]$
\begin{equation}\label{pla25}
\|f\|_{Z(I)}\lesssim \|f\|_{\widetilde{X}^1_\rho(I)}.
\end{equation}
Indeed, using the Strichartz estimate $\|P_{N,M}f\|_{L^{19/5}_{x,t}(I)}\lesssim N^{-11/19}\|P_{N,M}f\|_{U^3_\Delta(I,H^1)}$, see \eqref{UpEst}, and the simple bound $\|P_{N,M}f\|_{L^{\infty}_{x,t}(I)}\lesssim M^{1/2}N^{1/2}\|P_{N,M}f\|_{U^3_\Delta(I,H^1)}$, it follows that, for any dyadic numbers $1\leq M\leq N$ and any $f\in X^1(I)$, $I\subseteq (-1,1)$,
\begin{equation}\label{pla27}
\|P_{N,M}f\|_{L^4_{x,t}(I)}\lesssim (M/N)^{1/40}N^{-1/2}\|P_{N,M}f\|_{U^3_\Delta(I,H^1)}.
\end{equation}
Therefore
\begin{equation*}
\begin{split}
\|f&\|_{Z(I)}^4\lesssim \sum_{N\geq 1}N^2\|P_Nf\|_{L^4_{x,t}(I)}^4\lesssim \sum_{N\geq 1}N^2\big[\sum_{1\leq M\leq N}\|P_{N,M}f\|_{L^4_{x,t}(I)}\big]^4\\
&\lesssim \sum_{N\geq 1}\big[\sum_{1\leq M\leq N}(M/N)^{1/40}\|P_{N,M}f\|_{X^1(I)}\big]^4\lesssim\sum_{N\geq 1}\Big\|\sum_{1\leq M\leq N}(M/N)^{1/80}P_{N,M}f\Big\|_{X^1(I)}^4,
\end{split}
\end{equation*}
which proves \eqref{pla25}.

We will prove the following lemma:

\begin{lemma}\label{lemm3}
Assume $D\geq 1$ is a dyadic number, $\theta\in(0,1]$ and $f_k\in X^1(-1,1)$ is a solution of the linear equation
\begin{equation}\label{pla30}
\left(i\partial_t+\Delta\right)f_k-\sum_{1\le\alpha\le A}F^\prime((P_{\leq DN_k^\al}-P_{\leq D^{-1}N_k^\al})\omega^{\alpha,\theta}_k)f_k=0.
\end{equation}
Then, for $\rho\leq\rho_{\theta,D}$ sufficiently small, and $k\geq k_{\rho}$ sufficiently large
\begin{equation}\label{pla31}
\|f_k\|_{\widetilde{X}^1_\rho(-1,1)}\lesssim_{\theta,D}\|f_k(0)\|_{\widetilde{H}^1_\rho(\R\times\T^3)}.
\end{equation}
\end{lemma}

Assuming the lemma we can complete the proof of \eqref{pla21}. Indeed, the same argument as in step I, using the bound
\begin{equation*}
\|(1-P_{\leq R_\theta N_k^\gamma}+P_{\leq R_\theta^{-1} N_k^\gamma})\omega_k^{\gamma,\theta}\|_{X^1(I)}\le\theta,
\end{equation*}
shows that if $\widetilde{f}_k^{J,\theta,\kappa}\in X^1(-1,1)$ is the solution of the linear initial-value problem
\begin{equation*}
\left(i\partial_t+\Delta\right)\widetilde{f}^{J,\theta,\kappa}_k-\sum_{1\le\alpha\le A}F^\prime((P_{\leq R_\theta N_k^\al}-P_{\leq R_\theta^{-1}N_k^\al})\omega^{\alpha,\theta}_k)\widetilde{f}^{J,\theta,\kappa}_k=0,\qquad \widetilde{f}^{J,\theta,\kappa}_k(0)=R^{J,\kappa}_k,
\end{equation*}
then
\begin{equation*}
\|f^{J,\theta,\kappa}_k-\widetilde{f}^{J,\theta,\kappa}_k\|_{X^1(-1,1)}\lesssim\theta
\end{equation*}
for $k$ large enough. In view of Lemma \ref{lemm3}, setting $\kappa=\rho^{10}$
\begin{equation*}
\|\widetilde{f}^{J,\theta,\kappa}_k\|_{\widetilde{X}_\rho^1(-1,1)}\lesssim_\theta\|R_k^{J,\kappa}\|_{\widetilde{H}^1_\rho}\lesssim_\theta \kappa^{\delta'}
\end{equation*}
for $k$ large enough and some $\delta'>0$. The desired bound \eqref{pla21} follows from \eqref{pla25}.

Therefore it remains to prove Lemma \ref{lemm3}.

\begin{proof}[Proof of Lemma \ref{lemm3}] Given an interval $I\subseteq(-1,1)$, $0\in I$, let
\begin{equation*}
\|f\|_{\widetilde{N}_\rho(I)}=\Big\|\int_a^t e^{i(t-s)\Delta}f(s)\,ds\Big\|_{\widetilde{X}_\rho^1(I)}
\end{equation*}
The standard argument of dividing an interval into sufficiently many small intervals and the fungibility of the norm $Z$ shows that it suffices to prove that there is $\rho_0=\rho_0(\theta,D)$ sufficiently small such that if $I\subseteq(-1,1)$, $0\in I$, and
\begin{equation}\label{pla35}
\sum_{1\leq\al\leq A}\|(P_{\leq DN_k^\al}-P_{\leq D^{-1}N_k^\al})\omega^{\alpha,\theta}_k\|_{Z(I)}\leq\rho_0
\end{equation}
then
\begin{equation*}
\big\|F^\prime((P_{\leq DN_k^\al}-P_{\leq D^{-1}N_k^\al})\omega^{\alpha,\theta}_k)f\big\|_{\widetilde{N}_\rho(I)}\leq 1/(2A)\|f\|_{\widetilde{X}^1_\rho(I)}
\end{equation*}
for any $\rho\in(0,\rho_0]$, $\al\in\{1,\ldots,A\}$, $f\in X^1(I)$, and $k\geq k_\rho$ sufficiently large. Letting $\omega_{k,D}^{\alpha,\theta}:=(P_{\leq DN_k^\al}-P_{\leq D^{-1}N_k^\al})\omega^{\alpha,\theta}_k$ and recalling the definition $F'(\omega_{k,D}^{\alpha,\theta})f=2\omega_{k,D}^{\alpha,\theta}\overline{\omega_{k,D}^{\alpha,\theta}}f+(\omega_{k,D}^{\alpha,\theta})^2\overline{f}$, it suffices to prove that
\begin{equation}\label{pla36}
\big\|\omega_{k,D}^{\alpha,\theta}\overline{\omega_{k,D}^{\alpha,\theta}}f\big\|_{\widetilde{N}_\rho(I)}+\big\|(\omega_{k,D}^{\alpha,\theta})^2\overline{f}\big\|_{\widetilde{N}_\rho(I)}\leq 1/(10A)\|f\|_{\widetilde{X}^1_\rho(I)}
\end{equation}
for any $\rho\leq\rho_0$, $\al\in\{1,\ldots,A\}$, $f\in X^1(I)$, and $k\geq k_\rho$ sufficiently large.

We will need a refinement of the bound \eqref{pla100}, namely
\begin{equation}\label{pla40}
\begin{split}
\|P_{N_1,M_1}&u_1\cdot P_{N_2,\leq M_2}u_2\|_{L^2(\R\times\T^3\times I)}\\
&\lesssim N_2\Big[\Big(\frac{N_2}{N_1}+\frac{1}{N_2}\Big)\frac{M_2}{N_2}\frac{M_1}{N_2+M_1}\Big]^{\delta_0/2} \|P_{N_1,M_1}u_1\|_{Y^0(I)}\|P_{N_2,M_2}u_2\|_{Y^0(I)}
\end{split}
\end{equation}
for any dyadic numbers $N_1,N_2,M_1,M_2$, $1\leq M_1\leq N_1$, $1\leq M_2\leq N_2$, $N_2\leq N_1$. This follows from the bound \eqref{pla100} and the simpler bound
\begin{equation}\label{pla40.5}
\begin{split}
\|P_CP_{N_1,M_1}u_1\cdot &P_{N_2,\leq M_2}u_2\|_{L^2(\R\times\T^3\times I)}\leq \|P_CP_{N_1,M_1}u_1\|_{L^4}\|P_{N_2,\leq M_2}u_2\|_{L^4}\\
&\lesssim N_2^{1/2}(M_1/(N_2+M_1))^{1/40}\|P_CP_{N_1,M_1}u_1\|_{U^3_\Delta(I,L^2)}\cdot \|P_{N_2,\leq M_2}u_2\|_{L^4}
\end{split}
\end{equation}
for any cube $C$ of size $N_2$ centered at some $\xi_0\in\mathbb{Z}^4$, see \eqref{pla27}.

To prove \eqref{pla36} we analyze two cases.

{\bf{Case 1.}} $N_k^\al=1$ for all $k$. In this case $\omega_{k,D}^{\alpha,\theta}=P_{\leq DN_k^\al}\omega^{\alpha,\theta}_k$. In view of the definition \eqref{pla20.5}, $\|P_{\leq \rho^{-2}}f\|_{X^1(I)}\lesssim \|f\|_{\widetilde{X}^1(I)}$. Using Lemma \ref{NLEst2} and \eqref{pla35}
\begin{equation*}
\big\|\omega_{k,D}^{\alpha,\theta}\overline{\omega_{k,D}^{\alpha,\theta}}P_{\leq\alpha^{-2}}f\big\|_{N(I)}+\big\|(\omega_{k,D}^{\alpha,\theta})^2\overline{P_{\leq\alpha^{-2}}f}\big\|_{N(I)}\lesssim_{D,\theta}\rho_0^{1/2}\|P_{\leq \rho^{-2}}f\|_{X^1(I)}\lesssim_{D,\theta}\rho_0^{1/2}\|f\|_{\widetilde{X}_\rho^1(I)}.
\end{equation*}
To estimate the contribution of $P_{>\al^{-2}}f$, by examining Fourier supports we notice that
\begin{equation*}
P_{N,M}\big[\omega_{k,D}^{\alpha,\theta}\overline{\omega_{k,D}^{\alpha,\theta}}\cdot P_{>\alpha^{-2}}f\big]=P_{N,M}\big[\omega_{k,D}^{\alpha,\theta}\overline{\omega_{k,D}^{\alpha,\theta}}\cdot \big(\sum_{N'\sim_DN,\,M'\sim_DM}P_{>\alpha^{-2}}P_{N',M'}f\big)\big].
\end{equation*}
Therefore, by definition and Lemma \ref{NLEst2}
\begin{equation*}
\begin{split}
\|\omega_{k,D}^{\alpha,\theta}\overline{\omega_{k,D}^{\alpha,\theta}}&\cdot P_{>\alpha^{-2}}f\|_{\widetilde{N}_\rho(I)}^2\sim\sum_{N\geq 1}\sum_{1\leq M\leq N}c_{\rho,N,M}^2\|P_{N,M}\big[\omega_{k,D}^{\alpha,\theta}\overline{\omega_{k,D}^{\alpha,\theta}}\cdot P_{>\alpha^{-2}}f\big]\|_{N(I)}^2\\
&\lesssim_{D,\theta}\sum_{N\geq 1}\sum_{1\leq M\leq N}\sum_{N'\sim_DN,\,M'\sim_DM}c_{\rho,N,M}^2\|\omega_{k,D}^{\alpha,\theta}\overline{\omega_{k,D}^{\alpha,\theta}}\cdot P_{>\alpha^{-2}}P_{N',M'}f\|_{N(I)}^2\\
&\lesssim_{D,\theta}\sum_{N'\geq 1}\sum_{1\leq M'\leq N'}\rho_0c_{\rho,N',M'}^2\|P_{>\alpha^{-2}}P_{N',M'}f\|_{X^1(I)}^2\\
&\lesssim _{D,\theta}\rho_0\|P_{>\alpha^{-2}}f\|_{\widetilde{X}_\rho^1(I)}^2.
\end{split}
\end{equation*}
The bound on $\big\|(\omega_{k,D}^{\alpha,\theta})^2\overline{P_{\leq\alpha^{-2}}f}\big\|_{\widetilde{N}_\rho(I)}$ is similar, and \eqref{pla36} follows in this case.
\medskip

{\bf{Case 2.}} $\lim_{k\to+\infty}N_k^\al=+\infty$. By letting $k$ large enough, in this case we may always assume that $N_k^\al\geq\rho^{-2}$. We estimate first the contribution of $P_{>2^{10}DN_k^\alpha}f$. By examining Fourier supports, we notice that
\begin{equation*}
P_{N,M}\big[\omega_{k,D}^{\alpha,\theta}\overline{\omega_{k,D}^{\alpha,\theta}}\cdot P_{>2^{10}DN_k^\alpha}f\big]=P_{N,M}\big[\omega_{k,D}^{\alpha,\theta}\overline{\omega_{k,D}^{\alpha,\theta}}\cdot \big(\sum_{N'\sim N,M'\sim M}P_{>2^{10}DN_k^\al}P_{N',M'}f\big)\big]
\end{equation*}
if $M>2^5DN_k^\al$, and
\begin{equation}\label{pla59}
\begin{split}
P_{N,\leq 2^5DN_k^\al}\big[\omega_{k,D}^{\alpha,\theta}&\overline{\omega_{k,D}^{\alpha,\theta}}\cdot P_{>2^{10}DN_k^\alpha}f\big]\\
&=P_{N,\leq 2^5DN_k^\al}\big[\omega_{k,D}^{\alpha,\theta}\overline{\omega_{k,D}^{\alpha,\theta}}\cdot \big(\sum_{N'\sim N}P_{>2^{10}DN_k^\al}P_{N',\leq 2^{10}DN_k^\al}f\big)\big].
\end{split}
\end{equation}
Estimating as in Case 1, it follows that 
\begin{equation}\label{pla60}
\Big\|\sum_{N\geq 1}\sum_{M>2^5DN_k^\al}P_{N,M}\big[\omega_{k,D}^{\alpha,\theta}\overline{\omega_{k,D}^{\alpha,\theta}}\cdot P_{>2^{10}DN_k^\alpha}f\big]\Big\|_{\widetilde{N}_\rho(I)}\lesssim_{D,\theta}\rho_0^{1/2} \|f\|_{\widetilde{X}^1_\rho}.
\end{equation}

Using \cite[Proposition 2.10]{HeTaTz2} and the estimates \eqref{pla40} and \eqref{pla40.5}, for $N'\geq 2^5DN_k^\alpha$ and $f':=P_{>2^{10}DN_k^\al}f$
\begin{equation}\label{pla50}
\begin{split}
&\big\|\omega_{k,D}^{\alpha,\theta}\overline{\omega_{k,D}^{\alpha,\theta}}P_{N',M'}f'\big\|_{N(I)}\lesssim \sup_{\|u_0\|_{Y^{-1}(I)}=1}\Big|\int_{\R\times\T^3\times I}u_0\omega_{k,D}^{\alpha,\theta}\overline{\omega_{k,D}^{\alpha,\theta}}P_{N',M'}f'\,dxdt\Big|\\
&\lesssim N'\|\omega_{k,D}^{\alpha,\theta}\cdot P_{N',M'}f'\|_{L^2(\R\times\T^3\times I)}\sup_{\|u_0\|_{Y^0(I)}=1}\|(P_{\leq 8N'}-P_{\leq N'/8})u_0\cdot \omega_{k,D}^{\alpha,\theta}\|_{L^2(\R\times\T^3\times I)}\\
&\lesssim_{D,\theta} \rho_0\Big[\Big(\frac{N_k^\al}{N'}+\frac{1}{N_k^\al}\Big)\frac{M'}{N_k^\al+M'}\Big]^{\delta_0/2} \|P_{N',M'}f'\|_{X^1(I)}.
\end{split}
\end{equation}
Using \eqref{pla50} and the definition of $c_{\rho,N',M'}$ in \eqref{pla20.5}, in particular
\begin{equation*}
(M'/N_k^\alpha)^{\delta_0/2}\lesssim_D (M'/N_k^\al)^{\delta_0/4}\cdot c_{\rho,N',M'}\Big(\frac{N_k^\al}{N'}+\frac{1}{N_k^\al}\Big)^{-\delta_0/4}\quad\text{ if }1\leq M'\leq 2^{10}DN_k^\al<N',
\end{equation*}
it follows that
\begin{equation*}
\sum_{1\leq M'\leq 2^{10}DN_k^\al} \big\|\omega_{k,D}^{\alpha,\theta}\overline{\omega_{k,D}^{\alpha,\theta}}P_{N',M'}f'\big\|_{N(I)}\lesssim _{D,\theta}\rho_0\Big(\frac{N_k^\al}{N'}+ \frac{1}{N_k^\al}\Big)^{\delta_0/4}\|P_{N'}f'\|_{\widetilde{X}_\rho^1(I)}.
\end{equation*}
Therefore, using also \eqref{pla59},
\begin{equation*}
\Big\|\sum_{N\geq 1}P_{N,\leq 2^5DN_k^\al}\big[\omega_{k,D}^{\alpha,\theta}\overline{\omega_{k,D}^{\alpha,\theta}}\cdot P_{>2^{10}DN_k^\alpha}f\big]\Big\|_{N(I)}\lesssim_{D,\theta}\rho_0 \|f\|_{\widetilde{X}^1_\rho}.
\end{equation*}
Combining with \eqref{pla60}, we derive the desired bound
\begin{equation}\label{pla65}
\|\omega_{k,D}^{\alpha,\theta}\overline{\omega_{k,D}^{\alpha,\theta}}\cdot P_{>2^{10}DN_k^\alpha}f\|_{\widetilde{N}_\rho(I)}\lesssim_{D,\theta}\rho_0^{1/2} \|f\|_{\widetilde{X}^1_\rho}.
\end{equation}

To bound the contribution of $P_{\leq 2^{10}DN_k^\alpha}$, we use \cite[Proposition 2.10]{HeTaTz2}, the assumption \eqref{pla35}, and the description \eqref{DecNP}
\begin{equation*}
\begin{split}
&\big\|\omega_{k,D}^{\alpha,\theta}\overline{\omega_{k,D}^{\alpha,\theta}}P_{\leq 2^{10}DN_k^\al}f\big\|_{N(I)}\lesssim \sup_{\|u_0\|_{Y^{-1}(I)}=1}\Big|\int_{\R\times\T^3\times I}u_0\omega_{k,D}^{\alpha,\theta}\overline{\omega_{k,D}^{\alpha,\theta}}P_{\leq 2^{10}DN_k^\al}f\,dxdt\Big|\\
&\lesssim \|\omega_{k,D}^{\alpha,\theta}\cdot P_{\leq 2^{10}DN_k^\al}f\|_{L^2(\R\times\T^3\times I)}\sup_{\|u_0\|_{Y^{-1}(I)}=1}\|P_{\leq 2^{20}DN_k^\al}u_0\cdot \omega_{k,D}^{\alpha,\theta}\|_{L^2(\R\times\T^3\times I)}\\
&\lesssim_{D,\theta}\rho_0 (N_k^\al)^{-1/2}\|P_{\leq 2^{10}DN_k^\al}f\|_{L^4_tL^\infty_x}\lesssim_{D,\theta}\rho_0 \|f\|_{\widetilde{X}^1_\rho}.
\end{split}
\end{equation*}
The main bound \eqref{pla36} for the term $\omega_{k,D}^{\alpha,\theta}\overline{\omega_{k,D}^{\alpha,\theta}}f$ follows, by combining with \eqref{pla65}. The estimate on the other term is similar, which completes the proof of the lemma.
\end{proof}

\section{Proof of Proposition \ref{Stric2}}\label{lastsection}

The main ingredient in the proof of Proposition \ref{Stric2} is the following distributional inequality:

\begin{lemma}\label{Distr2}
Assume $p_0>18/5$, $N\geq 1$, $\lambda\in[N^{(2p_0-6)/(p_0-2)},2^{10}N^2]$, $\|m\|_{L^2(\R\times\mathbb{Z}^3)}\leq 1$, and $m(\xi)=0$ for $|\xi|>N$, then
\begin{equation}\label{nomain}
\Big|\{(x,t)\in\R\times\mathbb{T}^3\times[-2^{-10},2^{-10}]:\Big|\int_{\R\times\mathbb{Z}^3}m(\xi)e^{-it|\xi|^2}e^{ix\cdot\xi}\,d\xi\Big|\geq\lambda\}\Big|\lesssim N^{2p_0-6}\lambda^{-p_0}.
\end{equation}
\end{lemma}

It is not hard to see that Proposition \ref{Stric2} follows from Lemma \ref{Distr2}:

\begin{proof}[Proof of Proposition \ref{Stric2}] Letting
\begin{equation*}
F(x,t):=\int_{\R\times\mathbb{Z}^3}m(\xi)e^{-it|\xi|^2}e^{ix\cdot\xi}\,d\xi,
\end{equation*}
where $m$ is as in Lemma \ref{Distr2}, it suffices to prove that if $p>p_1$ and $N\geq 1$ then
\begin{equation}\label{bv1}
\|\mathbf{1}_{[-2^{-10},2^{-10}]}(t)\cdot F\|_{L^p(\R\times\mathbb{T}^3\times\mathbb{R})}\lesssim_p N^{2-6/p}.
\end{equation}
We may assume $p\in (p_1,4]$ and $N\gg 1$. Then
\begin{equation*}
\begin{split}
\|\mathbf{1}_{[-2^{-10},2^{-10}]}(t)&\cdot F\|_{L^p(\R\times\mathbb{T}^3\times\mathbb{R})}^p\\
&\leq\sum_{2^l\leq 2^{10}N^2}2^{pl}|\{(x,t)\in\R\times\mathbb{T}^3\times[-2^{-10},2^{-10}]:|F(x,t)|\geq 2^l\}|.
\end{split}
\end{equation*}
If $2^l\geq N^{(2p_0-6)/(p_0-2)}$, $p_0\in (18/5,p)$, we use the distributional inequality \eqref{nomain}. If $2^l\leq N^{(2p_0-6)/(p_0-2)}$ we use the simple bound
\begin{equation*}
2^{2l}|\{(x,t)\in\R\times\mathbb{T}^3\times[-2^{-10},2^{-10}]:|F(x,t)|\geq 2^l\}|\leq \|F\|^2_{L^2(\R\times\mathbb{T}^3\times\mathbb{R})}\lesssim 1.
\end{equation*}
Therefore
\begin{equation*}
\begin{split}
\|\mathbf{1}_{[-2^{-10},2^{-10}]}(t)&\cdot F\|_{L^p(\R\times\mathbb{T}^3\times\mathbb{R})}^p\\
&\lesssim \sum_{2^l\leq N^{(2p_0-6)/(p_0-2)}}2^{(p-2)l}+\sum_{N^{(2p_0-6)/(p_0-2)}\leq 2^l\leq 2^{10}N^2}2^{pl}\cdot N^{2p_0-6}2^{-p_0l}\\
&\lesssim_p N^{2p-6},
\end{split}
\end{equation*}
which gives \eqref{bv1}.
\end{proof}

We prove Lemma \ref{Distr2} by adapting the arguments of Bourgain \cite[Section 3]{Bo2}. The proof we present here is self-contained, with the exception of the bounds \eqref{bo3}, \eqref{divisors}, and \eqref{no2}-\eqref{no3}. Let $Z_q:=\{0,\ldots,q-1\}$. We start with a lemma:

\begin{lemma}\label{decomposition}
Assume $\gamma>0$, $Q,M\in\mathbb{Z}_+^\ast$, $M\geq 8Q$, $S\subseteq\{1,\ldots,Q\}$, and $\eta:\mathbb{R}\to[0,1]$ is a smooth function supported in $[-2,2]$. Then, for any $t\in\mathbb{R}$,
\begin{equation}\label{bo1}
\sum_{q\in S,\,a\in\mathbb{Z}\,(a,q)=1}\eta(MQ(t-a/q))=\sum_{m\in\mathbb{Z}}(MQ)^{-1}\widehat{\eta}(2\pi m/(MQ))c_m e^{2\pi imt},
\end{equation}
and the coefficients $c_m$ have the properties
\begin{equation}\label{bo2}
c_m=\sum_{q\in S,\,a\in Z_q,\,(a,q)=1}e^{-2\pi im\cdot a/q},\qquad\sup_{m\in\mathbb{Z}}|c_m|\leq 4Q^2,\qquad |c_m|\leq C_\gamma d(m,Q)Q^{1+\gamma},
\end{equation}
where $d(m,Q)$ denotes the number of divisors of $m$ less than or equal to $Q$.
\end{lemma}

\begin{proof}[Proof of Lemma \ref{decomposition}]
We have
\begin{equation*}
\begin{split}
&\int_0^1e^{-2\pi imt}\sum_{q\in S,\,a\in\mathbb{Z}\,(a,q)=1}\eta(MQ(t-a/q))\,dt\\
&=\sum_{q\in S,\,a\in Z_q,\,(a,q)=1}\int_{\mathbb{R}}e^{-2\pi imt}\eta(MQ(t-a/q))\,dt\\
&=(MQ)^{-1}\widehat{\eta}(2\pi m/(MQ))\sum_{q\in S,\,a\in Z_q,\,(a,q)=1}e^{-2\pi im\cdot a/q},
\end{split}
\end{equation*}
for any $m\in\mathbb{Z}$. The first inequality in \eqref{bo2} is clear. To prove the second inequality we may assume $Q\geq 100$ and it suffice to prove that for any $m\in\mathbb{Z}$
\begin{equation}\label{bo3}
\sum_{q=1}^Q\Big|\sum_{a\in Z_q,\,(a,q)=1}e^{-2\pi im\cdot a/q}\Big|\leq C_\gamma d(m,Q)Q^{1+\gamma}.
\end{equation}
This is proved in \cite[Lemma 3.33]{Bo2}.
\end{proof}

We will also need the estimate in \cite[Lemma 3.47]{Bo2}: for any $\gamma,B>0$ there is $C_{\gamma,B}\geq 1$ such that
\begin{equation}\label{divisors}
|\{m\in\{0,\ldots,P\}:d(m,Q)\geq D\}|\leq C_{\gamma,B}(D^{-B}Q^\gamma P+Q^B)\qquad\text{ for any }P,Q,D\in\mathbb{Z}_+^\ast.
\end{equation}

We turn now to the proof of Lemma \ref{Distr2}. The implicit constants may all depend on the value of $p_0$. We may assume that $N\gg 1$ and consider the kernel $K_{N}:\R\times\mathbb{T}^3\times\mathbb{R}\to\mathbb{C}$,
\begin{equation}\label{no1}
\begin{split}
&K_{N}(x,t)=\eta^1(2^5t/(2\pi))\int_{\R\times\mathbb{Z}^3}e^{-it|\xi|^2}e^{ix\cdot\xi}\eta^4(\xi/N)\,d\xi\\
&=\eta^1(2^5t/(2\pi))\prod_{j=2}^4\Big[\sum_{\xi_j\in\mathbb{Z}}e^{-it|\xi_j|^2}e^{ix_j\xi_j}\eta^1(\xi_j/N)^2\Big]\cdot \int_{\mathbb{R}}e^{-it|\xi_1|^2}e^{ix_1\cdot\xi_1}\eta^1(\xi_1/N)^2\,d\xi_1.
\end{split}
\end{equation}

Let
\begin{equation*}
S_\lambda:=\{(x,t)\in\R\times\mathbb{T}^3\times[-2^{-10},2^{-10}]:\Big|\int_{\R\times\mathbb{Z}^3}m(\xi)e^{-it|\xi|^2}e^{ix\cdot\xi}\,d\xi\Big|\geq\lambda\}
\end{equation*}
and fix a function $f:\R\times\mathbb{T}^3\times[-2^{-10},2^{-10}]\to\mathbb{C}$ such that
\begin{equation}\label{no20}
|f|\leq\mathbf{1}_{S_\lambda}
\end{equation}
and
\begin{equation*}
\lambda|S_\lambda|\leq\Big|\int_{\R\times\mathbb{T}^3\times\mathbb{R}}f(x,t)\cdot\Big[\int_{\R\times\mathbb{Z}^3}m(\xi)e^{-it|\xi|^2}e^{ix\cdot\xi}\,d\xi\Big]\,dxdt\Big|.
\end{equation*}
Using the assumptions on $m$ we estimate the right-hand side of the inequality above by
\begin{equation*}
\Big\|\prod_{j=1}^4\eta^1(\xi_j/N)\cdot\int_{\R\times\mathbb{T}^3\times\mathbb{R}}f(x,t)e^{-it|\xi|^2}e^{ix\cdot\xi}\,dxdt\Big\|_{L^2_\xi}.
\end{equation*}
Thus
\begin{equation}\label{no21}
\lambda^2|S_\lambda|^2\leq\int_{\R\times\mathbb{T}^3\times\mathbb{R}}\int_{\R\times\mathbb{T}^3\times\mathbb{R}} f(x,t)\overline{f(y,s)}K_N(t-s,x-y)\,dtdxdsdy.
\end{equation}
Using Lemma \ref{nodecomp} below, we estimate
\begin{equation*}
\begin{split}
&\int_{\R\times\mathbb{T}^3\times\mathbb{R}}\int_{\R\times\mathbb{T}^3\times\mathbb{R}} f(x,t)\overline{f(y,s)}K_N(t-s,x-y)\,dtdxdsdy\\
&\leq\sum_{\mu\in\{1,2,3\}}\Big|\int_{\R\times\mathbb{T}^3\times\mathbb{R}}\int_{\R\times\mathbb{T}^3\times\mathbb{R}} f(x,t)\overline{f(y,s)}K_N^{\mu,\lambda}(t-s,x-y)\,dtdxdsdy\Big|\\
&\leq (\lambda^2/2)\|f\|_{L^1}^2+C\lambda^2(N^{2p_0-6}\lambda^{-p_0})\|f\|_{L^2}^2+C\lambda^2(N^{2p_0-6}\lambda^{-p_0})^{(r-1)/r}\|f\|_{L^{2r/(r+1)}}^2\\
&\leq (\lambda^2/2)|S_\lambda|^2+C\lambda^2(N^{2p_0-6}\lambda^{-p_0})|S_\lambda|+C\lambda^2(N^{2p_0-6}\lambda^{-p_0})^{(r-1)/r}|S_\lambda|^{(r+1)/r}.
\end{split}
\end{equation*}
Using \eqref{no21} it follows that
\begin{equation*}
|S_\lambda|\lesssim N^{2p_0-6}\lambda^{-p_0}+(N^{2p_0-6}\lambda^{-p_0})^{(r-1)/r}|S_\lambda|^{1/r}.
\end{equation*}
which easily gives \eqref{nomain}.

Therefore it remains to prove the following lemma:

\begin{lemma}\label{nodecomp} Assuming $\lambda\in[N^{(2p_0-6)/(p_0-2)},2^{10}N^2]$ as in Lemma \ref{Distr2} and $r\in[2,4]$, there is a decomposition
\begin{equation*}
K_N=K_N^{1,\lambda}+K_N^{2,\lambda}+K_N^{3,\lambda}
\end{equation*}
such that
\begin{equation}\label{no40}
\begin{split}
&\|K_N^{1,\lambda}\|_{L^\infty(\R\times\mathbb{T}^3\times\mathbb{R})}\leq \lambda^2/2,\\
&\|\widehat{K_N^{2,\lambda}}\|_{L^\infty(\R\times\mathbb{Z}^3\times\mathbb{R})}\lesssim \lambda^2(N^{2p_0-6}\lambda^{-p_0}),\\
&\|\widehat{K_N^{3,\lambda}}\|_{L^r(\R\times\mathbb{Z}^3\times\mathbb{R})}\lesssim \lambda^2(N^{2p_0-6}\lambda^{-p_0})^{(r-1)/r}.
\end{split}
\end{equation}
\end{lemma}

\begin{proof}[Proof of Lemma \ref{nodecomp}] In view of the definition \eqref{no1}, for any continuous function $h:\mathbb{R}\to\mathbb{C}$ and any $(\xi,\tau)\in\R\times\mathbb{Z}^3\times\mathbb{R}$
\begin{equation}\label{no71}
\begin{split}
\mathcal{F}[K_N(x,t)\cdot h(t)](\xi,\tau)&=\int_{\R\times\mathbb{T}^3\times\mathbb{R}}e^{-ix\cdot\xi} e^{-it\tau}K_N(x,t)\cdot h(t)\,dxdt\\
&=C\eta^4(\xi/N)\int_{\mathbb{R}}h(t)\eta^1(2^5t/(2\pi))e^{-it(\tau+|\xi|^2)}\,dt.
\end{split}
\end{equation}

It is shown in \cite[Lemma 3.18]{Bo2} that
\begin{equation}\label{no2}
\Big|\sum_{n\in\mathbb{Z}}e^{-it|n|^2}e^{ixn}\eta^1(\xi/N)^2\Big|\lesssim \frac{N}{\sqrt{q}(1+N|t/(2\pi)-a/q|^{1/2})}
\end{equation}
if
\begin{equation}\label{no3}
t/(2\pi)=a/q+\beta,\qquad q\in\{1,\ldots,N\},\,\,a\in\mathbb{Z},\,\,(a,q)=1,\,\,|\beta|\leq (Nq)^{-1}.
\end{equation}
Therefore, for $t$ as in \eqref{no3},
\begin{equation}\label{no4}
|K_{N}(x,t)|\lesssim \frac{N^3}{q^{3/2}(1+N|t/(2\pi)-a/q|^{1/2})^3}\cdot\frac{N}{1+N|t/(2\pi)|^{1/2}}.
\end{equation}

For $j\in\mathbb{Z}$ we define $\eta_j,\eta\geq 0:\mathbb{R}\to[0,1]$,
\begin{equation*}
\eta_j(s):=\eta^1(2^js)-\eta^1(2^{j+1}s),\qquad \eta_{\geq j}(s):=\eta^1(2^js)=\sum_{k\geq j}\eta_k(s).
\end{equation*}
Clearly $\eta_j$ is supported in the set $\{s\in\mathbb{R}:|s|\in[2^{-j-1},2^{-j+1}]$ and $\eta_{\geq j}$ is supported in the interval $[-2^{-j+1},2^{-j+1}]$. Fix integers $K,L$,
\begin{equation}\label{no70}
\begin{split}
&K\in\mathbb{Z}_+,\qquad N\in [2^{K+4},2^{K+5}),\\
&L\in\mathbb{Z}\cap[0,2K+20],\qquad \lambda^{p_0-2}N^{6-2p_0}\in [2^{L},2^{L+1}).
\end{split}
\end{equation}

We start with the resolution
\begin{equation}\label{no5}
\begin{split}
&1=\big[\sum_{k=0}^{K-1}\sum_{j=0}^{K-k}p_{k,j}(s)\big]+e(s),\\
&p_{k,j}(s):=\sum_{q=2^{k}}^{2^{k+1}-1}\sum_{a\in\mathbb{Z},\,(a,q)=1}\eta_{j+K+k+10}(s/(2\pi)-a/q)\qquad\text{ if }j\leq K-k-1,\\
&p_{k,K-k}(s):=\sum_{q=2^{k}}^{2^{k+1}-1}\sum_{a\in\mathbb{Z},\,(a,q)=1}\eta_{\geq 2K+10}(s/(2\pi)-a/q).
\end{split}
\end{equation}
Let $T_K=\{(k,j)\in\{0,\ldots,K-1\}\times\{0,\ldots,K\}:k+j\leq K\}$. In view of Dirichlet's lemma, we observe that
\begin{equation}\label{no5.1}
\text{ if }t\in\mathrm{supp}(e)\text{ satisfies \eqref{no3} then either }N\lesssim q\text{ or }(Nq)^{-1}\approx |t/(2\pi)-a/q|.
\end{equation}

We define the first component of the kernel $K_N^{2,\lambda}$,
\begin{equation}\label{no72}
K_{N,1}^{2,\lambda}(x,t):=K_N(x,t)\cdot \eta^1(2^{L-40}t/(2\pi)).
\end{equation}
It follows from \eqref{no71} that
\begin{equation}\label{no72.5}
\|\widehat{K_{N,1}^{2,\lambda}}\|_{L^\infty(\R\times\mathbb{Z}^3\times\mathbb{R})}\lesssim 2^{-L}\lesssim N^{2p_0-6}\lambda^{2-p_0},
\end{equation}
which agrees with \eqref{no40}.

Therefore we may assume that $L\geq 45$ and write
\begin{equation}\label{no72.8}
K_N(x,t)-K_{N,1}^{2,\lambda}(x,t)=\sum_{l=4}^{L-41}K_N(x,t)\cdot\eta_l(t/(2\pi)).
\end{equation}
Using \eqref{no4} and \eqref{no5.1}, for any $(k,j)\in T_K$ and $l\in[4,L-41]\cap \mathbb{Z}$,
\begin{equation}\label{no73}
\begin{split}
&\sup_{x,t}|K_N(x,t)\cdot\eta_l(t/(2\pi))p_{k,j}(t)|\lesssim 2^{l/2}2^{(3K+3j)/2},\\
&\sup_{x,t}|K_N(x,t)\cdot\eta_l(t/(2\pi))e(t)|\lesssim 2^{l/2}2^{3K/2}.
\end{split}
\end{equation}
Notice that $2^{l/2}2^{3K/2}\lesssim \lambda^2$ (therefore the error term can be estimated in $L^\infty$ in the physical space, which is the key restriction to make the proof work) provided that
\begin{equation*}
p_0\geq 18/5.
\end{equation*}
This is the main reason for the choice in \eqref{no0}.

We analyze two cases.

{\bf{Case 1:}} $L\leq 2K-\delta K$, $\delta=1/100$. In this case we set
\begin{equation}\label{no74}
\begin{split}
&K_N^{1,\lambda}(x,t):=K_N(x,t)\cdot\Big[\sum_{l=4}^{L-41}\eta_l(t/(2\pi))\big[e(t)+\sum_{k,j\in T_K,\,2j\leq L}p_{k,j}(t)+\sum_{k,j\in T_K,\,2j> L}\rho_{k,j}p_{k,0}(t)\big]\Big],\\
&K_N^{2,\lambda}(x,t):=K_{N,1}^{2,\lambda}(x,t)+K_N(x,t)\cdot\Big[\sum_{l=4}^{L-41}\eta_l(t/(2\pi))\big[\sum_{k,j\in T_K,\,2j>L}(p_{k,j}(t)-\rho_{k,j}p_{k,0}(t))\big]\Big],\\
&K_{N}^{3,\lambda}(x,t):=0,
\end{split}
\end{equation}
where $K_{N,1}^{2,\lambda}$ is defined as in \eqref{no72} and
\begin{equation}\label{no74.5}
\rho_{k,j}:=2^{-j}\text{ if }j\leq K-k-1\text{ and }\rho_{k,K-k}:=2^{-K+k+1}.
\end{equation}
The coefficients $\rho_{k,j}$ are chosen to achieve a cancellation (see \eqref{no81} and \eqref{no82} below). In view of \eqref{no5} and \eqref{no72.8}, $K_N=K_N^{1,\lambda}+K_N^{2,\lambda}+K_N^{3,\lambda}$. It remains to prove the bounds in \eqref{no40}.

The bound on $K_{N}^{3,\lambda}$ is trivial. Using \eqref{no73}, for any $(x,t)\in\mathbb{T}^3\times\mathbb{R}\times\mathbb{R}$
\begin{equation*}
\begin{split}
|K_N^{1,\lambda}(x,t)|&\leq \sum_{l=4}^{L-41}|K_N(x,t)\cdot\eta_l(t/(2\pi))e(t)|\\
&+\sum_{l=4}^{L-41}\sum_{k,j\in T_K,\,2j\leq L}|K_N(x,t)\cdot\eta_l(t/(2\pi))p_{k,j}(t)|\\
&+\sum_{l=4}^{L-41}\sum_{k,j\in T_K,\,2j>L}|K_N(x,t)\cdot\eta_l(t/(2\pi))\rho_{k,j}p_{k,0}(t)|\\
&\lesssim 2^{L/2}2^{3K/2}+K2^{L/2}2^{3K/2+3L/4}+K2^{L/2}2^{3K/2}\\
&\lesssim K2^{(L-2K)(5p_0-18)/(4p_0-8)}\cdot 2^{2L/(p_0-2)}2^{K(4p_0-12)/(p_0-2)}.
\end{split}
\end{equation*}
Recall that $(2K-L)\geq \delta K$, $N\gg 1$, and $p_0>18/5$. Since $\lambda^2\approx 2^{2L/(p_0-2)}2^{K(4p_0-12)/(p_0-2)}$ (see \eqref{no70})  the desired bound $|K_N^{1,\lambda}(x,t)|\leq \lambda^2/2$ follows.

To prove the bound \eqref{no40} on the kernel $K_N^{2,\lambda}$ we use \eqref{no72.5} and the identity \eqref{no71}. Let $\widetilde{\eta}_L(s):=\sum_{l=4}^{L-41}\eta_l(s)=\eta^1(2^4s)-\eta^1(2^{L-40}s)$. It remains to prove that for any $b\in\mathbb{R}$
\begin{equation}\label{no80}
\sum_{k,j\in T_K,\,2j>L}\Big|\int_{\mathbb{R}}\widetilde{\eta}_L(t)(p_{k,j}(2\pi t)-\rho_{k,j}p_{k,0}(2\pi t))\eta^1(2^5t)e^{-itb}\,dt\Big|\lesssim 2^{-L}.
\end{equation}
Using Lemma \ref{decomposition} and the definition \eqref{no5}, we write, for $k,j\in T_K$,
\begin{equation}\label{no81}
\begin{split}
&p_{k,j}(2\pi t)-\rho_{k,j}p_{k,0}(2\pi t)\\
&=\sum_{m\in\mathbb{Z}}c_me^{2\pi imt}\big[2^{-j-K-k-10}\widehat{\eta_0}(2\pi m/2^{j+K+k+10})-\rho_{k,j}2^{-K-k-10}\widehat{\eta_0}(2\pi m/2^{K+k+10})\big].
\end{split}
\end{equation}
if $j\leq K-k-1$, and
\begin{equation}\label{no82}
\begin{split}
&p_{k,K-k}(2\pi t)-\rho_{k,K-k}p_{k,0}(2\pi t)\\
&=\sum_{m\in\mathbb{Z}}c_me^{2\pi imt}\big[2^{-2K-10}\widehat{\eta^1}(2\pi m/2^{2K+10})-\rho_{k,K-k}2^{-K-k-10}\widehat{\eta_0}(2\pi m/2^{K+k+10})\big],
\end{split}
\end{equation}
where
\begin{equation}\label{no82.5}
c_m=\sum_{q=2^k}^{2^{k+1}-1}\sum_{a\in Z_q,\,(a,q)=1}e^{-2\pi im\cdot a/q}.
\end{equation}
In view of the definition \eqref{no74.5} and the inequality \eqref{bo3}, for any $(k,j)\in T_K$
\begin{equation*}
\begin{split}
&p_{k,j}(2\pi t)-\rho_{k,j}p_{k,0}(2\pi t)=\sum_{m\in\mathbb{Z}}b_me^{2\pi imt},\\
&b_0=0,\qquad|b_m|\leq C_\gamma d(m,2^{k+1})2^{k(1+\gamma)}2^{-j-k-K}(1+|m|/2^{K+k})^{-10},\qquad\gamma>0.
\end{split}
\end{equation*}
Since $d(m,Q)\leq C_\gamma |m|^\gamma$, $\gamma>0$, for any $(m,Q)\in\mathbb{Z}^\ast\times\mathbb{Z}_+^\ast$, it follows that
\begin{equation*}
|b_m|\leq C_\gamma 2^{\gamma(j+K+k)}2^{-j-K}\qquad\text{ for any }m\in\mathbb{Z}.
\end{equation*}
Thus
\begin{equation*}
\begin{split}
\Big|\int_{\mathbb{R}}\widetilde{\eta}_L(t)(p_{k,j}(2\pi t)-\rho_{k,j}p_{k,0}(2\pi t))\eta^1(2^5t)e^{-itb}\,dt\Big|&\leq\sum_{m\in\mathbb{Z}}|b_m|\Big|\int_{\mathbb{R}}\widetilde{\eta}_L(t)\eta^1(2^5t)e^{-it(b-2\pi m)}\,dt\Big|\\
&\leq C_\gamma 2^{\gamma(j+K+k)}2^{-j-K}.
\end{split}
\end{equation*}
The desired bound \eqref{no80} follows if $\gamma$ is chosen sufficiently small, for example $\gamma=\delta/100$.
\medskip

{\bf{Case 2:}} $L\geq 2K-\delta K$, $\delta=1/100$. For $b\in\mathbb{Z}_+$ sufficiently large, we set
\begin{equation}\label{no90}
K_N^{1,\lambda}(x,t):=K_N(x,t)\cdot\Big[\sum_{l=4}^{L-41}\eta_l(t/(2\pi))\big[e(t)+\sum_{k,j\in T_K,\,2j\leq L-b}p_{k,j}(t)\big]\Big].
\end{equation}
Using the bounds \eqref{no73}, for any $(x,t)\in\R\times\mathbb{T}^3\times\mathbb{R}$
\begin{equation*}
\begin{split}
&|K_N^{1,\lambda}(x,t)|\\
&\lesssim \sum_{l=4}^{L-41}|K_N(x,t)\cdot\eta_l(t/(2\pi))e(t)|+\sum_{l=4}^{L-41}\sup_{k,j\in T_K,\,2j\leq L-b}|K_N(x,t)\cdot\eta_l(t/(2\pi))p_{k,j}(t)|\\
&\lesssim 2^{L/2}2^{3K/2}+2^{L/2}2^{3K/2+3L/4}2^{-3b/4}\\
&\lesssim 2^{-3b/4}2^{(L-2K)(5p_0-18)/(4p_0-8)}\cdot 2^{2L/(p_0-2)}2^{K(4p_0-12)/(p_0-2)}.
\end{split}
\end{equation*}
Since $\lambda^2\approx 2^{2L/(p_0-2)}2^{K(4p_0-12)/(p_0-2)}$, it follows that $|K_N^{1,\lambda}(x,t)|\leq\lambda^2/2$ provided that $b$ is fixed sufficiently large.

Let (see \eqref{no5} and \eqref{no72.8})
\begin{equation*}
L_N(x,t):=K_N(x,t)-K_{N,1}^{2,\lambda}(x,t)-K_N^{1,\lambda}(x,t)=\sum_{l=4}^{L-41}\sum_{k,j\in T_K,\,2j> L-b}K_N(x,t)\cdot\eta_l(t)p_{k,j}(t).
\end{equation*}
For \eqref{no40} it remains to prove that one can decompose
\begin{equation}\label{no91}
\begin{split}
&L_N=L_N^{2,\lambda}+L_N^{3,\lambda},\\
&\|\widehat{L_N^{2,\lambda}}\|_{L^\infty(\R\times\mathbb{Z}^3\times\mathbb{R})}\lesssim 2^{-L},\qquad \|\widehat{L_N^{3,\lambda}}\|_{L^r(\R\times\mathbb{Z}^3\times\mathbb{R})}\lesssim \lambda^{2/r}2^{-L(r-1)/r}.
\end{split}
\end{equation}

As before, let $\widetilde{\eta}_L(s):=\sum_{l=4}^{L-41}\eta_l(s)=\eta^1(2^4s)-\eta^1(2^{L-40}s)$. In view of the formula \eqref{no71}
\begin{equation}\label{no92}
\widehat{L_N}(\xi,\tau)=C\sum_{k,j\in T_K,\,2j> L-b}\eta^4(\xi/N)\int_{\mathbb{R}}\widetilde{\eta}_L(t)\eta^1(2^5t)p_{k,j}(2\pi t)e^{-2\pi it(\tau+|\xi|^2)}\,dt.
\end{equation}
The cardinality of the set $T_{K,L}:=\{k,j\in T_K:2j> L-b\}$ is bounded by $C(1+|2K-L|)^2$. Letting $f_{k,j}:\mathbb{R}\to\mathbb{C}$,
\begin{equation}\label{no93}
f_{k,j}(\mu):=\int_{\mathbb{R}}\widetilde{\eta}_L(t)\eta^1(2^5t)p_{k,j}(2\pi t)e^{-2\pi it\mu}\,dt,
\end{equation}
it suffices to prove that for any $(k,j)\in T_{K,L}$ one can decompose
\begin{equation}\label{no94}
\begin{split}
&f_{k,j}=f_{k,j}^2+f_{k,j}^3,\\
&\|f_{k,j}^2\|_{L^\infty(\mathbb{R})}\lesssim 2^{-L}(1+|2K-L|)^{-2},\\
&\|f_{k,j}^3\|_{L^r(\mathbb{R})}\lesssim (\lambda/N^2)^{2/r}2^{-L(r-1)/r}(1+|2K-L|)^{-2}.
\end{split}
\end{equation}

As in \eqref{no80} and \eqref{no81}, using Lemma \ref{decomposition}, we write
\begin{equation*}
p_{k,j}(2\pi t)=\sum_{m\in\mathbb{Z}}c_me^{2\pi imt}2^{-j-K-k-10}\chi(2\pi m/2^{j+K+k+10}),
\end{equation*}
where $c_m$ are as in \eqref{no82.5} and $\chi\in\{\widehat{\eta_0},\widehat{\eta^1}\}$. Letting $g_L:\mathbb{R}\to\mathbb{C}$,
\begin{equation}\label{no95}
g_L(\nu):=\int_{\mathbb{R}}\widetilde{\eta}_L(t)\eta^1(2^5t)e^{-2\pi it\nu}\,dt,
\end{equation}
it follows from \eqref{no93} that
\begin{equation}\label{no96}
f_{k,j}(\mu)=2^{-j-K-k-10}\sum_{m\in\mathbb{Z}}g_L(\mu-m)\cdot c_m\chi(2\pi m/2^{j+K+k+10}).
\end{equation}

We define
\begin{equation}\label{no97}
\begin{split}
&U_{N,k,j}=\{m\in\mathbb{Z}:|m|\leq 2^{j+K+2k}\text{ and }d(m,2^{k+1})\geq 2^{(2K-L)/4}\},\\
&f_{k,j}^2(\mu)=2^{-j-K-k-10}\sum_{m\notin U_{N,k,j}}g_L(\mu-m)\cdot c_m\chi(2\pi m/2^{j+K+k+10}),\\
&f_{k,j}^3(\mu)=2^{-j-K-k-10}\sum_{m\in U_{N,k,j}}g_L(\mu-m)\cdot c_m\chi(2\pi m/2^{j+K+k+10}).
\end{split}
\end{equation}
Clearly $f_{k,j}=f_{k,j}^2+f_{k,j}^3$. Using the inequalities in \eqref{bo2}
\begin{equation*}
\text{ if }m\notin U_{N,k,j}\text{ then }|c_m\chi(2\pi m/2^{j+K+k+10})|\lesssim 2^{(2K-L)/4}2^{11k/10}.
\end{equation*}
In view of the definition \eqref{no95},
\begin{equation}\label{no98}
\int_\mathbb{R} \sup_{|y|\leq 1}|g_L(\nu+y)|\,d\nu\lesssim 1.
\end{equation}
Therefore, if $(k,j)\in T_{K,L}$ (in particular $j\geq L/2-b/2$ and $k\leq K-L/2+b/2$)
\begin{equation*}
\|f_{k,j}^2\|_{L^\infty(\mathbb{R})}\lesssim 2^{-j-K}2^{(2K-L)3/8}\lesssim 2^{-L/2-K}2^{(2K-L)3/8},
\end{equation*}
which suffices to prove the first inequality in \eqref{no94}.

Since $|c_m|\lesssim 2^{2k}$,
\begin{equation*}
|f_{k,j}^3(\mu)|\lesssim 2^{-j-K+k}\sum_{m\in\mathbb{Z}}|g_L(\mu-m)|\mathbf{1}_{U_{N,k,j}}(m).
\end{equation*}
Using \eqref{no98} it follows that
\begin{equation*}
\|f_{k,j}^3\|_{L^r(\mathbb{R})}\lesssim 2^{-j-K+k}|U_{N,k,j}|^{1/r}.
\end{equation*}
Therefore, the last inequality in \eqref{no94} is a consequence of the bound
\begin{equation}\label{no99}
|U_{N,k,j}|\lesssim 2^{2K}2^{-10(2K-L)}.
\end{equation}
On the other hand, using the bound \eqref{divisors} with
\begin{equation*}
P=C2^{2K+(2K-L)},\quad Q=2^{k+1},\quad D=2^{(2K-L)/4},\quad\gamma=1,\quad B=100,
\end{equation*}
it follows that
\begin{equation*}
|U_{N,k,j}|\lesssim 2^{2K}2^{-20(2K-L)}+2^{50(2K-L)}.
\end{equation*}
The bound \eqref{no99} follows since $2K-L\leq \delta K$, $\delta=1/100$.
\end{proof}

\end{document}